\documentclass[12pt]{article}
\usepackage{amssymb}
\usepackage{amsmath}
\usepackage{amsthm}
\usepackage{color}
\usepackage{comment}
\usepackage{enumitem}
\usepackage{graphicx}
\usepackage{mathtools}		% for :=
\usepackage{multirow}
\usepackage{xcolor}									% <-- for \colorbox
\setitemize{itemsep=0pt}

\begin{document}

\def\cF{\mathcal{F}}
\def\cG{\mathcal{G}}
\def\cH{\mathcal{H}}
\def\cO{\mathcal{O}}
\def\cS{\mathcal{S}}
\def\cX{\mathcal{X}}
\def\cY{\mathcal{Y}}
\def\defeq{\vcentcolon=}
\def\id{\rm id}

\newcommand{\myStep}[1]{\noindent {\bf Step #1}.}

\newcommand{\removableFootnote}[1]{}

\newtheorem{theorem}{Theorem}[section]
\newtheorem{lemma}[theorem]{Lemma}
\newtheorem{proposition}[theorem]{Proposition}

\theoremstyle{definition}
\newtheorem{definition}{Definition}[section]

\title{
The structure of mode-locking regions of piecewise-linear continuous maps: II. Skew sawtooth maps.
}
\author{
D.J.W.~Simpson\\\\
Institute of Fundamental Sciences\\
Massey University\\
Palmerston North\\
New Zealand
}
\maketitle

\begin{abstract}

In two-parameter bifurcation diagrams of piecewise-linear continuous maps on $\mathbb{R}^N$, mode-locking regions typically have points of zero width known as shrinking points. Near any shrinking point, but outside the associated mode-locking region, a significant proportion of parameter space can be usefully partitioned into a two-dimensional array of nearly-hyperbolic annular sectors. The purpose of this paper is to show that in these sectors the dynamics is well-approximated by a three-parameter family of skew sawtooth circle maps, where the relationship between the skew sawtooth maps and the $N$-dimensional map is fixed within each sector. The skew sawtooth maps are continuous, degree-one, and piecewise-linear, with two different slopes. They approximate the stable dynamics of the $N$-dimensional map with an error that goes to zero with the distance from the shrinking point. The results explain the complicated radial pattern of periodic, quasi-periodic, and chaotic dynamics that occurs near shrinking points.

\end{abstract}

% keywords: border-collision bifurcation, nonsmooth, piecewise-smooth, mode-locking,
% MSC classes:
%		37G15: Bifurcations of limit cycles and periodic orbits
%		37G35: Attractors and their bifurcations
%		39A23: Periodic solutions
%		[also 39A28: Bifurcation theory]
%		[also 37E45: Rotation numbers and vectors]

%=====================================================================
\section{Introduction}
\label{sec:intro}
\setcounter{equation}{0}

%This paper concerns piecewise-linear maps on $\mathbb{R}^N$ that are continuous and involve two pieces.
This paper extends the study of \cite{Si17c} for
piecewise-linear maps on $\mathbb{R}^N$ ($N \ge 2$) that are continuous and involve two pieces.
For such maps, coordinates $x \in \mathbb{R}^N$ can be chosen such that the
hyperplane along which the two pieces are joined, termed the switching manifold,
is simply where the first coordinate of $x$ vanishes.
Using $s$ to denote the first coordinate of $x$, i.e.
\begin{equation}
s \defeq e_1^{\sf T} x \;,
\label{eq:s}
\end{equation}
such a map takes the form
\begin{equation}
x_{i+1} = f(x_i;\xi) \defeq
\begin{cases}
A_L(\xi) x_i + B(\xi), & s_i \le 0 \;, \\
A_R(\xi) x_i + B(\xi), & s_i \ge 0 \;,
\end{cases}
\label{eq:f}
\end{equation}
where $A_L$ and $A_R$ are real-valued $N \times N$ matrices and $B \in \mathbb{R}^N$.
It is assumed that $A_L$, $A_R$, and $B$ have a $C^K$ ($K \ge 2$) dependency on a parameter $\xi \in \mathbb{R}^M$ ($M \ge 2$).
The assumption that $f$ is continuous on $s=0$
implies $A_R = A_L + C e_1^{\sf T}$
%implies that $A_L$ and $A_R$ differ in only their first columns, i.e.
%\begin{equation}
%A_R = A_L + C e_1^{\sf T} \;,
%\label{eq:continuityCondition}
%\end{equation}
for some $C \in \mathbb{R}^N$.

Maps of the form (\ref{eq:f}) describe the dynamics
near border-collision bifurcations of piecewise-smooth maps \cite{DiBu08,Si16}.
Piecewise-smooth maps arise as return maps for
nonsmooth differential equations which serve as useful mathematical models of phenomena
with events that are discontinuous or at least fast relative to the usual motion of the system.
Examples of such systems exhibiting border-collision bifurcations include
neurons with square wave forcing \cite{Ti02},
mechanical oscillators with friction \cite{DiKo03,KoPi08},
and DC/DC power converters \cite{DiBu98,ZhMo03}.
Maps of the form (\ref{eq:f}) also arise as models of discrete-time phenomena involving a switching element,
such as trade cycle models with non-negativity constraints \cite{PuSu06}.

%The dynamics of (\ref{eq:f}) can be simple or extraordinarily complicated;
%see \cite{Si16} for a recent review.
%In this paper we address mode-locking regions of (\ref{eq:f}).
To illustrate the ideas of this paper, consider
\begin{equation}
A_L = \begin{bmatrix}
\tau_L & 1 & 0 \\
-\sigma_L & 0 & 1 \\
\delta_L & 0 & 0
\end{bmatrix}, \qquad
A_R = \begin{bmatrix}
\tau_R & 1 & 0 \\
-\sigma_R & 0 & 1 \\
\delta_R & 0 & 0
\end{bmatrix}, \qquad
B = \begin{bmatrix}
\mu \\ 0 \\ 0
\end{bmatrix},
\label{eq:ALAREx20}
\end{equation}
where $\xi = (\tau_L,\sigma_L,\delta_L,\tau_R,\sigma_R,\delta_R,\mu) \in \mathbb{R}^7$.
The map (\ref{eq:f}) with (\ref{eq:ALAREx20}) is the
border-collision normal form in three dimensions \cite{Di03}. %\cite{DiBu08,Si16}.
Fig.~\ref{fig:modeLockEx20} shows mode-locking regions of (\ref{eq:f}) with (\ref{eq:ALAREx20}) and
\begin{equation}
\tau_L = 0 \;, \qquad
\sigma_L = -1 \;, \qquad
\sigma_R = 0 \;, \qquad
\delta_R = 2 \;, \qquad
\mu = 1 \;.
\label{eq:paramModeLockEx20}
\end{equation}
Each coloured region is a mode-locking region where (\ref{eq:f})
has an attracting periodic solution of a fixed rotation number.
The regions are roughly ordered by rotation number and are more narrow for higher periods.

Unlike mode-locking regions of smooth maps, % \cite{Ar88,Ku04,MePa93,Se10}.
the mode-locking regions in Fig.~\ref{fig:modeLockEx20} have points of zero width.
These are termed shrinking points and have been described in a wide variety of mathematical models
that can be put in the form (\ref{eq:f}),
or are at least well-approximated by a map of form (\ref{eq:f}),
see for instance \cite{Ti02,SuGa04,LaMo06,ZhMo08b,SzOs09}.
In a mode-locking region, as we cross a shrinking point the number of points that the
corresponding periodic solution has on each side of the switching manifold changes by one.
In a neighbourhood of a shrinking point, the mode-locking region is bounded by four curves
along which an $\cS$-cycle (periodic solution with symbolic itinerary $\cS \in \{ L,R \}^{\mathbb{Z}}$)
has one point on the switching manifold for a particular sequence $\cS$ \cite{SiMe09},
and the shrinking point is referred to as an $\cS$-shrinking point.
%In the neighbourhood of a shrinking point, the mode-locking region is bounded by four smooth curves of border-collision bifurcations.
%where the corresponding periodic solution has one point on the switching manifold \cite{SiMe09}.
%These border-collision bifurcations admit a simple characterisation in terms of
%the symbolic itinerary of the associated periodic solution \cite{SiMe09}.

%%%%%%%%%%%%%%%%%%%%%%%%%%%%%%%%%%%%%%%%%%%%%%%%%%%%%%%%%%%%%
\begin{figure}[t!]
\begin{center}
\setlength{\unitlength}{1cm}
\begin{picture}(15,5.2)
\put(0,0){\includegraphics[height=5cm]{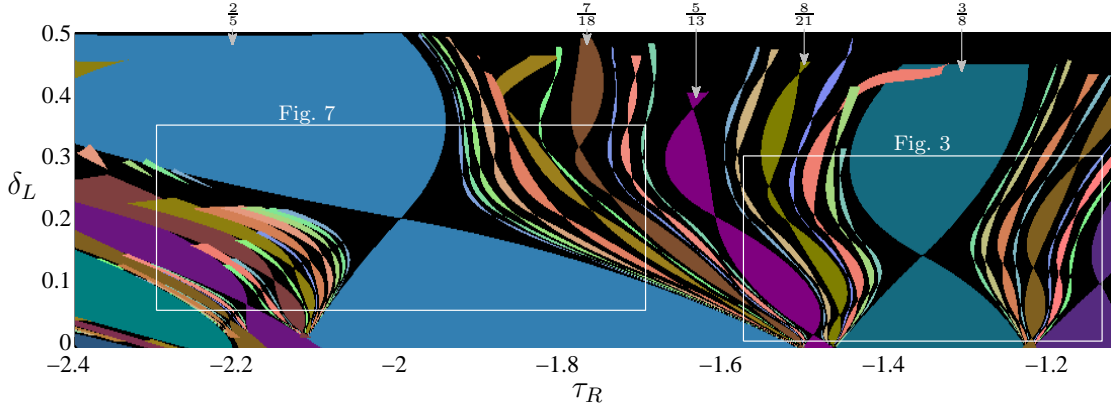}}
\put(7.5,0){$\tau_R$}
\put(0,2.7){$\delta_L$}
\put(2.9,5.04){\tiny $\frac{2}{5}$}
\put(7.54,5.04){\tiny $\frac{7}{18}$}
\put(8.99,5.04){\tiny $\frac{5}{13}$}
\put(10.43,5.04){\tiny $\frac{8}{21}$}
\put(12.59,5.04){\tiny $\frac{3}{8}$}
\put(3.6,3.69){\scriptsize \color{white} Fig.~\ref{fig:modeLockApprox20}}
\put(11.78,3.28){\scriptsize \color{white} Fig.~\ref{fig:modeLockApprox50}}
\end{picture}
\caption{
Mode-locking regions of (\ref{eq:f})--(\ref{eq:paramModeLockEx20}).
Only mode-locking regions corresponding to rotational periodic solutions
(defined in \S\ref{sub:shrPoints}) are shown as other mode-locking regions are relatively small,
do not exhibit shrinking points, and require more computation time to obtain to the same degree of accuracy.
This figure was computed by numerically checking the existence and stability of
rotational periodic solutions up to period $50$ on a $1024 \times 256$ grid of $\tau_R$ and $\delta_L$ values.
At grid points where mode-locking regions overlap, the solution of the highest period is indicated.
Five of the mode-locking regions are labelled by their rotation number.
\label{fig:modeLockEx20}
}
\end{center}
\end{figure}
%%%%%%%%%%%%%%%%%%%%%%%%%%%%%%%%%%%%%%%%%%%%%%%%%%%%%%%%%%%%%

Now consider the dynamics near an $\cS$-shrinking point but outside its corresponding mode-locking region.
Let $\frac{m}{n}$ be the rotation number associated with $\cS$.
As explained in \cite{Si17c}, 
the most dominant nearby mode-locking regions have rotation numbers
$\frac{m_k^-}{n_k^-} = \frac{k m + m^-}{k n + n^-}$ and
$\frac{m_k^+}{n_k^+} = \frac{k m + m^+}{k n + n^+}$,
where $k \in \mathbb{Z}^+$ and
$\frac{m^-}{n^-}$ and $\frac{m^+}{n^+}$ are the left and right Farey roots of $\frac{m}{n}$.
These regions, termed $\cG^\pm_k$-mode-locking regions,
form two sequences that approach the $\cS$-shrinking point from opposite sides as $k \to \infty$\removableFootnote{
Conceivably, one of these sequences of mode-locking regions could be virtual,
but I think this would be a special case (i.e.~higher codimension).
}.
The existence and location of shrinking points on the $\cG^\pm_k$-mode-locking regions is determined by
various scalar quantities associated with the $\cS$-shrinking point.

%The results of \cite{Si17c} were obtained
%by performing calculations on one-dimensional centre manifolds of (\ref{eq:f}).
%The purpose of this paper is to construct and analyse one-dimensional maps
%that describe the return dynamics near fundamental domains of such manifolds.
%In this paper we identify segments of these centre manifolds
%(bounded by a point $x$ and it's $n^{\rm th}$ iterate $f^n(x)$)
%such that each point in the segment returns approximately to the segment for which ...
%The centre manifolds are approximately recurrent,
%so by looking at appropriate powers of the one-dimensional restriction of (\ref{eq:f}) to a centre manifold ...
%approximate derive and study the one-dimensional restriction of (\ref{eq:f}) to such a centre manifold.
%The purpose of this paper is to derive and study the one-dimensional restriction of (\ref{eq:f}) to such a centre manifold.
%We find that the resulting one-dimensional map
%To leading order we find that the resulting map can be put in the form
%Such maps can be written as

The results in \cite{Si17c} were obtained
by performing calculations on one-dimensional centre manifolds.
%of (\ref{eq:f}) that exist near shrinking points.
This paper concerns circle maps
that capture an approximate return to a fundamental domain of such a manifold.
The maps are continuous, piecewise-linear, and degree-one.
They are written as\removableFootnote{
Here $w$ has a different definition than in earlier versions.
I like this form because it is defined on $[0,1)$
(rather than an interval containing $0$),
has just one weird quantity ($\frac{a_R-1}{a_R-a_L}$),
and the later formulas for the three parameters $a_L$, $a_R$ and $w$
are about as simple as possible.
}
\begin{equation}
z_{i+1} = g(z_i;a_L,a_R,w) \defeq \begin{cases}
\left( w + a_L \left( z_i - z_{\rm sw} \right) + z_{\rm sw} \right) {\rm \,mod\,} 1 \;,
& 0 \le z_i \le z_{\rm sw} \;, \\
\left( w + a_R \left( z_i - z_{\rm sw} \right) + z_{\rm sw} \right) {\rm \,mod\,} 1 \;,
& z_{\rm sw} \le z_i < 1 \;,
\end{cases}
\label{eq:g}
\end{equation}
where $z \in [0,1)$ is the state variable\removableFootnote{
We do not need consider the case $a_R < 1 < a_L$ in view of a coordinate change that switches $a_L$ and $a_R$:
presumably just $z \mapsto z - z_{\rm sw}$.
}\removableFootnote{
The naive coordinate change
(attempting to move the kink from $z_{\rm sw}$ to $\frac{1}{2}$)
\begin{equation}
z \mapsto \begin{cases}
\frac{(a_R-a_L) z}{2 (a_R-1)} \;, & 0 \le z \le z_{\rm sw} \\
1 + \frac{(z-1)}{2 \left( 1 - z_{\rm sw} \right)} \;, &
z_{\rm sw} \le z < 1
\end{cases} \;,
\end{equation}
is not helpful as the resulting PWL map involves more than two linear pieces.
},
\begin{equation}
a_L < 1 \;, \qquad
a_R > 1 \;, \qquad
w \in \mathbb{R} \;,
\label{eq:skewSawtoothParamRange}
\end{equation}
are parameters, and
\begin{equation}
z_{\rm sw} \defeq \frac{a_R-1}{a_R-a_L} \;.
\label{eq:zKink}
\end{equation}
As shown in Fig.~\ref{fig:cobwebSkewSawtooth},
$a_L$ and $a_R$ are the slopes of (\ref{eq:g})
and $w$ is the vertical displacement (modulo $1$) of (\ref{eq:g}) at the kink $z = z_{\rm sw}$.
%The maps (\ref{eq:g}) form a three-parameter family of continuous, piecewise-linear, degree-one circle maps,
%that, with (\ref{eq:skewSawtoothParamRange}), are homeomorphisms if and only if $a_L > 0$.
Equation (\ref{eq:g}) is termed a {\em skew sawtooth map}
because with $a_L + a_R = 2$ (giving $z_{\rm sw} = \frac{1}{2}$)
it is commonly known as a sawtooth map \cite{YaHa87}.
As a three-parameter family, (\ref{eq:g}) is equivalent to the ``tip maps'' studied in \cite{CaGa96}.

%%%%%%%%%%%%%%%%%%%%%%%%%%%%%%%%%%%%%%%%%%%%%%%%%%%%%%%%%%%%%
\begin{figure}[t!]
\begin{center}
\setlength{\unitlength}{1cm}
\begin{picture}(6,6)
\put(0,0){\includegraphics[height=6cm]{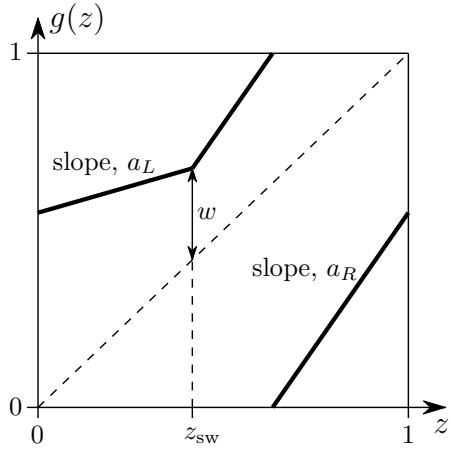}}
\put(5.76,.41){$z$}
\put(.66,5.82){$g(z)$}
\put(.41,.28){\footnotesize $0$}
\put(2.44,.33){\footnotesize $z_{\rm sw}$}
\put(5.34,.28){\footnotesize $1$}
\put(.12,.65){\footnotesize $0$}
\put(.12,5.37){\footnotesize $1$}
\put(2.62,3.26){\footnotesize $w$}
\put(.7,3.92){\footnotesize slope, $a_L$}
\put(3.35,2.5){\footnotesize slope, $a_R$}
\end{picture}
\caption{
A sketch of the skew sawtooth map (\ref{eq:g}).
\label{fig:cobwebSkewSawtooth}
}
\end{center}
\end{figure}
%%%%%%%%%%%%%%%%%%%%%%%%%%%%%%%%%%%%%%%%%%%%%%%%%%%%%%%%%%%%%

%Now let us outline the manner by which (\ref{eq:g}) can be used to describe the dynamics near shrinking points of (\ref{eq:f}).
%Details are provided throughout the remainder of the paper.
%Here we provide an example to outline the manner by which the circle maps (\ref{eq:g})
%approximate the dynamics of (\ref{eq:f}) near shrinking points with details to be provided in later sections.
%As the base shrinking point we use the shrinking point of Fig.~\ref{fig:modeLockEx20} with $\frac{m}{n} = \frac{3}{8}$.

%Here use the shrinking point of Fig.~\ref{fig:modeLockEx20} with $\frac{m}{n} = \frac{3}{8}$
%as an example to outline the manner by which the circle maps (\ref{eq:g})
%approximate the dynamics of (\ref{eq:f}) near shrinking points with details to be provided in later sections.
%This is an $\cS$-shrinking point with $\cS = LRRLRRLR$.
%Fig.~\ref{fig:modeLockApprox50} shows a magnified area of Fig.~\ref{fig:modeLockEx20} near the $\cS$-shrinking point.
%The $\cG^\pm_k$-mode-locking regions are shaded dark grey; all other nearby mode-locking regions are shaded light grey.

%%%%%%%%%%%%%%%%%%%%%%%%%%%%%%%%%%%%%%%%%%%%%%%%%%%%%%%%%%%%%
\begin{figure}[t!]
\begin{center}
\setlength{\unitlength}{1cm}
\begin{picture}(14.1,9.5)
\put(0,.3){\includegraphics[height=9cm]{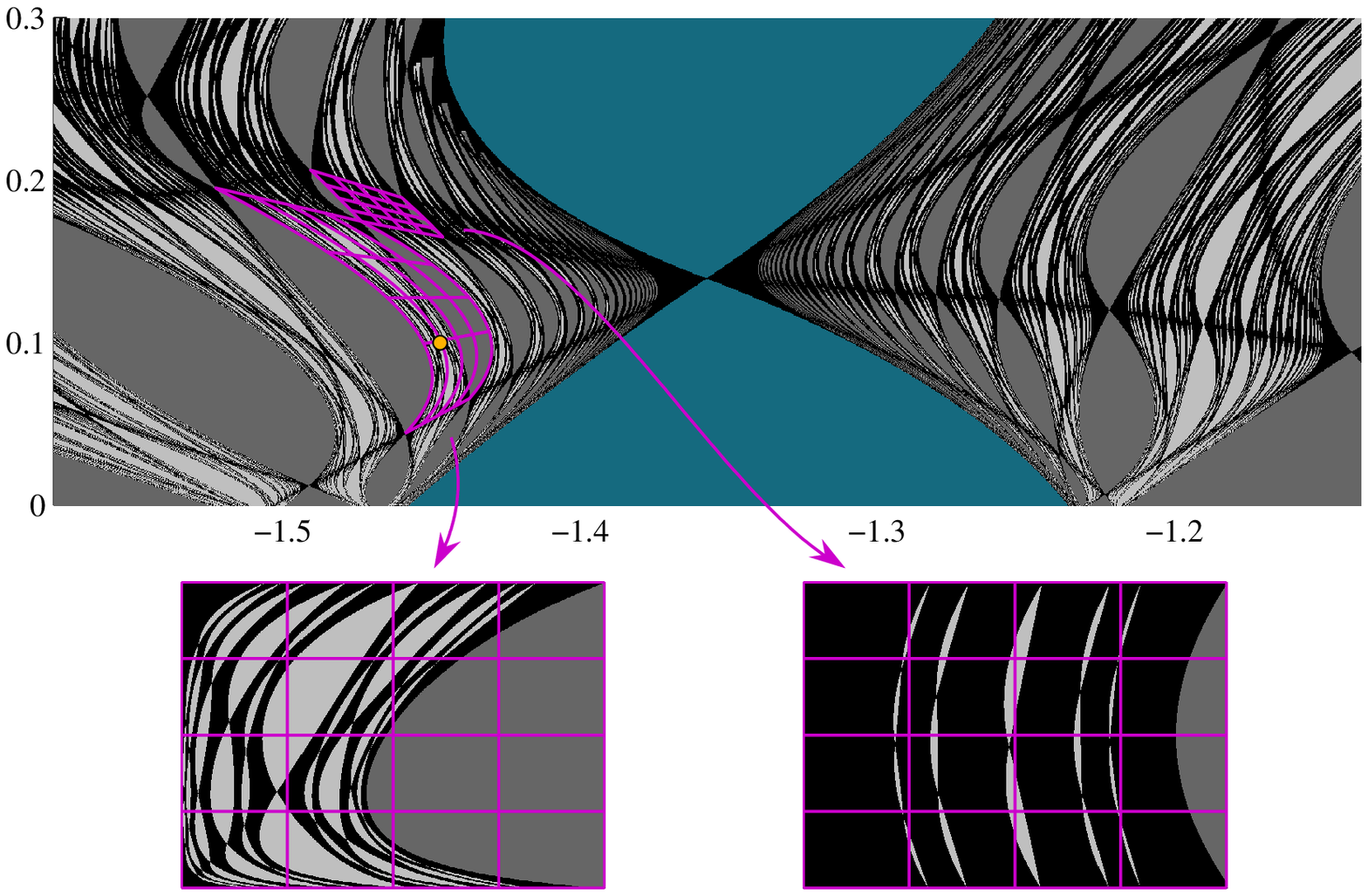}}
\put(7.42,3.76){$\tau_R$}
\put(0,6.5){$\delta_L$}
\put(7.27,9.2){\scriptsize $\frac{3}{8}$}
\put(1.6,9.32){\tiny $\frac{m_2^+}{n_2^+}$}
\put(3.2,9.32){\tiny $\frac{m_3^+}{n_3^+}$}
\put(3.86,9.32){\tiny $\frac{m_4^+}{n_4^+}$}
\put(12.9,9.32){\tiny $\frac{m_2^-}{n_2^-}$}
\put(11.58,9.32){\tiny $\frac{m_3^-}{n_3^-}$}
\put(10.84,9.32){\tiny $\frac{m_4^-}{n_4^-}$}
\put(1.76,0){\scriptsize $\delta \approx 0$}
\put(5.98,0){\scriptsize $\delta \approx \frac{1}{k^2}$}
\put(.88,.32){\scriptsize $\theta \approx \theta_{\rm max}$}
\put(.93,3.38){\scriptsize $\theta \approx \theta_{\rm min}$}
\put(1.3,1.78){\small $\Sigma^+_{2,0}$}
\put(12.74,1.78){\small $\Sigma^+_{3,1}$}
\put(7.93,0){\scriptsize $\delta \approx 0$}
\put(12.15,0){\scriptsize $\delta \approx \frac{1}{k^2}$}
\put(7.05,.32){\scriptsize $\theta \approx \theta_{\rm max}$}
\put(7.1,3.38){\scriptsize $\theta \approx \theta_{\rm min}$}
\end{picture}
\caption{
Mode-locking regions of (\ref{eq:f})--(\ref{eq:paramModeLockEx20})
centred about the $\cS$-shrinking point of Fig.~\ref{fig:modeLockEx20} with $\cS = \cF[3,3,8] = LRRLRRLR$ (here $\frac{m}{n} = \frac{3}{8}$).
In order to show many mode-locking regions (up to period $150$),
the mode-locking regions were approximated by detecting periodicity,
within some tolerance, of the forward orbit of $x = (0,0,0)$
computed for $10^5$ iterates over a $2048 \times 512$ grid of $\tau_R$ and $\delta_L$ values.
The rotation numbers $\frac{m_k^-}{n_k^-}$ and $\frac{m_k^+}{n_k^+}$
are given by (\ref{eq:lmnkpm}) where here $\frac{m^-}{n^-} = \frac{1}{3}$
and $\frac{m^+}{n^+} = \frac{2}{5}$ are the left and right Farey roots of $\frac{m}{n} = \frac{3}{8}$.
The sectors $\Sigma^+_{2,0}$ and $\Sigma^+_{3,1}$ are also shown.
To four decimal places, these admit the approximation (\ref{eq:SigmaApprox}) with
$\theta_{\rm min} = 5.1288$ and $\theta_{\rm max} = 6.1315$ for $\Sigma^+_{2,0}$, and
$\theta_{\rm min} = 4.9844$ and $\theta_{\rm max} = 5.1288$ for $\Sigma^+_{3,1}$.
%The grids are linear in $\delta$ and $\theta$.
The coloured dot in $\Sigma^+_{2,0}$ indicates the parameter values of Fig.~\ref{fig:iteratePhi}-A.
\label{fig:modeLockApprox50}
}
\end{center}
\end{figure}
%%%%%%%%%%%%%%%%%%%%%%%%%%%%%%%%%%%%%%%%%%%%%%%%%%%%%%%%%%%%%

Here it is briefly shown how (\ref{eq:g}) can be used to approximate the dynamics of (\ref{eq:f}) near shrinking points,
with details to be provided in later sections.
In a neighbourhood of an $\cS$-shrinking point, the inner boundaries
of the $\cG^\pm_k$-mode-locking regions are used to define a grid of sectors $\Sigma^\pm_{k,\Delta \ell}$,
where $\Delta \ell \in \mathbb{Z}$.
As an example, Fig.~\ref{fig:modeLockApprox50} shows a magnification of Fig.~\ref{fig:modeLockEx20}
centred at the $\cS$-shrinking point with $\cS = LRRLRRLR$.
Here $\cG^\pm_k$-mode-locking regions are shaded dark grey and all other nearby mode-locking regions are shaded light grey.
Fig.~\ref{fig:modeLockApprox50} also shows two representative sectors, $\Sigma^+_{2,0}$ and $\Sigma^+_{3,1}$,
and indicates the dynamics of (\ref{eq:f}) in these sectors.
Fig.~\ref{fig:sectorsApprox50} shows the approximate location of the sectors
$\Sigma^\pm_{k,\Delta \ell}$ for Fig.~\ref{fig:modeLockApprox50}
as given by truncating an asymptotic expansion in powers of $\frac{1}{k}$ to leading order\removableFootnote{
This was obtained by using Newton's method (in two dimensions) to achieve the coordinate change.

To numerically create the insets of Fig.~\ref{fig:modeLockApprox50}:
\begin{enumerate}
\setlength{\itemsep}{0pt}
\item
I located the $\cG^\pm[k,\Delta \ell]$ and $\cG^\pm[k,\Delta \ell+1]$-shrinking points,
and determined their $\theta$-values.
\item
I created a vector of equally spaced $\theta$-values using those of (i) as end points.
For each value of $\theta$ in this vector, I found the $r$-value
of the mode-locking region boundary connecting the
$\cG^\pm[k,\Delta \ell]$ and $\cG^\pm[k,\Delta \ell+1]$-shrinking points.
\item
I repeated the above two steps for the
$\cG^\pm[k+1,\Delta \ell]$ and $\cG^\pm[k+1,\Delta \ell+1]$-shrinking points.
\item
Note, the two vectors of $\theta$-values are slightly different
(they differ by $\cO \!\left( \frac{1}{k^2} \right)$).
For each corresponding pair of $\theta$-values, and their corresponding $r$-values,
I created a vector of equally spaced $\theta$-values and $r$-values.
This gives me a grid of $(r,\theta)$-values over the sector.
\item
For each point $(r,\theta)$ in this grid,
I computed the corresponding point $(\tau_R,\delta_L)$
(using Newton's method in 2d)
and iterated the map at this point.
\end{enumerate}
}.

%In a neighbourhood of an $\cS$-shrinking point,
%we define a grid of sectors $\Sigma^\pm_{k,\Delta \ell}$, where $k, \Delta \ell \in \mathbb{Z}$.
%The inner and outer boundaries of each $\Sigma^\pm_{k,\Delta \ell}$ are defined as the inner boundaries
%of the $\cG^\pm_k$ and $\cG^\pm_{k+1}$-mode-locking regions.
%The remaining two boundaries are linear.
%Fig.~\ref{fig:sectorsApprox50} shows their approximate location as given by truncating an asymptotic expansion
%in powers of $\frac{1}{k}$ to leading order.

%%%%%%%%%%%%%%%%%%%%%%%%%%%%%%%%%%%%%%%%%%%%%%%%%%%%%%%%%%%%%
\begin{figure}[t!]
\begin{center}
\setlength{\unitlength}{1cm}
\begin{picture}(15,9)
\put(0,0){\includegraphics[height=9cm]{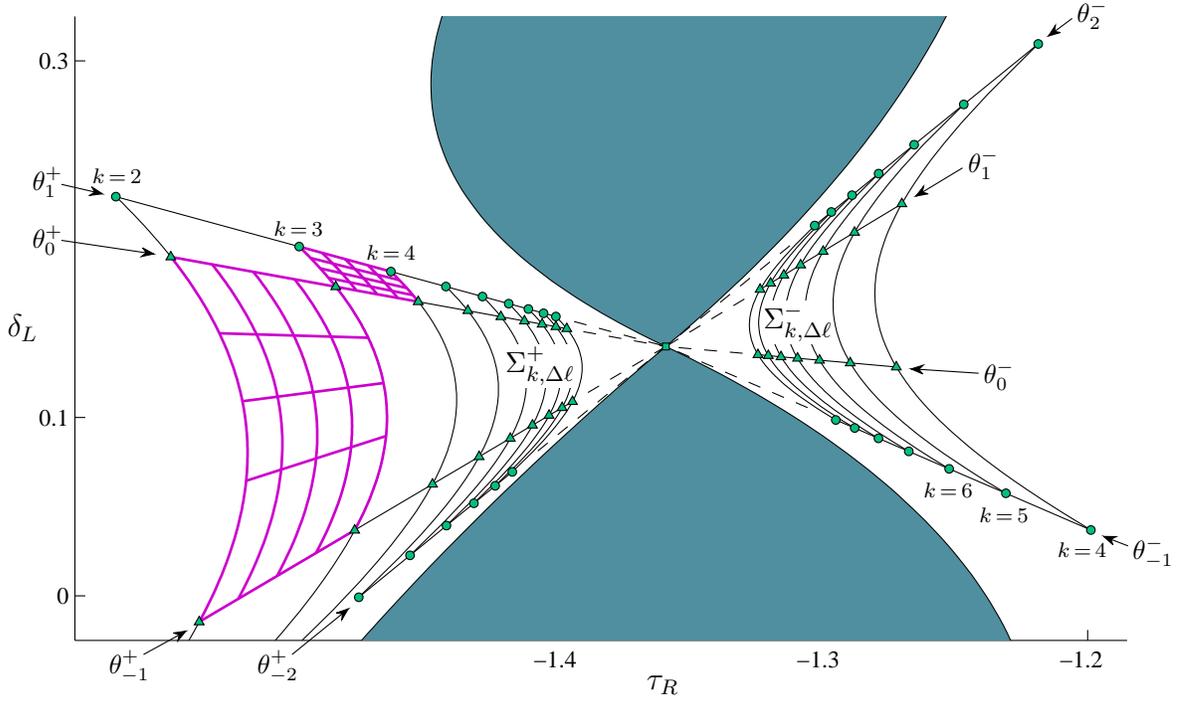}}
\put(8.5,0){$\tau_R$}
\put(0,4.7){$\delta_L$}
\put(6.62,4.13){\small \colorbox{white}{\makebox(.7,.35){$\Sigma^+_{k,\Delta \ell}$}}}
\put(10.05,4.71){\small \colorbox{white}{\makebox(.7,.35){$\Sigma^-_{k,\Delta \ell}$}}}
\put(.33,6.65){\footnotesize $\theta^+_1$}
\put(.33,5.88){\footnotesize $\theta^+_0$}
\put(1.35,.22){\footnotesize $\theta^+_{-1}$}
\put(3.32,.22){\footnotesize $\theta^+_{-2}$}
\put(14.96,1.74){\footnotesize $\theta^-_{-1}$}
\put(12.98,4.08){\footnotesize $\theta^-_0$}
\put(12.77,6.89){\footnotesize $\theta^-_1$}
\put(14.22,8.88){\footnotesize $\theta^-_2$}
\put(1.13,6.74){\scriptsize $k \hspace{-.3mm}=\hspace{-.3mm} 2$}
\put(3.54,6.04){\scriptsize $k \hspace{-.3mm}=\hspace{-.3mm} 3$}
\put(4.77,5.72){\scriptsize $k \hspace{-.3mm}=\hspace{-.3mm} 4$}
\put(13.96,1.76){\scriptsize $k \hspace{-.3mm}=\hspace{-.3mm} 4$}
\put(12.92,2.23){\scriptsize $k \hspace{-.3mm}=\hspace{-.3mm} 5$}
\put(12.18,2.56){\scriptsize $k \hspace{-.3mm}=\hspace{-.3mm} 6$}
\end{picture}
\caption{
The leading order approximation to the location of the sectors $\Sigma^\pm_{k,\Delta \ell}$
for the example shown in Fig.~\ref{fig:modeLockApprox50}.
To four decimal places:
$\theta^-_{-1} = 1.6222$,
$\theta^-_0 = 2.1401$,
$\theta^-_1 = 3.0002$,
$\theta^-_2 = 3.1268$,
$\theta^+_1 = 4.9844$,
$\theta^+_0 = 5.1288$,
$\theta^+_{-1} = 6.1315$,
$\theta^+_{-2} = 6.2705$.
These values provide the bounds for (\ref{eq:SigmaApprox}).
%for the shrinking point of Fig.~\ref{fig:modeLockEx20} with $\frac{m}{n} = \frac{3}{8}$.
%Specifically, we show $\Sigma^+_{k,\Delta \ell}$ for all $2 \le k \le 20$ and $-1 \le \Delta \ell \le 1$,
%and $\Sigma^-_{k,\Delta \ell}$ for all $4 \le k \le 20$ and $0 \le \Delta \ell \le 2$.
%The sectors $\Sigma^+_{2,-1}$ and $\Sigma^+_{3,0}$, shown in Fig.~\ref{fig:sectorsApprox50}, are highlighted.	
\label{fig:sectorsApprox50}
}
\end{center}
\end{figure}
%%%%%%%%%%%%%%%%%%%%%%%%%%%%%%%%%%%%%%%%%%%%%%%%%%%%%%%%%%%%%

Within each $\Sigma^\pm_{k,\Delta \ell}$,
$\delta \ge 0$ measures the distance to the outer boundary of
$\Sigma^\pm_{k,\Delta \ell}$ in a radial direction
and $\theta \in [0,2 \pi)$ measures the angle about the $\cS$-shrinking point.
In $(\delta,\theta)$-coordinates the sectors are rectangular in the approximation of Fig.~\ref{fig:sectorsApprox50}:
\begin{equation}
\Sigma^\pm_{k,\Delta \ell} \approx \left\{ (\delta,\theta) ~\middle|~
0 \le \delta \le {\textstyle \frac{1}{k^2}} ,\,
\theta_{\rm min} \le \theta \le \theta_{\rm max} \right\},
\label{eq:SigmaApprox}
\end{equation}
where $\theta_{\rm min}, \theta_{\rm max} \in [0,2 \pi)$ depend only on the value of $\Delta \ell$\removableFootnote{
The point is that we require a mesh that is linear in both $\theta$ and $r$ because
this is what we have done for Fig.~\ref{fig:modeLockSkewSaw50},
as it seems like the most natural thing to do.
I converted the $(\delta,\theta)$-mesh into $(\tau_R,\delta_L)$-coordinates
by using Newton's method in two-dimensions.
}.

For $\Sigma^+_{2,0}$ and $\Sigma^+_{3,1}$,
Fig.~\ref{fig:modeLockSkewSaw50} shows mode-locking regions of (\ref{eq:g}) using
\begin{equation}
a_L = \frac{\tan(\theta)}{\tan \!\left( \theta_{\rm min} \right)} \;, \qquad
a_R = \frac{\tan(\theta)}{\tan \!\left( \theta_{\rm max} \right)} \;, \qquad
w = k^2 \delta \;.
\label{eq:aLaRw}
\end{equation}
Equation (\ref{eq:aLaRw}) represents the appropriate transformation from
$(\delta,\theta)$-coordinates to the parameter space of (\ref{eq:g}) for $\Sigma^+_{2,0}$ and $\Sigma^+_{3,1}$.
%This slice is characterised by the value of $\frac{a_R}{a_L}$ which constant throughout 
%$\Sigma^+_{k,\Delta \ell}$ and independent of $k$.
%Slightly different equations are needed for $\Sigma^-_{k,\Delta \ell}$, see \S\ref{sub:results}.
%The value of $w$ varies from $0$ to $1$ as we move from the
%outer boundary to the inner boundary of $\Sigma^\pm_{k,\Delta \ell}$.
%On the remaining two boundaries we have $a_L = 1$ on one boundary and $a_R = 1$ on the other.
%By comparing Fig.~\ref{fig:modeLockSkewSaw50} with Fig.~\ref{fig:modeLockApprox50},
%we see that the mode-locking regions of (\ref{eq:g})
%match well to those of (\ref{eq:f}) throughout $\Sigma^+_{2,0}$ and $\Sigma^+_{3,1}$.
%For sectors closer to the shrinking point the approximation is more accurate.
%In this way the one-dimensional map (\ref{eq:g})
%captures the dynamics of the $N$-dimensional map (\ref{eq:f}) near a shrinking point.
By comparing Figs.~\ref{fig:modeLockApprox50} and \ref{fig:modeLockSkewSaw50},
it is evident that the dynamics of (\ref{eq:f}) in $\Sigma^+_{2,0}$ and $\Sigma^+_{3,1}$
matches well to that of (\ref{eq:g}) using (\ref{eq:aLaRw}).
Moreover, the dynamics of (\ref{eq:f}) in each $\Sigma^+_{k,0}$ and $\Sigma^+_{k,1}$
matches that of (\ref{eq:g}) using (\ref{eq:aLaRw})
with an error of order $\cO \!\left( \frac{1}{k} \right)$ and where
one iterate of (\ref{eq:g}) corresponds to roughly $k n$ iterates of (\ref{eq:f}) (here $n = 8$).

The remainder of this paper is organised as follows.
In \S\ref{sec:results}, the sectors
$\Sigma^\pm_{k,\Delta \ell}$ and $(\delta,\theta)$-coordinates are defined more precisely.
The correspondence between (\ref{eq:f}) and (\ref{eq:g}) is clarified and 
used to explain the bifurcation structure of (\ref{eq:f}) near a typical shrinking point.
In \S\ref{sec:approxRecurrent} some important identities for the symbolic
itineraries of periodic solutions in $\cG^\pm$-mode-locking regions are derived
and an attracting one-dimensional centre manifold $W^c$
and a fundamental domain $\Omega_{\Delta \ell} \subset W^c$ are introduced.

In \S\ref{sec:Lambda}, $\Omega_{\Delta \ell}$ is enlarged into an $N$-dimensional set $\Phi$
that forward orbits of (\ref{eq:f}) regularly visit.
The set $\Phi$ is interpreted as a cylinder and topological arguments are used to prove that
the first return map $F : \Phi \to \Phi$
has an attracting invariant set that is homotopic to $\Omega_{\Delta \ell}$.
This shows that (\ref{eq:f}) has an attracting invariant set homotopic to a circle
on which the dynamics is well approximated by (\ref{eq:g}).
In \S\ref{sec:return} the necessary calculations to achieve this approximation are performed,
with the main result given by Theorem \ref{th:main}.
Finally \S\ref{sec:conc} provides concluding remarks.

Throughout this paper, but mostly in proofs, particular places in \cite{Si17c} are referenced.
Some proofs are deferred to Appendix \ref{app:proofs}
and some formulas of \cite{Si17c} are given in Appendix \ref{app:formulas}.
For brevity, Sections \ref{sec:approxRecurrent}--\ref{sec:return} only provide
calculations for sectors $\Sigma^+_{k,\Delta \ell}$ with $\Delta \ell \ge 0$.

%=====================================================================
\section{Main results}
\label{sec:results}
\setcounter{equation}{0}

In this section we first describe periodic solutions of (\ref{eq:f}), \S\ref{sub:perSolns}.
We then discuss shrinking points and briefly review the notation of \cite{Si17c} in \S\ref{sub:shrPoints}.
In \S\ref{sub:nearby} we describe $\cG^\pm_k$-mode-locking regions and
in \S\ref{sub:sectors} define the sectors $\Sigma^\pm_{k,\Delta \ell}$.
We then detail the correspondence between (\ref{eq:f}) and (\ref{eq:g}) in \S\ref{sub:skewSawtoothApprox}
and finally use this to explain the dynamics of (\ref{eq:f})
near a typical shrinking point in \S\ref{sub:skewSawtoothDyns}.

%---------------------------------------------------------------------
%\subsection{A summary of dynamics near shrinking points}
\subsection{Periodic solutions}
\label{sub:perSolns}

We denote the two components of (\ref{eq:f}) by
\begin{equation}
f^L(x;\xi) \defeq A_L(\xi) x + B(\xi), \qquad
f^R(x;\xi) \defeq A_R(\xi) x + B(\xi).
\label{eq:fLfR}
\end{equation}
Given a periodic symbol sequence $\cS \in \{ L, R \}^{\mathbb{Z}}$ of period $n$, we let 
\begin{equation}
f^{\cS} \defeq f^{\cS_{n-1}} \circ \cdots \circ f^{\cS_0} \;,
\label{eq:fS}
\end{equation}
denote the composition of $f^L$ and $f^R$ in the order specified by $\cS$.
The map $f^{\cS}$ is given by
\begin{equation}
f^{\cS}(x) = M_{\cS} x + P_{\cS} B \;,
\label{eq:fS2}
\end{equation}
where
\begin{align}
M_{\cS} &\defeq A_{\cS_{n-1}} \cdots A_{\cS_0} \;, \label{eq:MS} \\
P_{\cS} &\defeq I + A_{\cS_{n-1}} + A_{\cS_{n-1}} A_{\cS_{n-2}} + \cdots +
A_{\cS_{n-1}} \cdots A_{\cS_1} \;. \label{eq:PS}
\end{align}
An {\em $\cS$-cycle} is defined as an $n$-tuple, $\{ x^{\cS}_i \}_{i=0}^{n-1}$ for which
\begin{equation}
f^{\cS_0} \!\left( x^{\cS}_0 \right) = x^{\cS}_1 \;, \hspace{3mm}
f^{\cS_1} \!\left( x^{\cS}_1 \right) = x^{\cS}_2 \;, \;\ldots\;, \hspace{3mm} 
f^{\cS_{n-1}} \!\left( x^{\cS}_{n-1} \right) = x^{\cS}_0 \;.
\end{equation}
If $s^{\cS}_i \le 0$ for every $i$ for which $\cS_i = L$,
and $s^{\cS}_i \ge 0$ for every $i$ for which $\cS_i = R$,
then each point $x^{\cS}_i$ lies on the side of the switching manifold corresponding to $\cS_i$.
In this case the $\cS$-cycle is a periodic solution of $f$
and we say it is {\em admissible}.
If it is not admissible we say it is {\em virtual}.

Each point $x^{\cS}_i$ is a fixed point of $f^{\cS^{(i)}}$,
where $\cS^{(i)}$ denotes the $i^{\rm th}$ left shift permutation of $\cS$
(i.e.~$\cS^{(i)}_j = \cS_{i+j}$, for all $j \in \mathbb{Z}$).
If $\det \!\left( I-M_{\cS} \right) \ne 0$,
then $x^{\cS}_i$ is unique and given by
\begin{equation}
x^{\cS}_i = \left( I - M_{\cS^{(i)}} \right)^{-1} P_{\cS^{(i)}} B \;.
\end{equation}

%%%%%%%%%%%%%%%%%%%%%%%%%%%%%%%%%%%%%%%%%%%%%%%%%%%%%%%%%%%%%
\begin{figure}[t!]
\begin{center}
\setlength{\unitlength}{1cm}
\begin{picture}(14.9,5.7)
\put(0,0){\includegraphics[height=5.4cm]{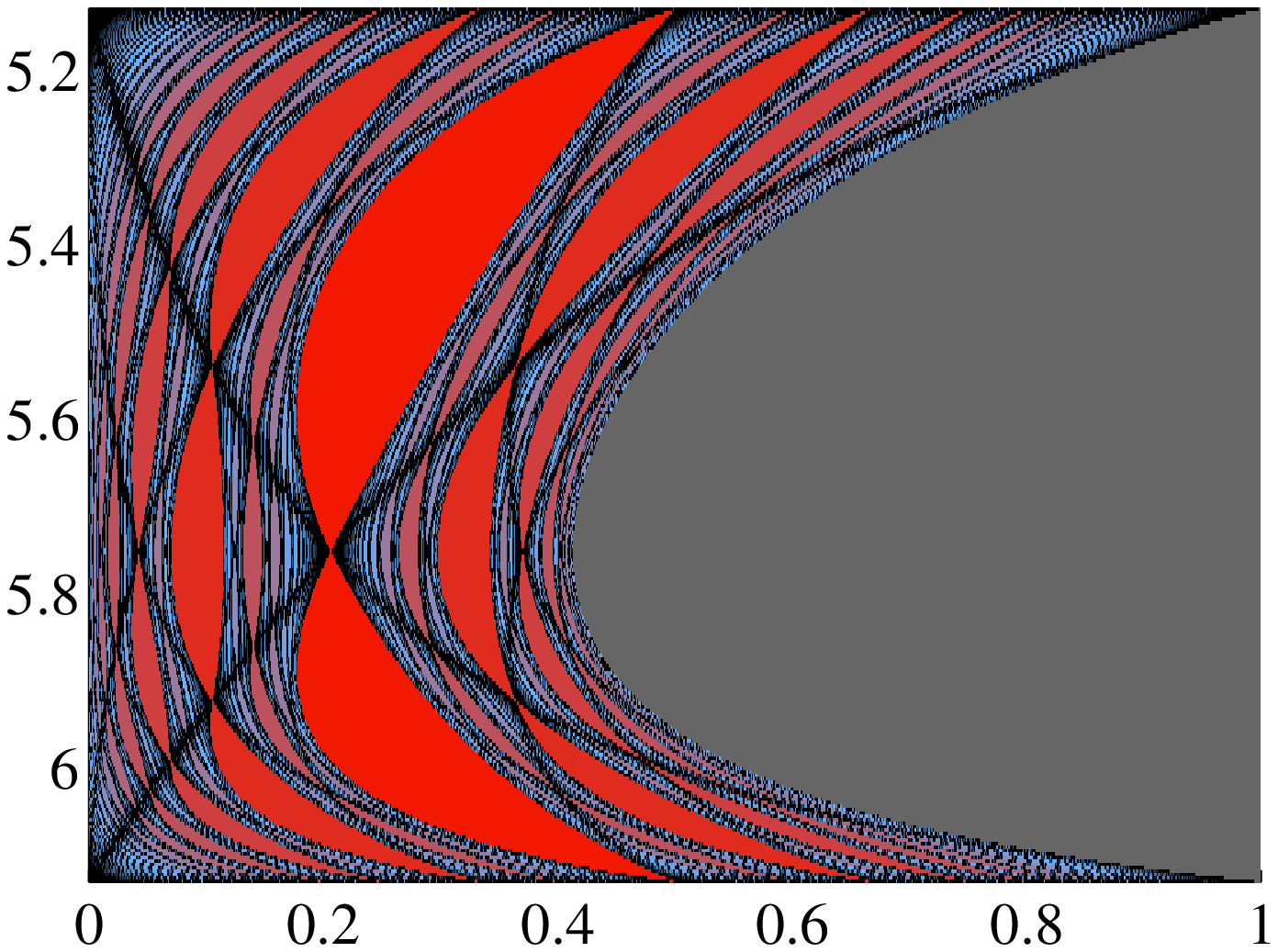}}
\put(7.7,0){\includegraphics[height=5.4cm]{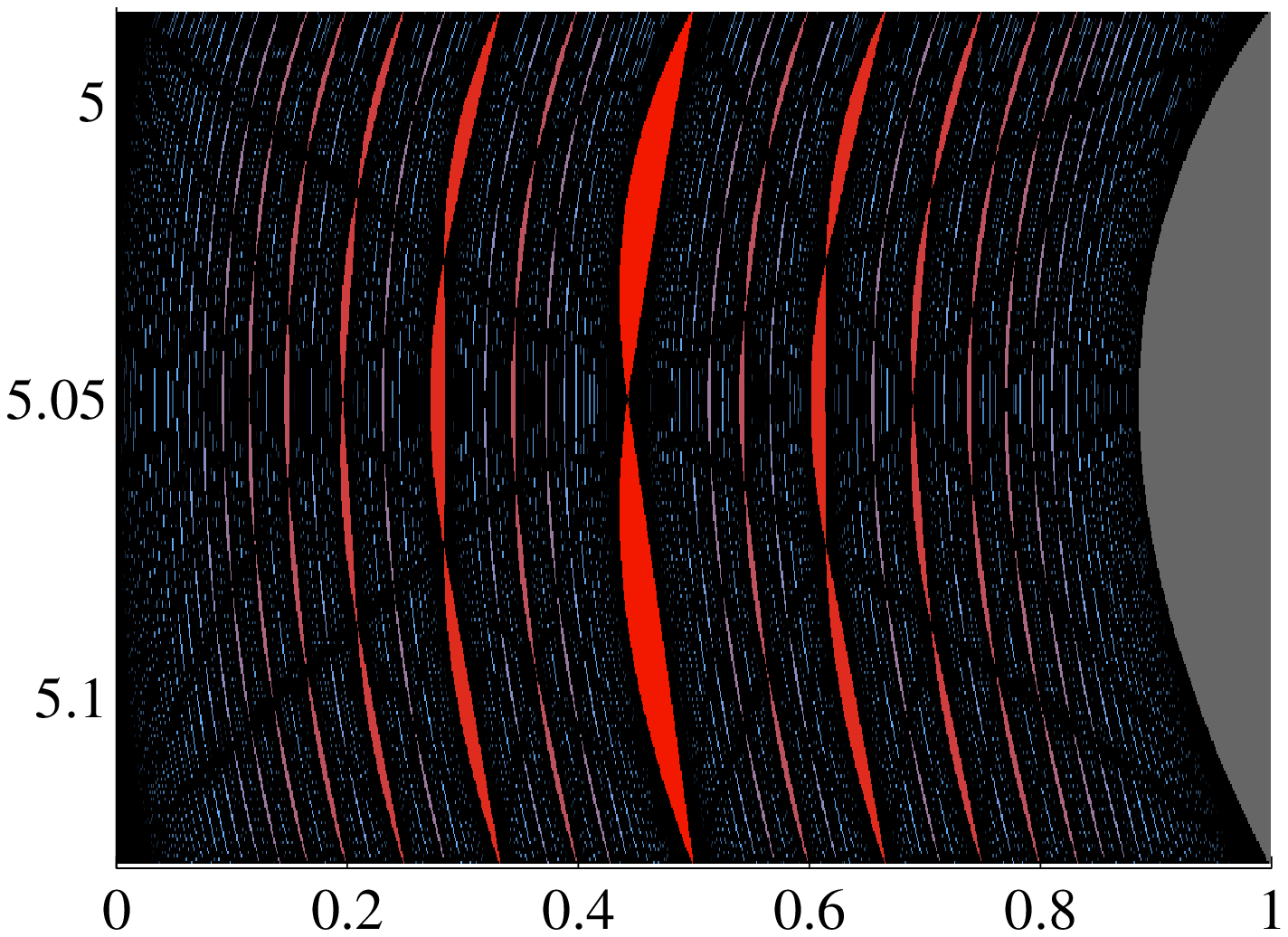}}
%\put(.1,5.8){\large \sf \bfseries A}
\put(3.6,0){$k^2 \delta$}
\put(0,3.7){$\theta$}
\put(11.3,0){$k^2 \delta$}
\put(7.7,3.7){$\theta$}
\put(3.5,5.58){$\Sigma^+_{k,0}$}
\put(11.2,5.58){$\Sigma^+_{k,1}$}
\end{picture}
\caption{
Mode-locking regions of (\ref{eq:g}) with (\ref{eq:aLaRw})
corresponding to $\Sigma^+_{k,0}$ and $\Sigma^+_{k,1}$ of Fig.~\ref{fig:modeLockApprox50}, where $k \in \mathbb{Z}^+$.
In both plots the rotation number (with respect to (\ref{eq:g})) associated with each mode-locking region
is equal to its value of $w = k^2 \delta$ at both the top of the figure (where $a_L = 1$)
and the bottom of the figure (where $a_R = 1$).
To four decimal places, $\frac{a_R}{a_L} = 14.7884$ for $\Sigma^+_{k,0}$,
and $\frac{a_R}{a_L} = 1.5856$ for $\Sigma^+_{k,1}$.
Note, the $\theta$-axis is reversed to match Fig.~\ref{fig:modeLockApprox50}.
\label{fig:modeLockSkewSaw50}
}
\end{center}
\end{figure}
%%%%%%%%%%%%%%%%%%%%%%%%%%%%%%%%%%%%%%%%%%%%%%%%%%%%%%%%%%%%%

Given $\ell,m,n \in \mathbb{Z}^+$, with $\ell < n$, $m < n$ and ${\rm gcd}(m,n) = 1$,
we define a symbol sequence $\cF[\ell,m,n] \in \{ L, R \}^{\mathbb{Z}}$ by
\begin{equation}
\cF[\ell,m,n]_i \defeq
\begin{cases}
L \;, & i m {\rm ~mod~} n < \ell \;, \\
R \;, & i m {\rm ~mod~} n \ge \ell \;.
\end{cases}
\label{eq:rss}
\end{equation}
Each $\cF[\ell,m,n]$ describes the action of stepping through $n$ points on a circle,
$\ell$ of which lie to the left of the switching manifold,
with rotation number $\frac{m}{n}$, \cite{SiMe09,Si10}.
Indeed, we call $\frac{m}{n}$ the {\em rotation number} of $\cF[\ell,m,n]$,
and refer to such symbol sequences and their shift permutations as {\em rotational}.
Rotational symbol sequences satisfy the important identity
\begin{equation}
\cF[\ell,m,n]^{\overline{0} \,\overline{\ell d}} = \cF[\ell,m,n]^{(-d)},
\label{eq:rssMainIdentity}
\end{equation}
where $d$ denotes the multiplicative inverse of $m$ modulo $n$,
and we use the notation $\cS^{\overline{i}}$ to
denote the symbol sequence that differs from $\cS$ in only the
indices $i + j n$, for all $j \in \mathbb{Z}$.

%---------------------------------------------------------------------
\subsection{Shrinking points}
\label{sub:shrPoints}

We use rotational symbol sequences to define shrinking points of (\ref{eq:f}).
Let $\cS = \cF[\ell,m,n]$ be a rotational symbol sequence with $2 \le \ell \le n-2$.
Suppose that at some point in parameter space $\xi$,
the following three genericity conditions are satisfied:
\begin{equation}
e_1^{\sf T} {\rm adj} \!\left( I - A_L \right) B \ne 0 \;, \qquad
\det \!\left( I - M_{\cS^{\overline{0}}} \right) \ne 0 \;, \qquad
\det \!\left( I - M_{\cS^{\overline{\ell d}}} \right) \ne 0 \;.
\label{eq:shrPointGeneric}
\end{equation}
In view of the second condition,
the $\cS^{\overline{0}}$-cycle is unique.
If the $\cS^{\overline{0}}$-cycle is also admissible,
with $s^{\cS^{\overline{0}}}_i = 0$ if and only if $i = 0$ or $i = \ell d {\rm \,mod\,} n$,
then we say that $\xi$ is an {\em $\cS$-shrinking point}.
%Shrinking points are codimension-two phenomena in view of the requirement
%$s^{\cS^{\overline{0}}}_0 = s^{\cS^{\overline{0}}}_{\ell d} = 0$.

%As in \cite{Si17c}, we use abbreviated notation for
%various key quantities of an $\cS$-shrinking point.
As in \cite{Si17c}, at an $\cS$-shrinking point we let
\begin{equation}
a \defeq \det \!\left( I-M_{\cS^{\overline{0}}} \right), \qquad
b \defeq \det \!\left( I-M_{\cS^{\overline{\ell d}}} \right),
\label{eq:ab}
\end{equation}
and
\begin{equation}
y_i \defeq x^{\cS^{\overline{0}}}_i \;, \qquad
t_i \defeq s^{\cS^{\overline{0}}}_i \;,
\label{eq:yiti}
\end{equation}
for each $i$.
Many identities involving quantities associated with an $\cS$-shrinking point,
such as (\ref{eq:ab})--(\ref{eq:yiti}), are given in \cite{Si17c}.
Of these, perhaps the most fundamental is
\begin{equation}
\frac{a}{b} = -\frac{t_d t_{(\ell-1)d}}{t_{-d} t_{(\ell+1)d}} \;,
\label{eq:fourtIdentity}
\end{equation}
which was first given in \cite{SiMe10}.

By Lemma 5.8 of \cite{Si17c}, at an $\cS$-shrinking point $0$ is an eigenvalue of $M_{\cS}$ with multiplicity one.
We let
\begin{equation}
\rho_{\rm max} \defeq \max_{j=2,\ldots,N} |\rho_j|,
\label{eq:rhoMax}
\end{equation}
where $\rho_j$ are the eigenvalues of $M_{\cS}$, counting multiplicity, with $\rho_1 = 1$, and
\begin{equation}
c \defeq \prod_{j=2}^N (1-\rho_j).
\label{eq:c}
\end{equation}
If $\rho_{\rm max} < 1$ then $\cS$-cycles are
stable for some parameter values near the $\cS$-shrinking point
ensuring that the $\cS$-shrinking point is connected to part of a mode-locking region involving $\cS$-cycles
(instead of merely a periodicity region involving only unstable $\cS$-cycles).
Note that $\rho_{\rm max} < 1$ implies $c > 0$ (Lemma 7.4 of \cite{Si17c}).

%---------------------------------------------------------------------
%\subsection{Nearby mode-locking regions}
\subsection{$\cG^\pm_k$-mode-locking regions}
\label{sub:nearby}

Let $\cF[\ell,m,n]$ be a rotational symbol sequence.
Let $\frac{m^-}{n^-}$ and $\frac{m^+}{n^+}$ denote the left and right Farey roots of $\frac{m}{n}$, and let
$\ell^- \defeq \left\lfloor \frac{\ell n^-}{n} \right\rfloor$ and
$\ell^+ \defeq \left\lceil \frac{\ell n^+}{n} \right\rceil$.
Let $k \in \mathbb{Z}^+$ and $\Delta \ell \in \mathbb{Z}$ with $| \Delta \ell | < k$.
We then define
\begin{equation}
\cG^\pm[k,\Delta \ell] \defeq \cF \!\left[ \ell_k^\pm + \Delta \ell, m_k^\pm, n_k^\pm \right],
\label{eq:Gplusminus}
\end{equation}
where
\begin{equation}
\ell_k^\pm \defeq k \ell + \ell^\pm \;, \qquad
m_k^\pm \defeq k m + m^\pm \;, \qquad
n_k^\pm \defeq k n + n^\pm \;.
\label{eq:lmnkpm}
\end{equation}
We also let $d_k^\pm$ denote the multiplicative inverse of $m_k^\pm$ modulo $n_k^\pm$.
The rotation number of $\cG^\pm[k,\Delta \ell]$ is $\frac{m_k^\pm}{n_k^\pm}$.
These rotation numbers represent the first level of complexity relative to $\frac{m}{n}$ \cite{ZhMo06b}.
As mentioned in \S\ref{sec:intro},
we refer to a mode-locking region with rotation number $\frac{m_k^\pm}{n_k^\pm}$
as a $\cG^\pm_k$-mode-locking region.

%%%%%%%%%%%%%%%%%%%%%%%%%%%%%%%%%%%%%%%%%%%%%%%%%%%%%%%%%%%%%
\begin{figure}[t!]
\begin{center}
\setlength{\unitlength}{1cm}
\begin{picture}(12,6)
\put(0,0){\includegraphics[height=6cm]{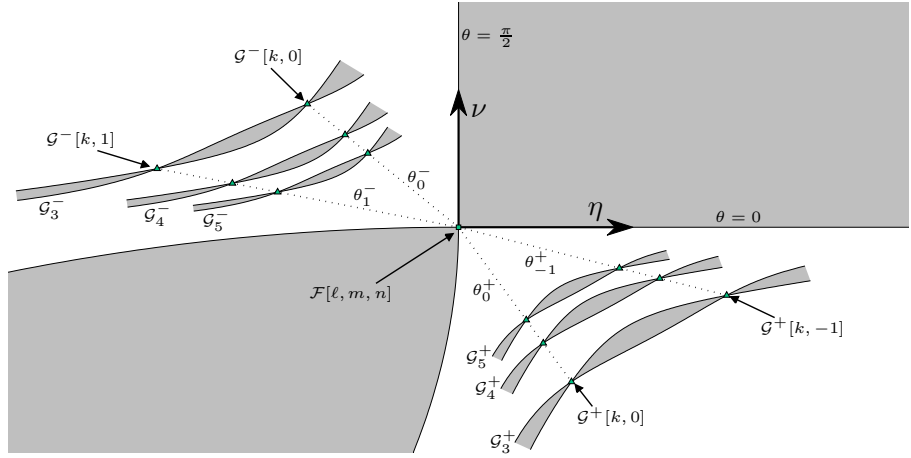}}
\put(7.7,3.17){\small $\eta$}
\put(6.15,4.47){\small $\nu$}
\put(9.4,3.08){\tiny $\theta = 0$}
\put(6.02,5.5){\tiny $\theta = \frac{\pi}{2}$}
\put(6.17,2.1){\tiny $\theta^+_0$}
\put(6.86,2.47){\tiny $\theta^+_{-1}$}
\put(5.3,3.6){\tiny $\theta^-_0$}
\put(4.58,3.37){\tiny $\theta^-_1$}
\put(7.56,.41){\tiny $\cG^+[k,0]$}
\put(10,1.6){\tiny $\cG^+[k,-1]$}
\put(3,5.22){\tiny $\cG^-[k,0]$}
\put(.5,4.11){\tiny $\cG^-[k,1]$}
\put(6.08,1.22){\tiny $\cG^+_5$}
\put(6.2,.8){\tiny $\cG^+_4$}
\put(6.4,.09){\tiny $\cG^+_3$}
\put(2.58,3.03){\tiny $\cG^-_5$}
\put(1.8,3.1){\tiny $\cG^-_4$}
\put(.4,3.2){\tiny $\cG^-_3$}
\put(4,2.06){\tiny $\cF[\ell,m,n]$}
\end{picture}
\caption{
A sketch of typical $\cG^\pm_k$-mode-locking regions
in $(\eta,\nu)$-coordinates in the case $a < 0$.
(Taken from \cite{Si17c}.)
%A sketch of typical $\cG^\pm_k$-mode-locking regions (for $k = 3,4,5$)
%near an $\cF[\ell,m,n]$-shrinking point
%in $(\eta,\nu)$-coordinates in the case $a < 0$.
%Each shrinking point is labelled by its associated symbol sequence.
%Formulas for the angles, denoted $\theta^\pm_{\Delta \ell}$, 
%about which sequences of $\cG^\pm[k,\Delta \ell]$-shrinking points
%emanate from the $\cF[\ell,m,n]$-shrinking point are given by (\ref{eq:thetaPlus})-(\ref{eq:thetaMinus}).
%In the case $a > 0$, the relative location of the $\cG^+_k$ and $\cG^-_k$-mode-locking regions is reversed.
\label{fig:shrPointSchem}
}
\end{center}
\end{figure}
%%%%%%%%%%%%%%%%%%%%%%%%%%%%%%%%%%%%%%%%%%%%%%%%%%%%%%%%%%%%%

For simplicity we assume $\xi \in \mathbb{R}^2$ and write $\xi = (\xi_1,\xi_2)$.
In a neighbourhood of an $\cS$-shrinking point, where $\cS = \cF[\ell,m,n]$, let
\begin{equation}
\eta \defeq s^{\cS^{\overline{0}}}_0(\xi_1,\xi_2), \qquad
\nu \defeq s^{\cS^{\overline{0}}}_{\ell d}(\xi_1,\xi_2).
\label{eq:etanu}
\end{equation}
Throughout this paper we assume that
$(\xi_1,\xi_2) \to (\eta,\nu)$ is a locally invertible coordinate change.
That is, $\det(J) \ne 0$,
where $J \defeq \begin{bmatrix}
\frac{\partial \eta}{\partial \xi_1} &
\frac{\partial \eta}{\partial \xi_2} \\
\frac{\partial \nu}{\partial \xi_1} &
\frac{\partial \nu}{\partial \xi_2}
\end{bmatrix}$ evaluated at the shrinking point.
We have $s^{\cS^{\overline{0}}}_0 = s^{\cS^{\overline{0}}}_{\ell d} = 0$ at the $\cS$-shrinking point.
Thus in $(\eta,\nu)$-coordinates the $\cS$-shrinking point is located at the origin.

As in \cite{Si17c}, we define polar coordinates $(r,\theta)$ by
\begin{equation}
\eta = \left| \frac{c t_d}{a} \right| r \cos(\theta), \qquad
\nu = \left| \frac{c t_{(\ell-1)d}}{a} \right| r \sin(\theta),
\label{eq:polarCoords}
\end{equation}
where $a$, $c$, $t_d$ and $t_{(\ell-1)d}$ are defined by (\ref{eq:ab})--(\ref{eq:yiti}).
If $a < 0$, then $\cG^+_k$-mode-locking regions exist for $\theta \in \left( \frac{3 \pi}{2}, 2 \pi \right)$
and $\cG^-_k$-mode-locking regions exist for $\theta \in \left( \frac{\pi}{2}, \pi \right)$,
see Fig.~\ref{fig:shrPointSchem}.
If $a > 0$, then the opposite is true.

For large values of $k$,
the $\cG^\pm_k$-mode-locking regions are narrow regions
that lie within $\cO \!\left( \frac{1}{k^2} \right)$ of the curve
\begin{equation}
r = \frac{1}{k} \Gamma(\theta),
\label{eq:nearbyCurve}
\end{equation}
where $\Gamma$ is defined by
\begin{equation}
\Gamma(\theta) \defeq \begin{cases}
\frac{\ln \left( \cos(\theta) \right) - \ln \left( \sin(\theta) \right)}
{\cos(\theta) - \sin(\theta)} \;, & \theta \in \left( 0, \frac{\pi}{2} \right)
\setminus \left\{ \frac{\pi}{4} \right\},\\
\sqrt{2} \;, & \theta = \frac{\pi}{4} \;,
\end{cases}
\label{eq:Gamma}
\end{equation}
for $\theta \in \left( 0, \frac{\pi}{2} \right)$, and defined by
$\Gamma(\theta) = \Gamma \!\left( \theta {\rm ~mod~} \frac{\pi}{2} \right)$
for all other non-integer multiples of $\frac{\pi}{2}$.
%\begin{equation}
%\Gamma(\theta) = \Gamma \!\left( \theta {\rm ~mod~} \frac{\pi}{2} \right) \;.
%\label{eq:GammaExtended}
%\end{equation}

The existence of shrinking points on the $\cG^\pm_k$-mode-locking regions for large values of $k$
is determined by scalar quantities $\kappa^\pm_{\Delta \ell}$
that are associated with the $\cS$-shrinking point.
Formulas for these are given in Appendix \ref{app:formulas}.
If $\rho_{\rm max} < 1$ and certain admissibility conditions are satisfied,
then given $\Delta \ell \in \mathbb{Z}$ and a choice of $+$ or $-$,
$\cG^\pm_k$-mode-locking regions have $\cG^\pm[k,\Delta \ell]$-shrinking points
for all sufficiently large values of $k$ if $\kappa^\pm_{\Delta \ell} > 0$,
see Theorem 2.2 of \cite{Si17c}.
If instead $\kappa^\pm_{\Delta \ell} < 0$, then the mode-locking regions
have short stability-loss boundaries where the $\cG^\pm[k,\Delta \ell]$-cycle has a stability multiplier of $-1$.

To leading order, the $\theta$-values of the $\cG^\pm[k,\Delta \ell]$-shrinking points
and stability loss boundaries are independent of $k$.
We denote the leading order components by $\theta^\pm_{\Delta \ell}$ and provide formulas for them in Appendix \ref{app:formulas}.

%---------------------------------------------------------------------
\subsection{Nearly-hyperbolic annular sectors}
\label{sub:sectors}

Let $\Delta \ell \in \mathbb{Z}$ and a choice of $+$ or $-$ be given.
Suppose $\kappa^\pm_{\Delta \ell}$ and $\kappa^\pm_{\Delta \ell - 1}$ are non-zero (as is generically the case)
with which $\theta^\pm_{\Delta \ell}$ and $\theta^\pm_{\Delta \ell - 1}$ are well-defined.
Assuming admissibility, for any sufficiently large value of $k \in \mathbb{Z}$
the inner boundary of the $\cG^\pm_{k}$-mode-locking region, call it $\gamma_k$,
is $C^K$ between endpoints, call them $p_k$ and $q_k$,
at $\theta = \theta^\pm_{\Delta \ell} + \cO \!\left( \frac{1}{k} \right)$ and
$\theta = \theta^\pm_{\Delta \ell - 1} + \cO \!\left( \frac{1}{k} \right)$.
If $\kappa^\pm_{\Delta \ell} > 0$, then $p_k$ is a $\cG^\pm[k,\Delta \ell]$-shrinking point,
while if $\kappa^\pm_{\Delta \ell} < 0$, then at $p_k$ there is
an $\cG^\pm[k,\Delta \ell]$-cycle with a stability multiplier of $-1$.
Similarly if $\kappa^\pm_{\Delta \ell - 1} > 0$, then $q_k$ is a $\cG^\pm[k,\Delta \ell - 1]$-shrinking point,
while if $\kappa^\pm_{\Delta \ell - 1} < 0$, then at $q_k$ there is
an $\cG^\pm[k,\Delta \ell - 1]$-cycle with a stability multiplier of $-1$.

We define $\Sigma^\pm_{k,\Delta \ell}$ as the region bounded by the curves
$\gamma_k$ and $\gamma_{k+1}$, the line segment connecting $p_k$ and $p_{k+1}$,
and the line segment connecting $q_k$ and $q_{k+1}$.
The boundaries $\gamma_k$ and $\gamma_{k+1}$ are given by (\ref{eq:nearbyCurve}), which resembles a hyperbola,
while $\theta$ is approximately constant along the two line segments.
In this sense $\Sigma^\pm_{k,\Delta \ell}$ is a ``nearly-hyperbolic annular sector''\removableFootnote{
We define $\Sigma^\pm_{k,\Delta \ell}$ using the shrinking points with $\Delta \ell$ and $\Delta \ell-1$
not using the shrinking points with $\Delta \ell$ and $\Delta \ell+1$ for the following reason.

Below we focus on $\cG^+[k,\Delta \ell]$-cycles with $\Delta \ell \ge 0$.
In order to use a centre manifold to study $\cG^+[k,\Delta \ell]$-cycles with $\Delta \ell \ge 0$,
it is desirable to study $\cG^+[k,\Delta \ell]$ and $\cG^+[k,\Delta \ell]^{\overline{0}}$-cycles.
Since $\cG^+[k,\Delta \ell]^{\overline{0}}$ is a permutation of $\cG^+[k,\Delta \ell-1]$,
it is the shrinking points with $\Delta \ell$ and $\Delta \ell-1$ that are relevant.

In order to study other cases in the same fashion it is necessary to look at permutations
of $\cG^\pm[k,\Delta \ell]$, for which when we flip the first symbol
the resulting symbol sequence is instead a permutation of $\cG^\pm[k,\Delta \ell+1]$ in some cases.
Nevertheless, as done below, we can state the results in all cases
using the shrinking points with $\Delta \ell$ and $\Delta \ell-1$.
}\removableFootnote{
To deal with other cases we have to work with certain permutations of the $\cG^{\pm}[k,\Delta \ell]$
as specified in the following definition.

%.....................................................................
\begin{definition}
For all $k \in \mathbb{Z}^+$ and $|\Delta \ell| < k$, define
\begin{align}
\cH^+[k,\Delta \ell] = \begin{cases}
\cG^+[k,\Delta \ell]^{\left( \tilde{\ell} d_k^+ \right)} \;, & \Delta \ell = -k+1,\ldots,-1 \\
\cG^+[k,\Delta \ell] \;, & \Delta \ell = 0,\ldots,k-1
\end{cases} \;, \\
\cH^-[k,\Delta \ell] = \begin{cases}
\cG^-[k,\Delta \ell]^{\left( -d_k^- \right)} \;, & \Delta \ell = -k+1,\ldots,0 \\
\cG^-[k,\Delta \ell]^{\left( \left( \tilde{\ell}-1 \right) d_k^- \right)} \;, & \Delta \ell = 1,\ldots,k-1
\end{cases} \;.
\end{align}
\end{definition}

As indicated by the following two lemmas,
the above formulation is useful because each $\cH^{\pm}[k,\Delta \ell]$
can be written in terms of $\cS$ and its partitions with an expression that
{\em ends} in a power involving $k$,
and each $\cH^{\pm}[k,\Delta \ell]^{\overline{0}}$ is a permutation of either
$\cG^{\pm}[k,\Delta \ell-1]$ or $\cG^{\pm}[k,\Delta \ell+1]$.

%.....................................................................
\begin{lemma}
For all $k \in \mathbb{Z}^+$ and $|\Delta \ell| < k$,
\begin{align}
\cH^+[k,\Delta \ell] = \begin{cases}
\left( \cY \cX^{\overline{0}} \right)^{-\Delta \ell-1} \check{\cX}
\left( \cS^{((\ell-1)d)} \right)^{k+\Delta \ell+1} \;, & \Delta \ell = -k+1,\ldots,-1 \\
\left( \cX \cY^{\overline{0}} \right)^{\Delta \ell} \hat{\cX}
\left( \cS^{(-d)} \right)^{k-\Delta \ell} \;, & \Delta \ell = 0,\ldots,k-1
\end{cases} \;, \\
\cH^-[k,\Delta \ell] = \begin{cases}
\left( \cX^{\overline{0}} \cY \right)^{-\Delta \ell} \hat{\cY}
\cS^{k+\Delta \ell} \;, & \Delta \ell = -k+1,\ldots,0 \\
\left( \cY^{\overline{0}} \cX \right)^{\Delta \ell-1} \check{\cY}
\left( \cS^{(\ell d)} \right)^{k-\Delta \ell+1} \;, & \Delta \ell = 1,\ldots,k-1
\end{cases} \;.
\end{align}
\end{lemma}

%.....................................................................
\begin{lemma}
For all $k \in \mathbb{Z}^+$,
\begin{enumerate}
\setlength{\itemsep}{0pt}
\item
$\cH^+[k,\Delta \ell]^{\overline{0}} = \cG^+[k,\Delta \ell+1]^{\left( \tilde{\ell} d_k^+ \right)}$,
for all $\Delta \ell = -k+1,\ldots,-1$,
\item
$\cH^+[k,\Delta \ell]^{\overline{0}} = \cG^+[k,\Delta \ell-1]^{\left( -d_k^+ \right)}$,
for all $\Delta \ell = 0,\ldots,k-1$,
\item
$\cH^-[k,\Delta \ell]^{\overline{0}} = \cG^-[k,\Delta \ell+1]$,
for all $\Delta \ell = -k+1,\ldots,0$,
\item
$\cH^-[k,\Delta \ell]^{\overline{0}} = \cG^-[k,\Delta \ell-1]^{\left( \left( \tilde{\ell}-1 \right) d_k^- \right)}$,
for all $\Delta \ell = 1,\ldots,k-1$.
\end{enumerate}
\label{le:cHall0bar}
\end{lemma}
}.

Equations (\ref{eq:nearbyCurve}) and (\ref{eq:thetaPlus})--(\ref{eq:thetaMinus}) provide
expressions for the boundaries of $\Sigma^\pm_{k,\Delta \ell}$ 
with an $\cO \!\left( \frac{1}{k^2} \right)$ error.
Fig.~\ref{fig:sectorsApprox50} illustrates the leading order
approximation for the example of Fig.~\ref{fig:modeLockApprox50}.

%For any sufficiently large value of $k$, we define a sector $\Sigma^\pm_{k,\Delta \ell}$
%as the region bounded by the following four curves:
%\begin{enumerate}
%\setlength{\itemsep}{0pt}
%\item
%the inner boundary of the $\cG^\pm_k$-mode-locking region,
%\item
%the inner boundary of the $\cG^\pm_{k+1}$-mode-locking region,
%\item
%the line through the $\cG^\pm[k,\Delta \ell]$ and $\cG^\pm[k,\Delta \ell-1]$-shrinking points,
%\item
%the line through the $\cG^\pm[k+1,\Delta \ell]$ and $\cG^\pm[k+1,\Delta \ell-1]$-shrinking points.
%\end{enumerate}
%Curves (i) and (ii) are given by (\ref{eq:nearbyCurve}), which resembles a hyperbola,
%while $\theta$ is approximately constant along curves (iii) and (iv).
%Each $\Sigma^\pm_{k,\Delta \ell}$ has two boundaries on which the value of $\theta$ is approximately constant
%and two boundaries given by (\ref{eq:nearbyCurve}), which resembles a hyperbola.

%As mentioned in \S\ref{sec:intro}, in any sector $\Sigma^\pm_{k,\Delta \ell}$
%we work with a coordinates specific to the sector.
For a given value of $\Delta \ell$ and a choice of $+$ or $-$,
let $r_k(\theta)$ denote the radial coordinate of $\gamma_k$
(i.e.~the inner boundary of the $\cG^\pm_k$-mode-locking region) for each $k$.
Extending (\ref{eq:nearbyCurve}),
we have $r_k(\theta) = \frac{1}{k} \Gamma(\theta) + \frac{1}{k^2} \Gamma_2(\theta) +
\cO \!\left( \frac{1}{k^3} \right)$, for a $C^K$ function $\Gamma_2$.
The difference in $r$-values between the inner and outer boundaries of $\Sigma^\pm_{k,\Delta \ell}$
is then\removableFootnote{
In an earlier write-up my proof of this result required about two pages!
}
\begin{equation}
r_k(\theta) - r_{k+1}(\theta) = \frac{1}{k^2} \Gamma(\theta) + \cO \!\left( \frac{1}{k^3} \right),
\label{eq:outerInnerDifference}
\end{equation}
because $\cO \!\left( \frac{1}{k^2} \right)$-terms involving $\Gamma_2$ vanish.

We then define
\begin{equation}
\delta \defeq \frac{1}{k} \left( 1 - \frac{r}{r_k(\theta)} \right),
\label{eq:delta}
\end{equation}
with which each point in $\Sigma^\pm_{k,\Delta \ell}$ is uniquely represented by the pair $(\delta,\theta)$.
We have $\delta = 0$ at any point on the outer boundary of $\Sigma^\pm_{k,\Delta \ell}$ and
by (\ref{eq:outerInnerDifference}) we have
$\delta = \frac{1}{k^2} + \cO \!\left( \frac{1}{k^3} \right)$
at any point on the inner boundary of $\Sigma^\pm_{k,\Delta \ell}$.
Let
\begin{equation}
\theta_{\rm min} \defeq \min \!\left( \theta^{\pm}_{\Delta \ell} ,\, \theta^{\pm}_{\Delta \ell-1} \right), \qquad
\theta_{\rm max} \defeq \max \!\left( \theta^{\pm}_{\Delta \ell} ,\, \theta^{\pm}_{\Delta \ell-1} \right).
\label{eq:thetaMinMax}
\end{equation}
On the two linear boundaries of $\Sigma^\pm_{k,\Delta \ell}$, we have
$\theta = \theta_{\rm min} + \cO \!\left( \frac{1}{k} \right)$ and
$\theta = \theta_{\rm max} + \cO \!\left( \frac{1}{k} \right)$.

The particular definition (\ref{eq:delta}) is useful because,
to leading order, $\Sigma^\pm_{k,\Delta \ell}$ is a rectangle in
$(\delta,\theta)$-coordinates as given by (\ref{eq:SigmaApprox}).
We include the factor $\frac{1}{k}$ in (\ref{eq:delta}) so that
an $\cO(1)$ change in the value of $\delta$ corresponds to an $\cO(1)$ change in the value of $r$.

%---------------------------------------------------------------------
\subsection{The approximation by skew sawtooth maps}
\label{sub:skewSawtoothApprox}

Throughout each sector $\Sigma^\pm_{k,\Delta \ell}$ we can approximate the dynamics of (\ref{eq:f})
by using the skew sawtooth map (\ref{eq:g}) with appropriate formulas for 
$a_L$, $a_R$ and $w$ in terms of $\delta$ and $\theta$.
The value of $w$ is independent of $\theta$ and given by
\begin{equation}
w = k^2 \delta \;.
\end{equation}
The values of $a_L$ and $a_R$ are independent of $\delta$ and given in Table \ref{tb:aLaR}.
Note that we do not consider sectors with $\kappa^\pm_{\Delta \ell} < 0$ and $\kappa^\pm_{\Delta \ell - 1} < 0$
as such sectors do not seem to involve admissible $\cG^\pm[k,\Delta \ell]$-cycles\removableFootnote{
I don't think I need to mention the following.
Suppose $\kappa^\pm_{\Delta \ell} \kappa^\pm_{\Delta \ell - 1} < 0$.
By Theorem 2.3 of \cite{Si17c}, if $a < 0$ then $\kappa^\pm_{\Delta \ell-1} > 0$,
while if $a > 0$ then $\kappa^\pm_{\Delta \ell} > 0$.
}.

%%%%%%%%%%%%%%%%%%%%%%%%%%%%%%%%%%%%%%%%%%%%%%%%%%%%%%%%%%%%%
\begin{table}[b!]
\renewcommand{\arraystretch}{1.5}
\begin{center}
\begin{tabular}{c|c|c|}
& $a_L$ & $a_R$ \\
\hline
for $\Sigma^+_{k,\Delta \ell}$ with $\kappa^+_{\Delta \ell} > 0$, $\kappa^+_{\Delta \ell - 1} > 0$ &
$\frac{\tan(\theta)}{\tan(\theta_{\rm min})}$ &
$\frac{\tan(\theta)}{\tan(\theta_{\rm max})}$ \\[1mm]
\hline
for $\Sigma^+_{k,\Delta \ell}$ with $\kappa^+_{\Delta \ell} \kappa^+_{\Delta \ell - 1} < 0$ &
$-\frac{\tan(\theta)}{\tan(\theta_{\rm min})}$ &
$\frac{\tan(\theta)}{\tan(\theta_{\rm max})}$ \\[1mm]
\hline
for $\Sigma^-_{k,\Delta \ell}$ with $\kappa^-_{\Delta \ell} > 0$, $\kappa^-_{\Delta \ell - 1} > 0$ &
$\frac{\tan(\theta_{\rm max})}{\tan(\theta)}$ &
$\frac{\tan(\theta_{\rm min})}{\tan(\theta)}$ \\[1mm]
\hline
for $\Sigma^-_{k,\Delta \ell}$ with $\kappa^-_{\Delta \ell} \kappa^-_{\Delta \ell - 1} < 0$ &
$-\frac{\tan(\theta_{\rm max})}{\tan(\theta)}$ &
$\frac{\tan(\theta_{\rm min})}{\tan(\theta)}$ \\[1mm]
\hline
\end{tabular}
\caption{
Formulas for $a_L$ and $a_R$ inside a sector $\Sigma^\pm_{k,\Delta \ell}$.
\label{tb:aLaR}
}
\end{center}
\end{table}
%%%%%%%%%%%%%%%%%%%%%%%%%%%%%%%%%%%%%%%%%%%%%%%%%%%%%%%%%%%%%

As indicated in Table \ref{tb:aLaR}, the formulas for $a_L$ and $a_R$ take slightly different forms in different cases,
but in each case $\frac{a_R}{a_L}$ is constant throughout $\Sigma^\pm_{k,\Delta \ell}$.
Therefore the value of $\frac{a_R}{a_L}$ characterises the two-dimensional slice
of parameter space of (\ref{eq:g}) to which $\Sigma^\pm_{k,\Delta \ell}$ corresponds.
In terms of $a_L$, $a_R$, and $w$,
in the approximation of (\ref{eq:SigmaApprox}) the four parts of the boundary of $\Sigma^\pm_{k,\Delta \ell}$ are
given by $w=0$, $w=1$, $a_R=1$, and either $a_L=1$ or $a_L=-1$.

One iteration of the skew sawtooth map (\ref{eq:g}) corresponds to many iterations of (\ref{eq:f}).
To be more specific, let
\begin{equation}
g_{\rm lift}(z) \defeq \begin{cases}
w + a_L \left( z_i - z_{\rm sw} \right) + z_{\rm sw} \;,
& 0 \le z_i \le z_{\rm sw} \;, \\
w + a_R \left( z_i - z_{\rm sw} \right) + z_{\rm sw} \;,
& z_{\rm sw} \le z_i < 1 \;,
\end{cases}
\label{eq:glift}
\end{equation}
be a lift of $g$.
At any point in $\Sigma^\pm_{k,\Delta \ell}$,
for all but an $\cO \!\left( \frac{1}{k} \right)$ set of $z$-values
for which (\ref{eq:f}) and (\ref{eq:g}) are misaligned,
if $0 \le z \le z_{\rm sw}$ then $g(z)$ corresponds to $f^{\cG^\pm[k+\Delta k,\Delta \ell]}$,
while if $z_{\rm sw} \le z < 1$ then $g(z)$ corresponds to $f^{\cG^\pm[k+\Delta k,\Delta \ell - 1]}$, where
\begin{equation}
\Delta k = g_{\rm lift}(z) - g(z).
\label{eq:Deltak}
\end{equation}
A formal statement of this approximation is given below by Theorem \ref{th:main}.

%---------------------------------------------------------------------
\subsection{Dynamics of skew sawtooth maps}
\label{sub:skewSawtoothDyns}

If $\frac{a_R}{a_L} > 0$, as in both plots shown in Fig.~\ref{fig:modeLockSkewSaw50},
then (\ref{eq:g}) is a homeomorphism and has a unique rotation number.
The dynamics of (\ref{eq:g}) approximates that of (\ref{eq:f}),
not only in $\Sigma^+_{2,0}$ and $\Sigma^+_{3,1}$,
but also in $\Sigma^+_{k,0}$ and $\Sigma^+_{k,1}$ for all $k \ge 1$, see Fig.~\ref{fig:modeLockApprox50}.

For instance in the large grey region of the left plot of Fig.~\ref{fig:modeLockSkewSaw50},
the map (\ref{eq:g}) has a stable fixed point with $0 \le z \le z_{\rm sw}$ and $g_{\rm lift}(z) - g(z) = 1$.
Therefore, in the corresponding region of each $\Sigma^+_{k,0}$,
the map (\ref{eq:f}) has a stable $\cG^+[k+1,0]$-cycle
(this belongs to the $\cG^+_{k+1}$-mode-locking-region).

As a further illustration, consider the mode-locking region of the left plot of
Fig.~\ref{fig:modeLockSkewSaw50} with rotation number $\frac{1}{2}$.
This region has two components connected by a shrinking point.
The map (\ref{eq:g}) has a stable period-two solution $\{ z_0, z_1 \}$
with $g_{\rm lift}(z_0) - g(z_0) = 0$, $g_{\rm lift}(z_1) - g(z_1) = 1$, $0 \le z_0 \le z_{\rm sw}$, and
$0 \le z_1 \le z_{\rm sw}$ in the upper component, and a stable period-two solution $\{ z_0, z_1 \}$
with the same properties except $z_{\rm sw} \le z_1 < 1$ in the lower component.
Therefore, in the upper [resp.~lower] component of each $\Sigma^+_{k,0}$,
the map (\ref{eq:f}) has a stable $\cG^+[k,0] \cG^+[k+1,0]$-cycle [resp.~$\cG^+[k,0] \cG^+[k+1,-1]$-cycle].
These belong to a mode-locking region with rotation number $\frac{m^+_k + m^+_{k+1}}{n^+_k + n^+_{k+1}}$.
In this way (\ref{eq:g}) can be used to describe all mode-locking regions of (\ref{eq:f}) near a shrinking point.

%%%%%%%%%%%%%%%%%%%%%%%%%%%%%%%%%%%%%%%%%%%%%%%%%%%%%%%%%%%%%
\begin{figure}[t!]
\begin{center}
\setlength{\unitlength}{1cm}
\begin{picture}(14.1,9.5)
\put(0,.3){\includegraphics[height=9cm]{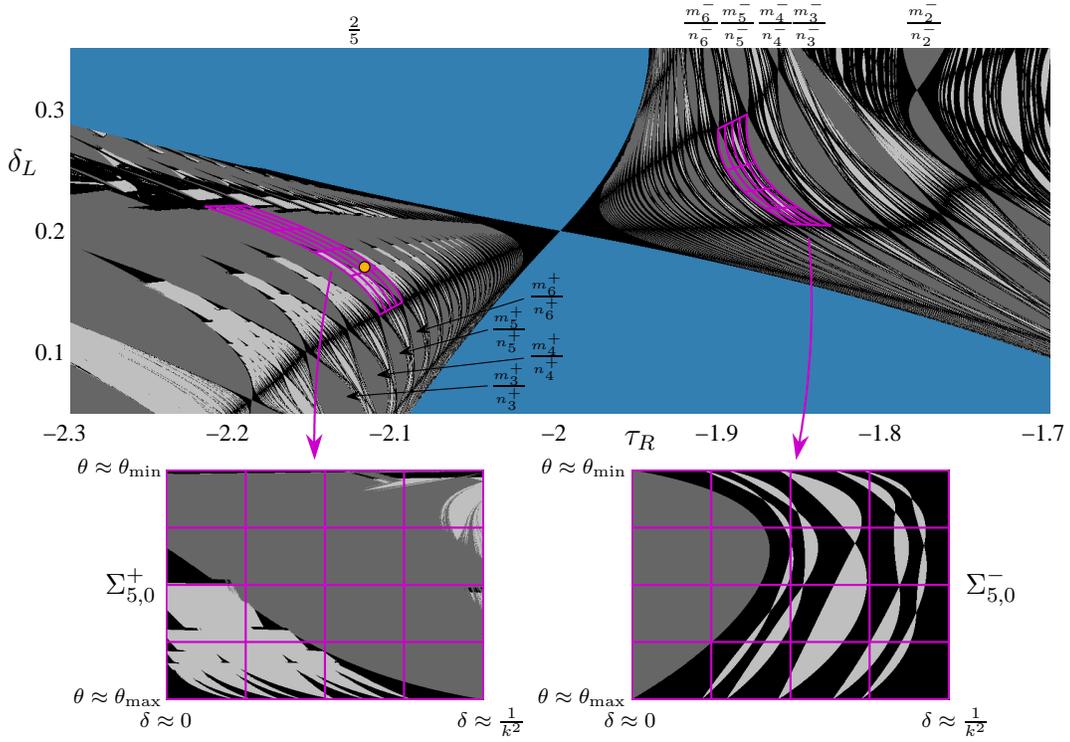}}
\put(8.2,3.76){$\tau_R$}
\put(0,7.4){$\delta_L$}
\put(4.5,9.2){\scriptsize $\frac{2}{5}$}
\put(11.93,9.32){\tiny $\frac{m_2^-}{n_2^-}$}
\put(10.4,9.32){\tiny $\frac{m_3^-}{n_3^-}$}
\put(9.95,9.32){\tiny $\frac{m_4^-}{n_4^-}$}
\put(9.45,9.32){\tiny $\frac{m_5^-}{n_5^-}$}
\put(8.97,9.32){\tiny $\frac{m_6^-}{n_6^-}$}
\put(6.42,4.48){\tiny $\frac{m_3^+}{n_3^+}$}
\put(6.92,4.88){\tiny $\frac{m_4^+}{n_4^+}$}
\put(6.42,5.24){\tiny $\frac{m_5^+}{n_5^+}$}
\put(6.92,5.68){\tiny $\frac{m_6^+}{n_6^+}$}
\put(1.76,0){\scriptsize $\delta \approx 0$}
\put(5.98,0){\scriptsize $\delta \approx \frac{1}{k^2}$}
\put(.88,.32){\scriptsize $\theta \approx \theta_{\rm max}$}
\put(.93,3.38){\scriptsize $\theta \approx \theta_{\rm min}$}
\put(1.3,1.78){\small $\Sigma^+_{5,0}$}
\put(12.74,1.78){\small $\Sigma^-_{5,0}$}
\put(7.93,0){\scriptsize $\delta \approx 0$}
\put(12.15,0){\scriptsize $\delta \approx \frac{1}{k^2}$}
\put(7.05,.32){\scriptsize $\theta \approx \theta_{\rm max}$}
\put(7.1,3.38){\scriptsize $\theta \approx \theta_{\rm min}$}
\end{picture}
\caption{
A magnification of Fig.~\ref{fig:modeLockEx20}
centred about the $\cS$-shrinking point with $\cS = \cF[2,2,5] = LRRLR$
computed in the same way as Fig.~\ref{fig:modeLockApprox50}.
Here $\frac{m}{n} = \frac{2}{5}$ and the Farey roots are
$\frac{m^-}{n^-} = \frac{1}{3}$ and $\frac{m^+}{n^+} = \frac{1}{2}$.
For $\Sigma^+_{5,0}$, $\theta_{\rm min} = \theta^+_0 \approx 4.9786$
and $\theta_{\rm max} = \theta^+_{-1} \approx 6.0014$.
For $\Sigma^-_{5,0}$, $\theta_{\rm min} = \theta^-_0 \approx 2.2254$
and $\theta_{\rm max} = \theta^-_1 \approx 3.0811$.
The dot inside $\Sigma^+_{5,0}$ indicates the parameter values of Fig.~\ref{fig:iteratePhi}-B.
\label{fig:modeLockApprox20}
}
\end{center}
\end{figure}
%%%%%%%%%%%%%%%%%%%%%%%%%%%%%%%%%%%%%%%%%%%%%%%%%%%%%%%%%%%%%

Fig.~\ref{fig:modeLockApprox20} shows a second example
using the $\cS$-shrinking point of Fig.~\ref{fig:modeLockEx20} with $\cS = LRRLR$.
The sectors $\Sigma^+_{5,0}$ and $\Sigma^-_{5,0}$ are highlighted
and Fig.~\ref{fig:modeLockSkewSaw20} shows the mode-locking regions of (\ref{eq:g}) corresponding to these sectors.

As expected, the dynamics of (\ref{eq:g}) shown in Fig.~\ref{fig:modeLockSkewSaw20}
matches well to the dynamics of (\ref{eq:f}) in $\Sigma^+_{5,0}$ and $\Sigma^-_{5,0}$
shown in Fig.~\ref{fig:modeLockApprox20}.
For each $\Sigma^-_{k,0}$, we have $\frac{a_R}{a_L} = 21 \frac{1}{2}$\removableFootnote{
For the $\cF[2,2,5]$-shrinking point,
$a = -\frac{41}{5}$,
$b = \frac{123}{125}$, and, as given by equation (3.2) of \cite{Si17c},
$\kappa^+_{-1} = \frac{38}{55}$,
$\kappa^+_0 = -\frac{5}{11}$,
$\kappa^-_0 = \frac{43}{55}$, and
$\kappa^-_1 = \frac{10}{33}$.
By Table \ref{tb:aLaR}, the formulas given in Appendix \ref{app:formulas}, and (\ref{eq:fourtIdentity}), for $\Sigma^+_{5,0}$,
\begin{equation}
\frac{a_R}{a_L} = -\frac{\tan \!\left( \theta^+_0 \right)}{\tan \!\left( \theta^+_{-1} \right)}
= -\frac{a \kappa^+_{-1}}{b \kappa^+_0}
= -\frac{38}{3}
= -12 {\textstyle \frac{2}{3}} \;,
\nonumber
\end{equation}
and for $\Sigma^-_{5,0}$,
\begin{equation}
\frac{a_R}{a_L} = -\frac{\tan \!\left( \theta^-_0 \right)}{\tan \!\left( \theta^-_1 \right)}
= -\frac{a \kappa^-_0}{b \kappa^-_1}
= \frac{43}{2}
= 21 {\textstyle \frac{1}{2}} \;.
\nonumber
\end{equation}
},
and so the right plot of Fig.~\ref{fig:modeLockSkewSaw20} is similar to
the left plot of Fig.~\ref{fig:modeLockSkewSaw50}
(for which $\frac{a_R}{a_L} \approx 14.7884$).
For each $\Sigma^+_{k,0}$, we have $\frac{a_R}{a_L} = -12 \frac{2}{3}$.
Here $\frac{a_R}{a_L} < 0$ because $\kappa^+_0 < 0$
(thus $\cG^+[k,0]$-shrinking points do not exist)
and $\kappa^+_{-1} > 0$, see Table \ref{tb:aLaR}.
With $\frac{a_R}{a_L} < 0$, the map (\ref{eq:g}) is not invertible
and exhibits a fundamentally different bifurcation structure.
There are no shrinking points.
Mode-locking regions emanate from $\theta = \theta_{\rm max}$ and terminate at critical values of $\theta$ where
the corresponding periodic solution attains a stability multiplier of $-1$.
Beyond these critical values the dynamics of (\ref{eq:g}) can be chaotic
and multiple attractors exist for some parameter values \cite{CaGa96}.

%%%%%%%%%%%%%%%%%%%%%%%%%%%%%%%%%%%%%%%%%%%%%%%%%%%%%%%%%%%%%
\begin{figure}[t!]
\begin{center}
\setlength{\unitlength}{1cm}
\begin{picture}(14.9,5.7)
\put(0,0){\includegraphics[height=5.4cm]{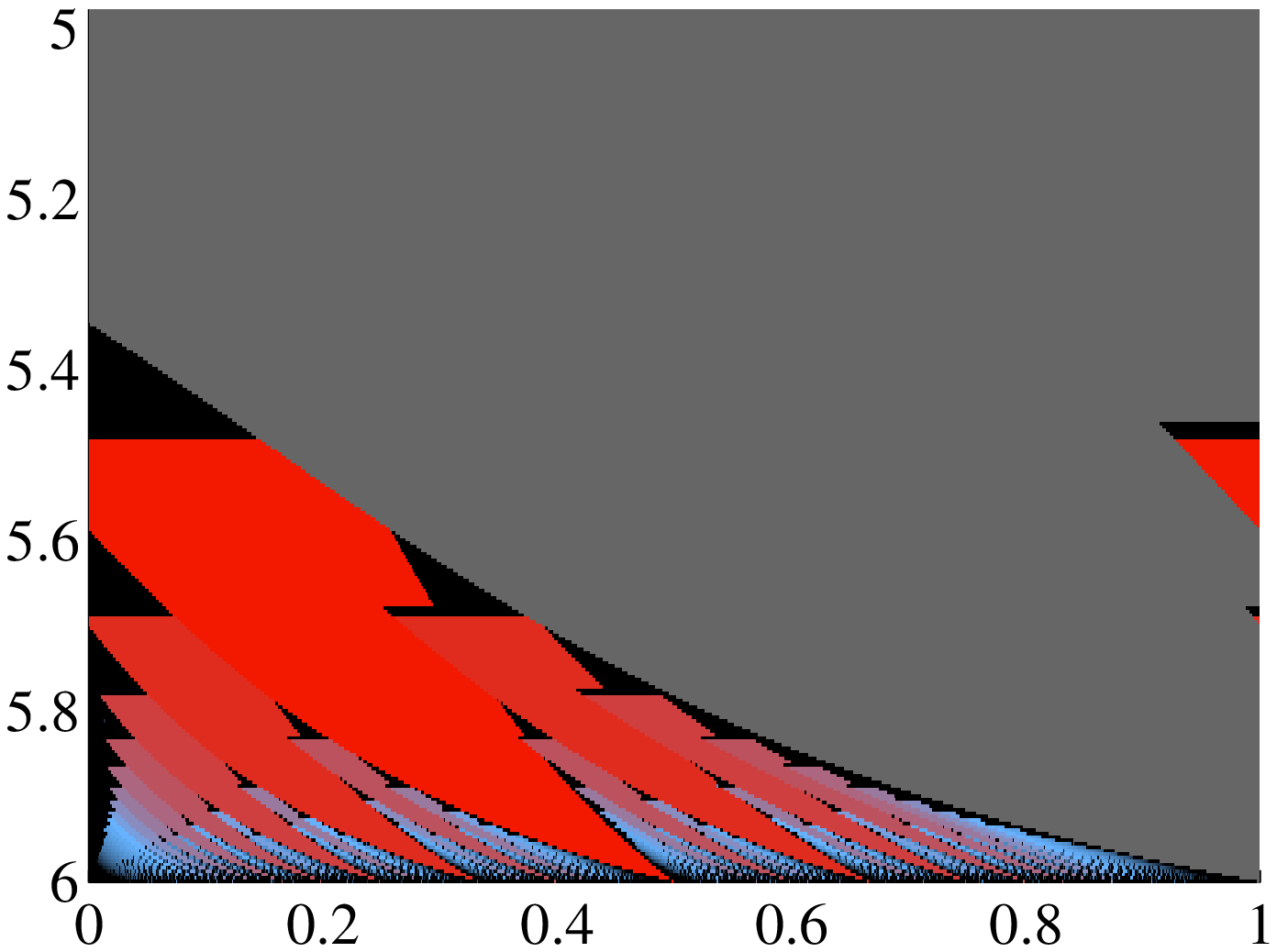}}
\put(7.7,0){\includegraphics[height=5.4cm]{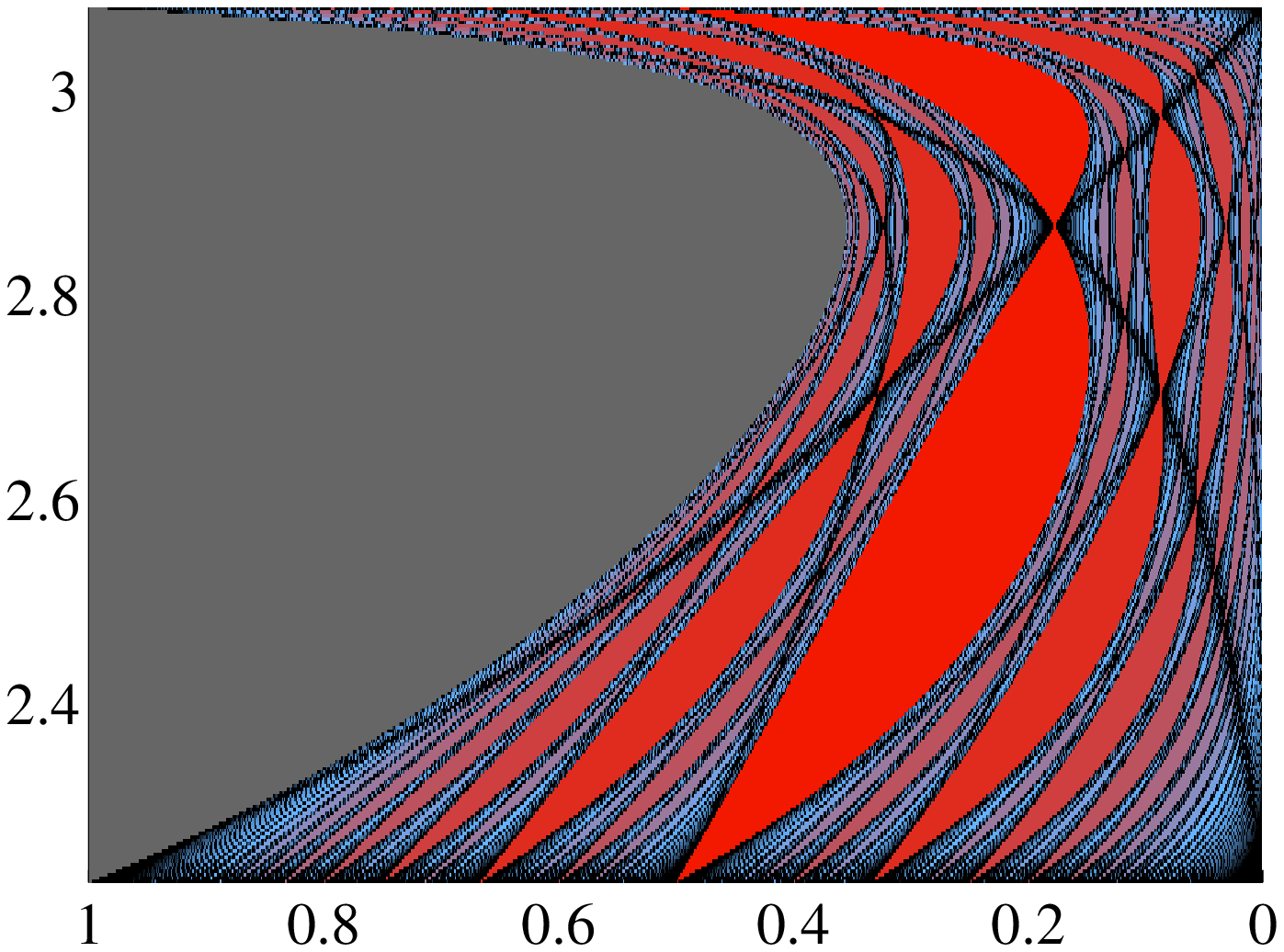}}
\put(3.6,0){$k^2 \delta$}
\put(0,3.7){$\theta$}
\put(11.3,0){$k^2 \delta$}
\put(7.7,3.7){$\theta$}
\put(3.5,5.58){$\Sigma^+_{k,0}$}
\put(11.2,5.58){$\Sigma^-_{k,0}$}
\end{picture}
\caption{
Mode-locking regions of (\ref{eq:g}) with (\ref{eq:aLaRw})
corresponding to $\Sigma^+_{k,0}$ and $\Sigma^-_{k,0}$ of Fig.~\ref{fig:modeLockApprox20}, where $k \in \mathbb{Z}^+$.
For $\Sigma^+_{k,0}$, $\frac{a_R}{a_L} = -12 \frac{2}{3}$.
For $\Sigma^-_{k,0}$, $\frac{a_R}{a_L} = 21 \frac{1}{2}$.
In both plots the axes are oriented to match Fig.~\ref{fig:modeLockApprox20}.
\label{fig:modeLockSkewSaw20}
}
\end{center}
\end{figure}
%%%%%%%%%%%%%%%%%%%%%%%%%%%%%%%%%%%%%%%%%%%%%%%%%%%%%%%%%%%%%

%=====================================================================
\section{A centre manifold and approximately recurrent dynamics}
\label{sec:approxRecurrent}
\setcounter{equation}{0}

In this section we first discuss symmetries of the centre manifolds
associated with an $\cS$-shrinking point to explain why the majority of the calculations
involve the map $f^{\cS^{(-d)}}$, \S\ref{sub:fourCentreManifolds}.
We then derive new properties of $\cG^+[k,\Delta \ell]$ in the case $\Delta \ell \ge 0$, \S\ref{sub:Gplus}.
In \S\ref{sub:centre} we describe the dynamics of $f^{\cS^{(-d)}}$ in relation to its centre manifold and
in \S\ref{sub:Omega} identify a fundamental domain of this manifold.

%---------------------------------------------------------------------
\subsection{Centre manifolds}
\label{sub:fourCentreManifolds}

%%%%%%%%%%%%%%%%%%%%%%%%%%%%%%%%%%%%%%%%%%%%%%%%%%%%%%%%%%%%%
\begin{figure}[b!]
\begin{center}
\setlength{\unitlength}{1cm}
\begin{picture}(10,5)
\put(0,0){\includegraphics[height=5cm]{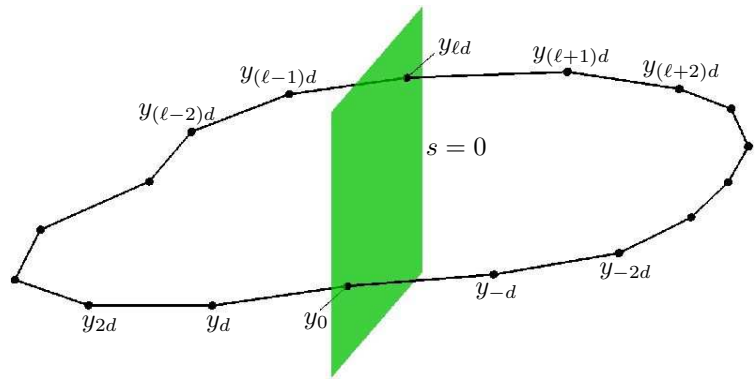}}
\put(5.52,3){\footnotesize $s=0$}
\put(3.9,.75){\footnotesize $y_0$}
\put(2.6,.72){\footnotesize $y_d$}
\put(.95,.72){\footnotesize $y_{2 d}$}
\put(1.7,3.54){\footnotesize $y_{(\ell-2) d}$}
\put(3.05,4.02){\footnotesize $y_{(\ell-1) d}$}
\put(5.68,4.4){\footnotesize $y_{\ell d}$}
\put(6.94,4.3){\footnotesize $y_{(\ell+1) d}$}
\put(8.4,4.1){\footnotesize $y_{(\ell+2) d}$}
\put(7.83,1.42){\footnotesize $y_{-2 d}$}
\put(6.18,1.15){\footnotesize $y_{-d}$}
\end{picture}
\caption{
A sketch of the uncountable collection of $\cS$-cycles of (\ref{eq:f}) at an $\cS$-shrinking point.
\label{fig:schemPolygon}
}
\end{center}
\end{figure}
%%%%%%%%%%%%%%%%%%%%%%%%%%%%%%%%%%%%%%%%%%%%%%%%%%%%%%%%%%%%%

At an $\cS$-shrinking point (\ref{eq:f}) has a unique $\cS^{\overline{0}}$-cycle, denoted $\{ y_i \}$.
There are also uncountably many $\cS$-cycles and
the union of these forms a non-planar polygon with vertices $y_i$, see Fig.~\ref{fig:schemPolygon}.

Each point on the line passing through $y_0$ and $y_d$ is a fixed point of $f^{\cS}$.
Of these, only points between $y_0$ and $y_d$ are admissible
(i.e.~are fixed points of $f^n$).
More generally, for each $j = 0,1,\ldots,n-1$, every point on the line passing through
$y_j$ and $y_{(d+j) \,{\rm mod}\, n}$ is a fixed point of $f^{\cS^{(j)}}$.
Each line is a centre manifold persisting as an extended centre manifold
for parameter values near the $\cS$-shrinking point.

As shown in \cite{Si17c}, the dynamics near an $\cS$-shrinking point can be analysed by working with
these (extended) centre manifolds for the four values $j \in \{ 0, (\ell-1)d, \ell d, -d \}$,
where from now on we omit ``${\rm mod\,} n$'' in indices for brevity.
These four values of $j$ correspond to the four line segments of the polygon that intersect the switching manifold.
Also, the four curves that bound the mode-locking region associated with the $\cS$-shrinking point
are where the $j^{\rm th}$ point of an $\cS$-cycle lies on the switching manifold for these four values of $j$.

For the remainder of this paper we work only with the $j = -d$ centre manifold.
As explained in the next section, we can use this manifold to extract information about
$\cG^+[k,\Delta \ell]$-cycles with $\Delta \ell \ge 0$.
%In brief, this is because with $\Delta \ell \ge 0$ the word $\cG^+[k,\Delta \ell]$ ends in a power of $\cS^{(-d)}$
%and $\cG^+[k,\Delta \ell]^{\overline{0}}$ is a permutation of $\cG^+[k,\Delta \ell-1]$.
Information about $\cG^+[k,\Delta \ell]$-cycles with $\Delta \ell < 0$ and $\cG^-[k,\Delta \ell]$-cycles
can be deduced by utilising two forms of symmetry associated with shrinking points.
First, we can simply swap the symbols $L$ and $R$ throughout the analysis.
More accurately, if we flip all the symbols of $\cF[\ell,m,n]$
(from $L$ to $R$ and vice-versa) we produce $\cF[n-\ell,m,n]^{(\ell d)}$.
Second, rather than viewing $\cF[\ell,m,n]$ as clockwise rotation with rotation number $\frac{m}{n}$,
we can view it as anticlockwise rotation with rotation number $\frac{n-m}{n}$.
This corresponds to the identity $\cF[\ell,n-m,n] = \cF[\ell,m,n]^{((\ell-1)d)}$.
It follows that by replacing $\ell$ and $m$ with $n-\ell$ and $n-m$, respectively, in the results below
we can generate analogous results for any $\cG^\pm[k,\Delta \ell]$-cycle.

%In a neighbourhood of an $\cS$-shrinking point the corresponding mode-locking region
%is bounded by four curves which emanate from the shrinking point.
%These curves are characterised by the existence of an $\cS$-cycle with
%some point $x^{\cS}_j$ on the switching manifold.
%By writing $\cS = \cF[\ell,m,n]$, the four curves correspond to
%$j = 0$, $(\ell-1)d$, $\ell d$ and $-d$ (taken modulo $n$ as usual) \cite{SiMe09}.

%The theory for shrinking points developed in \cite{Si17c}
%involves a certain symmetry around these four values of $j$.
%In particular, near an $\cS$-shrinking point each map $f^{\cS^{(j)}}$ has a one-dimensional centre manifold.
%By working on these centre manifolds we can deeply understand $\cG^\pm[k,\Delta \ell]$-cycles and their bifurcations.
%However, the most useful centre manifold to use depends on the choice of $\cG^\pm[k,\Delta \ell]$.
%For simplicity, for the remainder of this paper we focus on $\cG^+[k,\Delta \ell]$-cycles with $\Delta \ell \ge 0$.
%For these periodic solutions it is best to work with the centre manifold of $f^{\cS^{(-d)}}$.
%In brief, this is because when $\Delta \ell \ge 0$ the word $\cG^+[k,\Delta \ell]$ ends with powers of $\cS^{(-d)}$
%(consequently the point $x^{\cG^+[k,\Delta \ell]}_0$ lies near the centre manifold of $f^{\cS^{(-d}}$),
%and $\cG^+[k,\Delta \ell]^{\overline{0}}$ is a permutation of $\cG^+[k,\Delta \ell-1]$
%(hence if $x^{\cG^+[k,\Delta \ell]}_0$ lies on the switching manifold it belongs to a $\cG^+[k,\Delta \ell-1]$-cycle).

%---------------------------------------------------------------------
\subsection{Properties of $\cG^+[k,\Delta \ell]$ with $\Delta \ell \ge 0$}
\label{sub:Gplus}

Here we provide three symbolic results for $\cG^+[k,\Delta \ell]$ with $\Delta \ell \ge 0$
obtained in a straight-forward manner from formulas given in \cite{Si17c}.
Here it is helpful to work with words rather than sequences.
As explained in \cite{Si17c}, any periodic symbol sequence $\cS$ of minimal period $n$
is given by the infinite repetition of the primitive word $\cS_0 \cdots \cS_{n-1}$.
In this way there is a one-to-one correspondence between periodic symbol sequences and primitive words,
and so we also denote the word $\cS_0 \cdots \cS_{n-1}$ by $\cS$.

For any $\cS = \cF[\ell,m,n]$, we define the words
\begin{align}
\cX &\defeq \cS_0 \cdots \cS_{(\ell d-1) {\rm \,mod\,} n}
\;, \label{eq:X} \\
\cY &\defeq \cS_{\ell d {\rm \,mod\,} n} \cdots \cS_{n-1}
\;, \label{eq:Y} \\
\hat{\cX} &\defeq \cS_0 \cdots \cS_{(-d-1) {\rm \,mod\,} n}
\;. \label{eq:hatX}
\end{align}
We first show that the word $\cG^+[k,\Delta \ell]$ ends in powers of $\cS^{(-d)}$.

%.....................................................................
\begin{lemma}
For all $k \in \mathbb{Z}^+$ and $\Delta \ell = 0,\ldots,k-1$,
\begin{equation}
\cG^+[k,\Delta \ell] = \left( \cX \cY^{\overline{0}} \right)^{\Delta \ell}
\hat{\cX} \left( \cS^{(-d)} \right)^{k-\Delta \ell}.
\label{eq:cGplusFormula}
\end{equation}
\label{le:cGplusFormula}
\end{lemma}

%.....................................................................
\begin{proof}
By Proposition 4.8 of \cite{Si17c},
$\cG^+[k,\Delta \ell] = \left( \cX \cY^{\overline{0}} \right)^{\Delta \ell} \cS^{k-\Delta \ell} \hat{\cX}$.
This can rewritten as (\ref{eq:cGplusFormula}) because $\cS \hat{\cX} = \hat{\cX} \cS^{(-d)}$.
\end{proof}

The next result shows that if we flip the first symbol of $\cG^+[k,\Delta \ell]$
the result is a cyclic permutation of $\cG^+[k,\Delta \ell-1]$.

%.....................................................................
\begin{lemma}
For all $k \in \mathbb{Z}^+$ and $\Delta \ell = 0,\ldots,k-1$,
\begin{equation}
\cG^+[k,\Delta \ell]^{\overline{0}} = \cG^+[k,\Delta \ell-1]^{\left( -d_k^+ \right)} \;.
\label{eq:cGplus0bar}
\end{equation}
\label{le:cGplus0bar}
\end{lemma}

%.....................................................................
\begin{proof}
The result is obtained by simply applying the general identity
$\cF[\ell,m,n]^{\overline{0}} = \cF[\ell-1,m,n]^{(-d)}$ (equation (4.3) of \cite{Si17c}) to
$\cG^+[k,\Delta \ell] = \cF \left[ \ell_k^+ + \Delta \ell, m_k^+, n_k^+ \right]$.
\end{proof}

The $j = -d$ centre manifold is useful for analysing $\cG^+[k,\Delta \ell]$-cycles with $\Delta \ell \ge 0$
because, by Lemma \ref{le:cGplusFormula}, the word $\cG^+[k,\Delta \ell]$ ends in a large power of $\cS^{(-d)}$.
Thus, under certain assumptions, the fixed point of $f^{\cG^+[k,\Delta \ell]}$ lies close to the
$j = -d$ centre manifold.
Also, by Lemma \ref{le:cGplus0bar}, $\cG^+[k,\Delta \ell - 1]$-cycles can be analysed by studying
the fixed point of $f^{\cG^+[k,\Delta \ell]^{\overline{0}}}$ which lies close to the
$j = -d$ centre manifold for the same reason.

We also provide an additional result used in later proofs.

%.....................................................................
\begin{lemma}
For all $k \in \mathbb{Z}^+$ and $\Delta \ell = 0,\ldots,k-1$,
\begin{equation}
\cS^{(-d)} \cG^+[k,\Delta \ell] = 
\cG^+[k,\Delta \ell]^{\overline{0} \,\overline{\left( \ell^+_k + \Delta \ell \right) d_k^+}} \cS^{(-d)} \;.
\label{eq:cGcSmdIdentity}
\end{equation}
\label{le:cGcSmdIdentity}
\end{lemma}

%.....................................................................
\begin{proof}
Since $d_k^+ = n$ (see Lemma 4.5 of \cite{Si17c}), by (\ref{eq:cGplusFormula}),
\begin{equation}
\cG^+[k,\Delta \ell]^{(-d_k^+)} = \cS^{(-d)} \left( \cX \cY^{\overline{0}} \right)^{\Delta \ell} \hat{\cX}
\left( \cS^{(-d)} \right)^{k-\Delta \ell-1} \;.
\label{eq:cGcSmdIdentityProof1}
\end{equation}
Also, by (\ref{eq:rssMainIdentity}),
\begin{equation}
\cG^+[k,\Delta \ell]^{\overline{0} \,\overline{\left( \ell^+_k + \Delta \ell \right) d_k^+}} =
\cG^+[k,\Delta \ell]^{(-d_k^+)} \;.
\label{eq:cGcSmdIdentityProof2}
\end{equation}
Then (\ref{eq:cGcSmdIdentity}) is obtained by 
combining (\ref{eq:cGplusFormula}), (\ref{eq:cGcSmdIdentityProof1}) and (\ref{eq:cGcSmdIdentityProof2}).
\end{proof}

%---------------------------------------------------------------------
\subsection{Dynamics near the centre manifold}
\label{sub:centre}

%The $M_{\cS^{(i)}}$ share the same eigenvalues.
%Each has an eigenvalue $\lambda(\eta,\nu)$, where $\lambda(0,0) = 1$
%and $\lambda$ is a $C^K$ function of $\eta$ and $\nu$.
%As explained in \cite{Si17c}, near a shrinking point there are four one-dimensional centre manifolds.
%There is symmetry between the four centre manifolds and
%so it suffices to perform calculations on only one centre manifold.
%Since $\cG^+[k,\Delta \ell]$ ends in a power of $\cS^{(-d)}$ (\ref{eq:cGplusFormula}),
%we work with the centre manifold with $j = -d$.

The map $f^{\cS^{(-d)}}$ is affine with matrix part $M_{\cS^{(-d)}}$.
At the $\cS$-shrinking point, $(\eta,\nu) = (0,0)$,
$M_{\cS^{(-d)}}$ has a unit eigenvalue with multiplicity one.
Thus in a neighbourhood of $(\eta,\nu) = (0,0)$,
$M_{\cS^{(-d)}}$ has an eigenvalue $\lambda(\eta,\nu)$ with $\lambda(0,0) = 1$
and a $C^K$ dependency on $\eta$ and $\nu$.
Throughout this neighbourhood let $\omega_{-d}^{\sf T}$ and $\zeta_{-d}$
be the corresponding left and right eigenvectors of $M_{\cS^{(-d)}}$ normalised by
\begin{equation}
e_1^{\sf T} \zeta_{-d} = 1 \;, \qquad
\omega_{-d}^{\sf T} \zeta_{-d} = 1 \;.
\end{equation}

At points $(\eta,\nu)$ where $\lambda \ne 1$, the $\cS$-cycle, denoted $\left\{ x^{\cS}_i \right\}$, is unique.
The point $x^{\cS}_{-d}$ is a fixed point of $f^{\cS^{(-d)}}$.
As shown in \cite{Si17c} (Lemma 7.2), the quantities
\begin{equation}
x^{\rm int}_{-d} \defeq \left( I - \zeta_{-d} e_1^{\sf T} \right) x^{\cS}_{-d} \;, \qquad
s^{\rm step}_{-d} \defeq (1-\lambda) e_1^{\sf T} x^{\cS}_{-d} \;,
\label{eq:varphimdgammamd}
\end{equation}
whose purpose is explained below,
can be extended in a $C^K$ fashion to a neighbourhood of $(\eta,\nu) = (0,0)$
(i.e.~including points where $\lambda = 1$).

We define two functions $u : \mathbb{R}^N \to \mathbb{R}$ and $v : \mathbb{R} \to \mathbb{R}^N$ by
\begin{equation}
u(x) \defeq \omega_{-d}^{\sf T} \left( x - x^{\rm int}_{-d} \right), \qquad
v(h) \defeq x^{\rm int}_{-d} + h \zeta_{-d} \;.
\label{eq:uv}
\end{equation}
The line
\begin{equation}
W^c \defeq \left\{ v(h) ~\middle|~ h \in \mathbb{R} \right\},
\label{eq:centreManifold}
\end{equation}
has direction $\zeta_{-d}$ and passes through $x^{\cS}_{-d}$ whenever $\lambda \ne 1$
and so is a centre manifold of $f^{\cS^{(-d)}}$.

Any $x \in \mathbb{R}^N$ can be uniquely written as
\begin{equation}
x = x^{\rm int}_{-d} + h \zeta_{-d} + q \;,
\label{eq:xEigCoords}
\end{equation}
where $h \in \mathbb{R}$ and $q \in \mathbb{R}^N$ satisfies $\omega_{-d}^{\sf T} q = 0$
(equation (\ref{eq:xEigCoords}) represents a partial decomposition by eigenspaces).
Moreover, $h$ and $q$ are given by
\begin{equation}
h = u(x), \qquad
q = \left( I - \zeta_{-d} \omega_{-d}^{\sf T} \right) \left( x - x^{\rm int}_{-d} \right),
\label{eq:hq}
\end{equation}
and $x \in W^c$ if and only if $q = 0$.

%Then $u \circ v$ is the identity map on $\mathbb{R}$
%and $v \circ u$ takes points to the
%$W^c$-component of their eigenspace decomposition (\ref{eq:xEigCoords}).

In terms of the decomposition (\ref{eq:xEigCoords}),
we can write an arbitrary power of $f^{\cS^{(-d)}}$ as
\begin{equation}
f^{\left( \cS^{(-d)} \right)^k}(x) = x^{\rm int}_{-d} +
\left( s^{\rm step}_{-d} \sum_{j=0}^{k-1} \lambda^j + h \lambda^k \right) \zeta_{-d} +
M_{\cS^{(-d)}}^k q \;,
\label{eq:centreDyns}
\end{equation}
valid for any $x \in \mathbb{R}^N$ and any $k \in \mathbb{Z}^+$.
This is illustrated in Fig.~\ref{fig:schemOmega}.
Equation (\ref{eq:centreDyns}) has the form (\ref{eq:xEigCoords})
because the middle term is a scalar multiple of $\zeta_j$
and the last term is orthogonal to $\omega_j$.

%%%%%%%%%%%%%%%%%%%%%%%%%%%%%%%%%%%%%%%%%%%%%%%%%%%%%%%%%%%%%
\begin{figure}[b!]
\begin{center}
\setlength{\unitlength}{1cm}
\begin{picture}(15,7.5)
\put(0,0){\includegraphics[height=7.5cm]{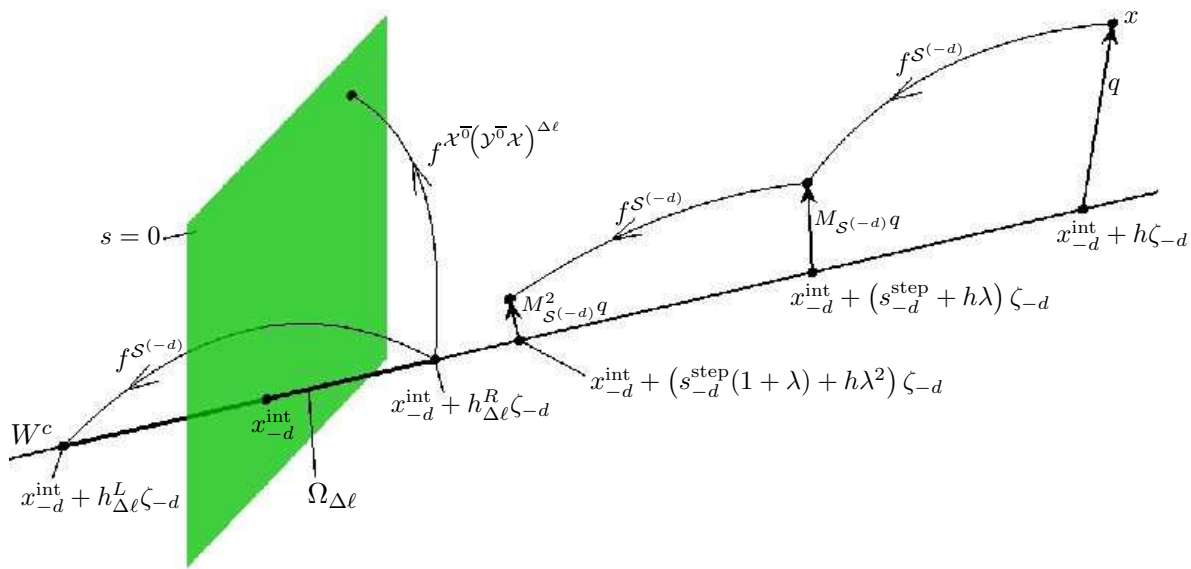}}
\put(1.2,4.37){\footnotesize $s=0$}
\put(.15,.91){\footnotesize $x^{\rm int}_{-d} + h^L_{\Delta \ell} \zeta_{-d}$}
\put(3.21,1.88){\footnotesize $x^{\rm int}_{-d}$}
\put(5.07,2.14){\footnotesize $x^{\rm int}_{-d} + h^R_{\Delta \ell} \zeta_{-d}$}
\put(7.72,2.44){\footnotesize $x^{\rm int}_{-d} + \left( s^{\rm step}_{-d} (1+\lambda) + h \lambda^2 \right) \zeta_{-d}$}
\put(10.38,3.56){\footnotesize $x^{\rm int}_{-d} + \left( s^{\rm step}_{-d} + h \lambda \right) \zeta_{-d}$}
\put(13.92,4.4){\footnotesize $x^{\rm int}_{-d} + h \zeta_{-d}$}
\put(1.4,2.74){\footnotesize $f^{\cS^{(-d)}}$}
\put(8.04,4.7){\footnotesize $f^{\cS^{(-d)}}$}
\put(11.77,6.68){\footnotesize $f^{\cS^{(-d)}}$}
\put(5.55,5.48){\footnotesize $f^{\cX^{\overline{0}} \!\left( \cY^{\overline{0}} \cX \right)^{\Delta \ell}}$}
\put(14.6,6.42){\footnotesize $q$}
\put(10.69,4.63){\scriptsize $M_{\cS^{(-d)}} q$}
\put(6.8,3.49){\scriptsize $M_{\cS^{(-d)}}^2 q$}
\put(14.81,7.32){\footnotesize $x$}
\put(.01,1.71){\small $W^c$}
\put(3.96,.93){\small $\Omega_{\Delta \ell}$}
\end{picture}
\caption{
A sketch illustrating dynamics near the centre manifold $W^c$ (\ref{eq:centreManifold}).
We show an arbitrary point $x$ and its first and second images under $f^{\cS^{(-d)}}$
and indicate the decomposition of these points as given by (\ref{eq:centreDyns}).
We also show the fundamental domain $\Omega_{\Delta \ell}$ (\ref{eq:Omega}).
\label{fig:schemOmega}
}
\end{center}
\end{figure}
%%%%%%%%%%%%%%%%%%%%%%%%%%%%%%%%%%%%%%%%%%%%%%%%%%%%%%%%%%%%%

In summary, the function $u$ returns the value of $h$ in the decomposition (\ref{eq:xEigCoords}),
while the function $v$ defines the centre manifold $W^c$.
The composition $v \circ u$ represents the projection onto $W^c$ in directions orthogonal to $\omega_{-d}$.
The point $x^{\rm int}_{-d}$ is the intersection of $W^c$ with the switching manifold,
while $s^{\rm step}_{-d} = u \!\left( f^{\cS^{(-d)}} \!\left( x^{\rm int}_{-d} \right) \right)$
is a scalar quantity used in (\ref{eq:centreDyns}) to describe the action of iteration under $f^{\cS^{(-d)}}$.

%Finally, let $\rho_{\rm max} \ge 0$ denote the maximum of the moduli
%of the eigenvalues of $M_{\cS^{(-d)}}$, excluding the unit eigenvalue, at the $\cS$-shrinking point.
If $\rho_{\rm max} < 1$, where $\rho_{\rm max}$ is defined by (\ref{eq:rhoMax}),
then forward orbits of $f^{\cS^{(-d)}}$ approach $W^c$.
Indeed, by (\ref{eq:centreDyns}) the vector $M_{\cS^{(-d)}}^k q$
is the displacement of $\left( f^{\cS^{(-d)}} \right)^k(x)$ from $W^c$
in directions orthogonal to $\omega_{-d}$.
If $\rho_{\rm max} < 1$, then $M_{\cS^{(-d)}}^k q = \cO \!\left( \rho_{\rm max}^k \right)$
for any fixed $q$ orthogonal to $\omega_{-d}$.
In view of (\ref{eq:hq}) we can therefore write
\begin{equation}
M_{\cS^{(-d)}}^k \left( I - \zeta_{-d} \omega_{-d}^{\sf T} \right) = \cO \!\left( \rho_{\rm max}^k \right).
\label{eq:fastDyns}
\end{equation}

%---------------------------------------------------------------------
\subsection{An approximately recurrent line segment, $\Omega_{\Delta \ell}$}
\label{sub:Omega}

Here we identify a useful fundamental domain $\Omega_{\Delta \ell} \subset W^c$, then provide three lemmas.
Lemma \ref{le:denominator} gives an algebraic identity,
Lemma \ref{le:hLhR} indicates the rough range of $\eta$ and $\nu$ values relevant for $\Omega_{\Delta \ell}$, and
Lemma \ref{le:tildegJApproxFixedPoint} relates $\Omega_{\Delta \ell}$
to a boundary of the $\cG^+_k$-mode-locking region.

For any $\check{x} \in W^c$, consider a half-open line segment with endpoints $\check{x}$ and $f^{\cS^{(-d)}}(\check{x})$.
Assuming $\check{x} \ne x^{\cS}_{-d}$ (in the generic case $\lambda \ne 1$ that $x^{\cS}_{-d}$ exists),
all orbits of $f^{\cS^{(-d)}}$ in $W^c$ on one side of $x^{\cS}_{-d}$ have exactly one point
in this line segment, and in this sense the line segment is a fundamental domain.

%As a fundamental domain of $W^c$ we take a half-open line segment connecting a point
%$\check{x} \in W^c$ and its image $f^{\cS^{(-d)}}(\check{x})$.
%Locally, or at least on the correct side of the fixed point $x^{\cS}_{-d}$,
%on $W^c$ all orbits under $f^{\cS^{(-d)}}$ have exactly one point in this line segment.

%The basic idea is that with $\rho_{\rm max} < 1$ and appropriate values for $\eta$ and $\nu$,
%points that lie on $W^c$ and near the switching manifold,
%return close to $W^c$ and switching manifold
%by iterating $f$ under either $\cG^+[k,\Delta \ell]$ or $\cG^+[k,\Delta \ell]^{\overline{0}}$.
%Moreover, the value of $k$ can be made unique 
%by restricting our attention to a segment of $W^c$
%defined by taking a point on the centre manifold, call it $\check{x}$,
%its image under $f^{\cS^{(-d)}}$, and all points in between (and then minus $\check{x}$).

We choose $\check{x} \in W^c$ such that its image under
$f^{\cX^{\overline{0}} \left( \cY^{\overline{0}} \cX \right)^{\Delta \ell}}$ lies on the switching manifold.
We do this because by (\ref{eq:cGplusFormula}) the first $\left( \ell^+_k + \Delta \ell \right) d_k^+$
symbols of $\cG^+[k,\Delta \ell]$ are $\left( \cX \cY^{\overline{0}} \right)^{\Delta \ell} \hat{\cX}$,
which, upon flipping the first symbol, can be rewritten as
$\cX^{\overline{0}} \left( \cY^{\overline{0}} \cX \right)^{\Delta \ell}$\removableFootnote{
Here we also need to use
\begin{equation}
\tilde{\ell} d_k^+ {\rm ~mod~} n_k^+ =
\left( \ell d {\rm ~mod~} n + \Delta \ell n \right) {\rm ~mod~} n_k^+ \;,
\label{eq:lkdkplus}
\end{equation}
given in \cite{Si17c}.
}.
Thus if there exists a $\cG^+[k,\Delta \ell]^{\overline{0}}$-cycle
for which $f^{\cX^{\overline{0}} \left( \cY^{\overline{0}} \cX \right)^{\Delta \ell}}
\!\left( x^{\cG^+[k,\Delta \ell]^{\overline{0}}}_0 \right)$ lies on the switching manifold,
then the $\cG^+[k,\Delta \ell]^{\overline{0}}$-cycle is also
a $\cG^+[k,\Delta \ell]^{\overline{0} \,\overline{\left( \ell^+_k + \Delta \ell \right) d_k^+}}$-cycle,
and by (\ref{eq:rssMainIdentity}) it is also a permutation of
a $\cG^+[k,\Delta \ell]$-cycle.

That is, we let $\check{x} = v \!\left( h^R_{\Delta \ell} \right)$,
where $h^R_{\Delta \ell} \in \mathbb{R}$ is defined by
\begin{equation}
e_1^{\sf T} f^{\cX^{\overline{0}} \left( \cY^{\overline{0}} \cX \right)^{\Delta \ell}}
\!\left( v \!\left( h^R_{\Delta \ell} \right) \right) = 0 \;.
\label{eq:hR}
\end{equation}
The image of $v \!\left( h^R_{\Delta \ell} \right)$ under $f^{\cS^{(-d)}}$ is
$v \!\left( h^L_{\Delta \ell} \right)$ where
\begin{equation}
h^L_{\Delta \ell} \defeq e_1^{\sf T} f^{\cS^{(-d)}}
\!\left( v \!\left( h^R_{\Delta \ell} \right) \right).
\label{eq:hL}
\end{equation}
Then our fundamental domain is
\begin{equation}
\Omega_{\Delta \ell} \defeq \left\{ v(h) ~\Big|~ 
h^L_{\Delta \ell} \le h < h^R_{\Delta \ell} \right\},
\label{eq:Omega}
\end{equation}
assuming $h^L_{\Delta \ell} < h^R_{\Delta \ell}$\removableFootnote{
This occurs whenever $s^{\rm step}_{-d} < 0$,
and, using a formula for $s^{\rm step}_{-d}$, we find that this is when $a \nu > 0$.
Originally I required $h^L_{\Delta \ell} \le 0$ and $h^R_{\Delta \ell} \ge 0$,
but this was problematic in my proofs below because these inequalities may not be true
very near the upper and lower boundaries of $\Sigma^+_{k,\Delta \ell}$.
}.

Before we continue we must verify that $h^R_{\Delta \ell}$ is unique and well-defined by (\ref{eq:hR}).
To do this we solve for $h^R_{\Delta \ell}$ in (\ref{eq:hR}) giving
\begin{equation}
h^R_{\Delta \ell} = -\frac{e_1^{\sf T} f^{\cX^{\overline{0}} \left( \cY^{\overline{0}} \cX \right)^{\Delta \ell}} \!\left( x^{\rm int}_{-d} \right)}
{e_1^{\sf T} M_{\cX^{\overline{0}} \left( \cY^{\overline{0}} \cX \right)^{\Delta \ell}} \zeta_{-d}} \;,
\label{eq:hLhRProof2}
\end{equation}
assuming that the denominator of (\ref{eq:hLhRProof2}) is non-zero.
Next we derive an explicit expression for the leading order component of this denominator.
Then Lemma \ref{le:hLhR} gives formulas for $h^L_{\Delta \ell}$ and $h^R_{\Delta \ell}$
and is proved in Appendix \ref{app:proofs}.
%The next result tells us that, for sufficiently small values of $\eta$ and $\nu$,
%if $\Delta \ell = 0$ then the denominator of (\ref{eq:hLhRProof2}) is non-zero,
%and if $\Delta \ell \ge 1$ and $\kappa^+_{\Delta \ell} \ne \kappa^+_{\Delta \ell-1}$
%(where these quantities are defined in Appendix ???)
%then the denominator of (\ref{eq:hLhRProof2}) is non-zero.

%For a given $\cS$-shrinking point and a given value of $\Delta \ell \ge 0$,
%we can easily evaluate (\ref{eq:denominator}).
%If it is non-zero then the denominator of (\ref{eq:hLhRProof2}) is non-zero for sufficiently small values of $\eta$ and $\nu$,

%.....................................................................
\begin{lemma}~\\
Suppose $\det(J) \ne 0$.
Then for any $\Delta \ell \ge 0$,
\begin{equation}
e_1^{\sf T} M_{\cX^{\overline{0}} \left( \cY^{\overline{0}} \cX \right)^{\Delta \ell}} \zeta_{-d} =
\begin{cases}
\frac{t_{(\ell-1)d}}{t_{-d}} + \cO(\eta,\nu), & \Delta \ell = 0 \;, \\
\frac{c t_{(\ell-1)d}}{a t_{-d}} \left( \kappa^+_{\Delta \ell} - \kappa^+_{\Delta \ell - 1} \right) + \cO(\eta,\nu),
& \Delta \ell \ge 1 \;.
\end{cases}
\label{eq:denominator}
\end{equation}
\label{le:denominator}
\end{lemma}

%.....................................................................
\begin{proof}
If $\Delta \ell = 0$, then the denominator of (\ref{eq:hLhRProof2}) is
$e_1^{\sf T} M_{\cX^{\overline{0}}} \zeta_{-d}$
and the result follows immediately from equation (6.9) of \cite{Si17c}.
If $\Delta \ell \ge 1$, then at $(\eta,\nu) = (0,0)$ we have
\begin{equation}
u_0^{\sf T} \left( I - M_{\cS^{\overline{\ell d}}} \right) =
\frac{b t_d}{c t_{(\ell+1)d}} \,e_1^{\sf T} M_{\cX} \;,
\label{eq:u0Identity}
\end{equation}
where $u_0^{\sf T}$ is the left eigenvector of $M_{\cS}$ corresponding to the unit eigenvalue $1$
(this is equation (A.34) of \cite{Si17c}).
Also
\begin{equation}
\cX^{\overline{0}} \left( \cY^{\overline{0}} \cX \right)^{\Delta \ell} =
\cX^{\overline{0}} \cY^{\overline{0}} \left( \cX \cY^{\overline{0}} \right)^{\Delta \ell-1} \cX =
\cS^{(-d)} \left( \cS^{\overline{\ell d}} \right)^{\Delta \ell-1} \cX \;,
\nonumber
\end{equation}
and so by (\ref{eq:u0Identity}) at $(\eta,\nu) = (0,0)$ we have
\begin{align}
e_1^{\sf T} M_{\cX^{\overline{0}} \left( \cY^{\overline{0}} \cX \right)^{\Delta \ell}} \zeta_{-d}
%&= e_1^{\sf T} M_{\cX} M_{\cS^{\overline{\ell d}}}^{\Delta \ell-1} M_{\cS^{(-d)}} \zeta_{-d} \nonumber \\
&= \frac{c t_{(\ell+1)d}}{b t_d} \,u_0^{\sf T} \left( I - M_{\cS^{\overline{\ell d}}} \right)
M_{\cS^{\overline{\ell d}}}^{\Delta \ell-1} \zeta_{-d} \;.
%\label{eq:hLhRProof3}
\nonumber
\end{align}
%where we have substituted $M_{\cS^{(-d)}} \zeta_{-d} = \lambda \zeta_{-d} = v_{-d} + \cO \!\left( \eta, \nu \right)$.
This yields (\ref{eq:denominator}) for $\Delta \ell \ge 1$ by using
(\ref{eq:fourtIdentity}) and the definition of $\kappa^+_{\Delta \ell}$ (\ref{eq:kappaPlus}).
\end{proof}

%.....................................................................
\begin{lemma}~\\
Suppose $\det(J) \ne 0$.
Then for any $\Delta \ell \ge 0$,
\begin{align}
h^L_{\Delta \ell} &= \begin{cases}
\eta - \frac{t_d \kappa^+_{-1}}{t_{(\ell+1)d}} \nu +
\cO \!\left( \left( \eta, \nu \right)^2 \right), & \Delta \ell = 0 \;, \\
\frac{a}{c \left( \kappa^+_{\Delta \ell} - \kappa^+_{\Delta \ell-1} \right)}
\left( \eta - \frac{t_{-d} \kappa^+_{\Delta \ell-1}}{t_{(\ell-1)d}} \nu \right) +
\cO \!\left( \left( \eta, \nu \right)^2 \right), & \Delta \ell \ge 1 \;,
\end{cases} \label{eq:hL2} \\
h^R_{\Delta \ell} &= \begin{cases}
\eta - \frac{t_{-d} \kappa^+_0}{t_{(\ell-1)d}} \nu +
\cO \!\left( \left( \eta, \nu \right)^2 \right), & \Delta \ell = 0 \;, \\
\frac{a}{c \left( \kappa^+_{\Delta \ell} - \kappa^+_{\Delta \ell-1} \right)}
\left( \eta - \frac{t_{-d} \kappa^+_{\Delta \ell}}{t_{(\ell-1)d}} \nu \right) +
\cO \!\left( \left( \eta, \nu \right)^2 \right), & \Delta \ell \ge 1 \;,
\end{cases} \label{eq:hR2}
\end{align}
assuming $\kappa^+_{\Delta \ell} \ne \kappa^+_{\Delta \ell-1}$ in the case $\Delta \ell \ge 1$.
\label{le:hLhR}
\end{lemma}

Inside the $\cG^+_k$-mode-locking region and
between the $\cG^+[k,\Delta \ell]$ and $\cG^+[k,\Delta \ell - 1]$-shrinking points (if they exist),
there are $\cG^+[k,\Delta \ell]$ and $\cG^+[k,\Delta \ell - 1]$-cycles,
one of which is attracting (as determined by the sign of $a$).
The mode-locking region is bounded by a curve where $x^{\cG^+[k,\Delta \ell]}_0$ lies on the switching manifold,
and a curve where $x^{\cG^+[k,\Delta \ell]}_{\left( \ell^+_k + \Delta \ell - 1 \right) d^+_k}$
lies on the switching manifold.
These boundaries are where $P_{\cG^+[k,\Delta \ell]^{(i)}}$ is singular
for $i = 0$ and $i = \left( \ell^+_k + \Delta \ell - 1 \right) d^+_k$\removableFootnote{
In an earlier write-up I also showed that on the curve
$\det \!\left( P_{\cG^+[k,\Delta \ell]^{\left( \left( \tilde{\ell}-1 \right) d_k^+ \right)}} \right) = 0$
(specified by Theorem 2.1 of \cite{Si17c}), we have
\begin{equation}
\tilde{g}^R_{k,\Delta \ell} \!\left( h^R_{\Delta \ell} \right) =
h^R_{\Delta \ell} + \cO \!\left( \rho_{\rm max}^k \right)\;.
\label{eq:tildegRApproxFixedPoint}
\end{equation}
}

%Below we use these ``$P$-matrices'' to characterise the boundaries of the mode-locking regions,
%rather than the location of points of a $\cG^+[k,\Delta \ell]$-cycle relative to the switching manifold,
%because the matrices are well-defined everywhere whereas $\cG^+[k,\Delta \ell]$-cycles
%are non-unique at $\cG^+[k,\Delta \ell]$-shrinking points
%and do not exist at some nearby points in parameter space.

%Here it suffices to consider boundaries
%$\det \!\left( P_{\cG^+[k,\Delta \ell]} \right) = 0$.
%On such a boundary there exists a $\cG^+[k,\Delta \ell]$-cycle
%for which $x^{\cG^+[k,\Delta \ell]}_0$ lies on the switching manifold.
%Since $\cG^+[k,\Delta \ell]$ ends in a power of $\cS^{(-d)}$ (\ref{eq:cGplusFormula}),
%the point $x^{\cG^+[k,\Delta \ell]}_0$ lies near $W^c$,
%thus $x^{\cG^+[k,\Delta \ell]}_0 \approx x^{\rm int}_{-d}$.
%Since $x^{\cG^+[k,\Delta \ell]}_0$ is a fixed point of $f^{\cG^+[k,\Delta \ell]}$,
%we have the following result which we state without proof.

%.....................................................................
\begin{lemma}
Suppose $\det(J) \ne 0$ and $\rho_{\rm max} < 1$.
Then for any $\Delta \ell \ge 0$
there exists a neighbourhood of the $\cS$-shrinking point within which
$\det \!\left( P_{\cG^+[k,\Delta \ell]} \right) = 0$ implies
\begin{equation}
f^{\cG^+[k,\Delta \ell]} \!\left( x^{\rm int}_{-d} \right) = x^{\rm int}_{-d} + \cO \!\left( \rho_{\rm max}^k \right),
\label{eq:tildegLApproxFixedPoint}
\end{equation}
where $k \in \mathbb{Z}^+$.
\label{le:tildegJApproxFixedPoint}
\end{lemma}

%.....................................................................
\begin{proof}
Suppose $\det \!\left( P_{\cG^+[k,\Delta \ell]} \right) = 0$.
Also suppose, for the moment, that the $\cG^+[k,\Delta \ell]$-cycle exists and is unique
(as is generically the case, although it may not be admissible).
Then $x^{\cG^+[k,\Delta \ell]}_0$ lies on the switching manifold.
Since $\rho_{\rm max} < 1$ and $\cG^+[k,\Delta \ell]$ ends in a power of $\cS^{(-d)}$ proportional to $k$,
$x^{\cG^+[k,\Delta \ell]}_0$ is an $\cO \!\left( \rho_{\rm max}^k \right)$ distance from $W^c$
(this is demonstrated formally in the proof of Lemma 7.7 of \cite{Si17c}).
Thus $x^{\cG^+[k,\Delta \ell]}_0 = x^{\rm int}_{-d} + \cO \!\left( \rho_{\rm max}^k \right)$,
and then (\ref{eq:tildegLApproxFixedPoint}) follows from the fact that
$x^{\cG^+[k,\Delta \ell]}_0$ is a fixed point of $f^{\cG^+[k,\Delta \ell]}$.

Equation (\ref{eq:tildegLApproxFixedPoint}) is also true if the $\cG^+[k,\Delta \ell]$-cycle
does not exist or is non-unique\linebreak							% <--------- manual line-break !!!
(i.e.~$\det \!\left( I - M_{\cG^+[k,\Delta \ell]} \right) = 0$)
because the components of (\ref{eq:tildegLApproxFixedPoint}) are smooth functions of the parameters
and $\det \!\left( I - M_{\cG^+[k,\Delta \ell]} \right) \ne 0$
on a dense subset of parameter space. % (this can be inferred from Lemma 7.8 of \cite{Si17c}).
\end{proof}

%We omit a formal proof of Lemma \ref{le:tildegJApproxFixedPoint} as it follows from the following observations.
%At a point in parameter space for which
%$\det \!\left( P_{\cG^+[k,\Delta \ell]} \right) = 0$ and there exists a unique $\cG^+[k,\Delta \ell]$-cycle,
%the point $x^{\cG^+[k,\Delta \ell]}_0$ lies on the switching manifold.
%But $\cG^+[k,\Delta \ell]$ ends in a large power of $\cS^{(-d)}$, see (\ref{eq:cGplusFormula}),
%hence $x^{\cG^+[k,\Delta \ell]}_0$ can be viewed as the image of another point of the $\cG^+[k,\Delta \ell]$-cycle
%under a large number of iterations of $f^{\cS^{(-d)}}$.
%Thus we can expect $x^{\cG^+[k,\Delta \ell]}_0$ to lie close to the centre manifold $W^c$ (\ref{eq:centreManifold}),
%which is attracting because $\rho_{\rm max} < 1$.
%Hence $x^{\cG^+[k,\Delta \ell]}_0 \approx x^{\rm int}_{-d}$ (the intersection of $W^c$ with the switching manifold).
%Since $x^{\cG^+[k,\Delta \ell]}_0$ is a fixed point of $f^{\cG^+[k,\Delta \ell]}$,
%we expect that $f^{\cG^+[k,\Delta \ell]} \!\left( x^{\rm int}_{-d} \right) \approx x^{\rm int}_{-d}$.
%Equation (\ref{eq:tildegLApproxFixedPoint}) provides the order of the error in this approximation.

%=====================================================================
\section{An attracting invariant set}
\label{sec:Lambda}
\setcounter{equation}{0}

In this section we enlarge the one-dimensional fundamental domain $\Omega_{\Delta \ell} \subset W^c$
into an $N$-dimensional set $\Phi$.
The set $\Phi$ will have the property that iterations of any $x \in \Phi$ under $f$
will return to $\Phi$ after following $\cG^+[k,\Delta \ell]$, or a similar sequence,
with the assumption $\rho_{\rm max} < 1$.
For this reason $\Phi$ is referred to as a {\em recurrent set}.

We first define $\Phi$ in \S\ref{sub:Phi}.
We then study the return dynamics on $\Phi$ in \S\ref{sub:tildeg} and \S\ref{sub:continuity}.
Lastly we use $\Phi$ to identify an attracting invariant set in \S\ref{sub:homotopy}.

%---------------------------------------------------------------------
\subsection{The construction of a recurrent set}
\label{sub:Phi}

By the definition of $\Omega_{\Delta \ell}$, see (\ref{eq:Omega}), the surface
\begin{equation}
\left\{ x \in \mathbb{R}^N ~\middle|~
e_1^{\sf T} f^{\cX^{\overline{0}} \!\left( \cY^{\overline{0}} \cX \right)^{\Delta \ell}}(x) = 0 \right\}
\label{eq:H}
\end{equation}
intersects $\Omega_{\Delta \ell}$ at its right endpoint $v \!\left( h^R_{\Delta \ell} \right)$.
%($v$ is defined by (\ref{eq:uv})).
%This surface will form one end of the region $\Phi$.
%The next result tells us that $H_{\Delta \ell}$ is $(N-1)$-dimensional
%and, in view of Lemma \ref{le:denominator}, is not parallel to $\Omega_{\Delta \ell}$.

%.....................................................................
\begin{lemma}
Suppose $\det(J) \ne 0$.
Let $\Delta \ell \ge 0$ and suppose $\kappa^+_{\Delta \ell} \ne \kappa^+_{\Delta \ell-1}$ if $\Delta \ell \ge 1$.
Then (\ref{eq:H}) is a hyperplane not parallel to $\Omega_{\Delta \ell}$.
%with normal vector $e_1^{\sf T} M_{\cX^{\overline{0}} \left( \cY^{\overline{0}} \cX \right)^{\Delta \ell}}$
\label{le:H}
\end{lemma}

%.....................................................................
\begin{proof}
By (\ref{eq:fS2}), $e_1^{\sf T} f^{\cX^{\overline{0}} \left( \cY^{\overline{0}} \cX \right)^{\Delta \ell}}(x) =
e_1^{\sf T} M_{\cX^{\overline{0}} \left( \cY^{\overline{0}} \cX \right)^{\Delta \ell}} \,x +
e_1^{\sf T} P_{\cX^{\overline{0}} \left( \cY^{\overline{0}} \cX \right)^{\Delta \ell}} \,B$.
This is the equation for a hyperplane with normal vector 
$e_1^{\sf T} M_{\cX^{\overline{0}} \left( \cY^{\overline{0}} \cX \right)^{\Delta \ell}}$.
By Lemma \ref{le:denominator},
$e_1^{\sf T} M_{\cX^{\overline{0}} \left( \cY^{\overline{0}} \cX \right)^{\Delta \ell}} \zeta_{-d} \ne 0$.
Thus (\ref{eq:H}) is a hyperplane (because the normal vector is nonzero)
and is not parallel to $\Omega_{\Delta \ell}$ (because $\Omega_{\Delta \ell}$
has direction $\zeta_{-d}$)\removableFootnote{
Originally I had a much longer proof of Lemma \ref{le:H} because I was doing things in terms of the decomposition
$x = x^{\rm int}_{-d} + h \zeta_{-d} + q$ which makes things unnecessarily more difficult,
although it did lead to the formula (\ref{eq:Hnormal3}).

By using the decomposition $x = x^{\rm int}_{-d} + h \zeta_{-d} + q$ (\ref{eq:xEigCoords}), we obtain
\begin{equation}
e_1^{\sf T} f^{\cX^{\overline{0}} \!\left( \cY^{\overline{0}} \cX \right)^{\Delta \ell}}(x) =
e_1^{\sf T} f^{\cX^{\overline{0}} \!\left( \cY^{\overline{0}} \cX \right)^{\Delta \ell}}
\!\left( x^{\rm int}_{-d} + h^R_{\Delta \ell} \zeta_{-d} \right) +
e_1^{\sf T} M_{\cX^{\overline{0}} \left( \cY^{\overline{0}} \cX \right)^{\Delta \ell}}
\left( \left( h-h^R_{\Delta \ell} \right) \zeta_{-d} + q \right) \;.
\label{eq:Hnormal1}
\end{equation}
Since $e_1^{\sf T} f^{\cX^{\overline{0}} \left( \cY^{\overline{0}} \cX \right)^{\Delta \ell}}
\!\left( x^{\rm int}_{-d} + h^R_{\Delta \ell} \zeta_{-d} \right) = 0$,
if $x \in H_{\Delta \ell}$, then by (\ref{eq:denominator})
we can solve for $h$ given $q$:
\begin{equation}
h = h^R_{\Delta \ell} - \frac{e_1^{\sf T} M_{\cX^{\overline{0}} \left( \cY^{\overline{0}} \cX \right)^{\Delta \ell}} q}
{e_1^{\sf T} M_{\cX^{\overline{0}} \left( \cY^{\overline{0}} \cX \right)^{\Delta \ell}} \zeta_{-d}} \;.
\label{eq:Hnormal2}
\end{equation}
We can therefore write
\begin{equation}
H_{\Delta \ell} = \left\{
x = x^{\rm int}_{-d} + h^R_{\Delta \ell} \zeta_{-d} +
\left( I - \frac{\zeta_{-d} e_1^{\sf T} M_{\cX^{\overline{0}} \left( \cY^{\overline{0}} \cX \right)^{\Delta \ell}}}
{e_1^{\sf T} M_{\cX^{\overline{0}} \left( \cY^{\overline{0}} \cX \right)^{\Delta \ell}} \zeta_{-d}} \right)
q ~\middle|~
\omega_{-d}^{\sf T} q = 0 \right\} \;.
\label{eq:Hnormal3}
\end{equation}
The matrix coefficient of $q$ in (\ref{eq:Hnormal3}) has rank $N-1$,
and $e_1^{\sf T} M_{\cX^{\overline{0}} \left( \cY^{\overline{0}} \cX \right)^{\Delta \ell}}$
is a left eigenvector for this matrix corresponding to the zero eigenvalue.
}.
\end{proof}

Given $k \in \mathbb{Z}^+$, $|\Delta \ell| < k$ and $Q > 0$, let
\begin{align}
H_{k,\Delta \ell,Q} \defeq
\bigg\{ x \in {\rm range} \!\left( f^{\cS^{(-d)}} \right) ~\bigg|~
e_1^{\sf T} f^{\cX^{\overline{0}} \left( \cY^{\overline{0}} \cX \right)^{\Delta \ell}}(x) = 0, \nonumber \\
\left\| \left( I - \zeta_{-d} \omega_{-d}^{\sf T} \right)
\left( x - x^{\rm int}_{-d} \right) \right\| \le Q \rho_{\rm max}^{k} \bigg\},
\label{eq:Hlocal}
\end{align}
be a subset of (\ref{eq:H}) that contains $v \!\left( h^R_{\Delta \ell} \right)$.
If we write any $x \in H_{k,\Delta \ell,Q}$
in the form $x = x^{\rm int}_{-d} + h \zeta_{-d} + q$ (\ref{eq:xEigCoords}),
then $\| q \| \le Q \rho_{\rm max}^{k}$\removableFootnote{
I think that $f^{\left( \cX \cY^{\overline{0}} \right)^{\Delta \ell} \hat{\cX}} \!\left( \Omega_{\Delta \ell} \right)$
actually lies within $\cO \!\left( \frac{1}{k} \right)$ of $W^c$
(rather than an $\cO(1)$ distance away)
thus $f^{\cG^+[k,\Delta \ell]} \!\left( \Omega_{\Delta \ell} \right)$ is actually within
$\cO \!\left( \frac{\rho_{\rm max}^{k}}{k} \right)$ of $W^c$.
The same is true for $\Phi$,
which explains why my two example values of $Q$ are rather small.
We don't need to discuss such complications (which seems difficult)
as we have shown that such a $Q$ exists, which is sufficient.
}.

If $0$ is an eigenvalue of $M_{\cS^{(-d)}}$,
then the range of $f^{\cS^{(-d)}}$ has dimension less than $N$
and the condition $H_{k,\Delta \ell,Q} \subset {\rm range} \!\left( f^{\cS^{(-d)}} \right)$ is helpful below.
It is tempting to ignore this complication, as it is a special case
within the space of maps of the form (\ref{eq:f}),
however it is in fact typical for $0$ to be an eigenvalue of $M_{\cS^{(-d)}}$ in the application to
grazing-sliding bifurcations \cite{DiKo02}\removableFootnote{
We want $\Phi$ in the range of $f^{\cS^{(-d)}}$
so that the left and right faces of $\Phi$ have the same dimension,
and so that $f^{\cS^{(-d)}}$ is a bijection between the faces
and so that $\Phi$ may be interpreted as a cylinder.
}.

Let
\begin{equation}
\Phi_{k,\Delta \ell,Q} \defeq \left\{ \alpha x_1 + (1-\alpha) x_2 ~\middle|~
x_1 \in H_{k,\Delta \ell,Q} ,\,
x_2 \in f^{\cS^{(-d)}} \!\left( H_{k,\Delta \ell,Q} \right) ,\,
0 \le \alpha < 1 \right\}.
\label{eq:Phi}
\end{equation}
For brevity we just write $\Phi$ when it is clear what values of
$k$, $\Delta \ell$ and $Q$ are being used.
The set $\Phi$ is the convex hull of $H_{k,\Delta \ell,Q}$ and its image under $f^{\cS^{(-d)}}$,
but not including $H_{k,\Delta \ell,Q}$\removableFootnote{
By including exactly one face in the definition of $\Phi$,
we have Lemma \ref{le:beta}.
By including the left face and not the right one,
this eventually corresponds to $0 \le z < 1$ in the skew sawtooth map.
}.
Fig.~\ref{fig:iteratePhi} shows the set $\Phi$ for two examples.

%Below we will show that with certain assumptions
%(e.g.~that the values of $k$ and $Q$ are sufficiently large),
%$\Phi$ is a recurrent set for $f$ throughout $\Sigma^+_{k,\Delta \ell}$.

%%%%%%%%%%%%%%%%%%%%%%%%%%%%%%%%%%%%%%%%%%%%%%%%%%%%%%%%%%%%%
\begin{figure}[b!]
\begin{center}
\setlength{\unitlength}{1cm}
\begin{picture}(14.5,5.25)
\put(0,0){\includegraphics[height=5.25cm]{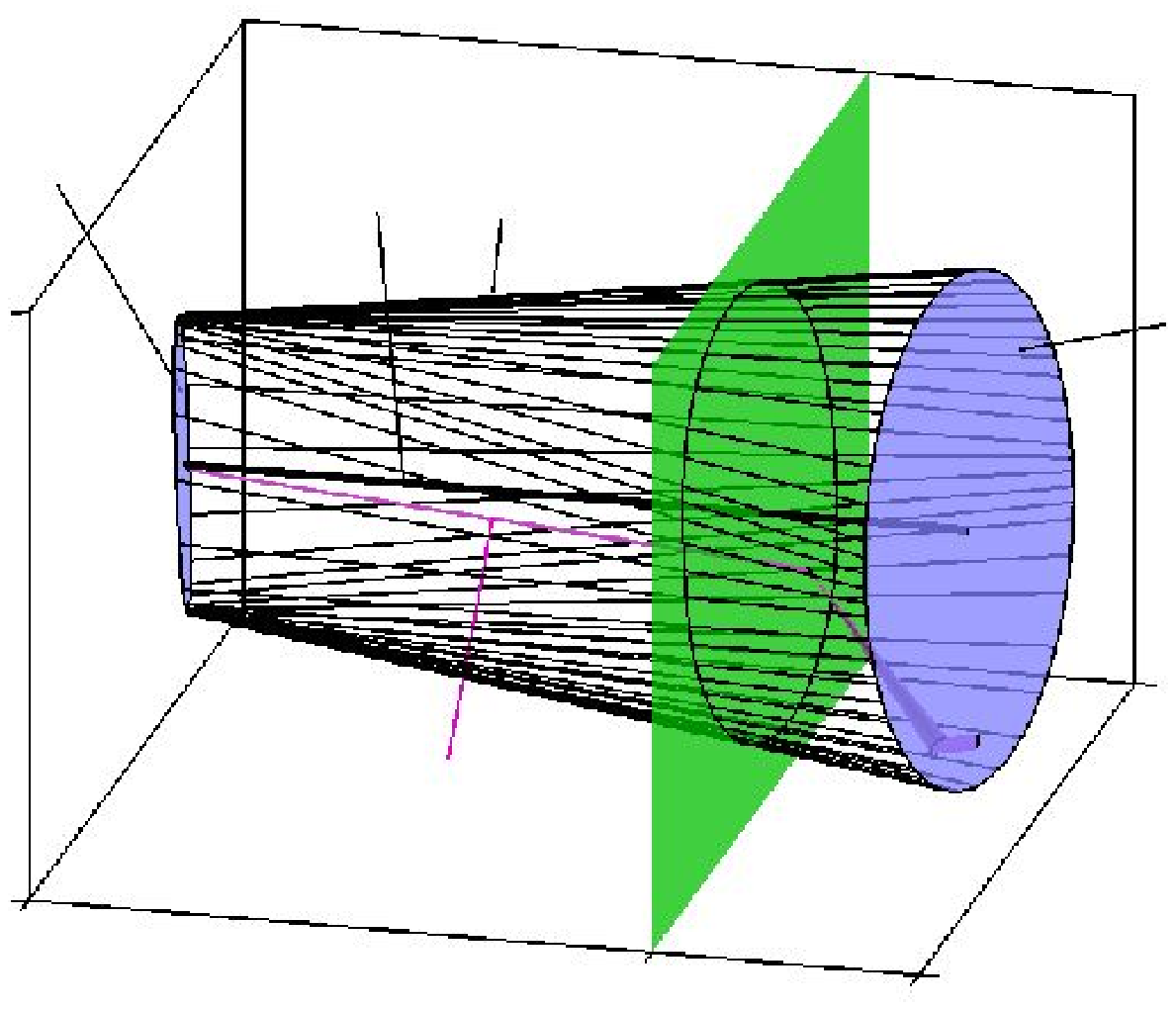}}
\put(7.5,0){\includegraphics[height=5.25cm]{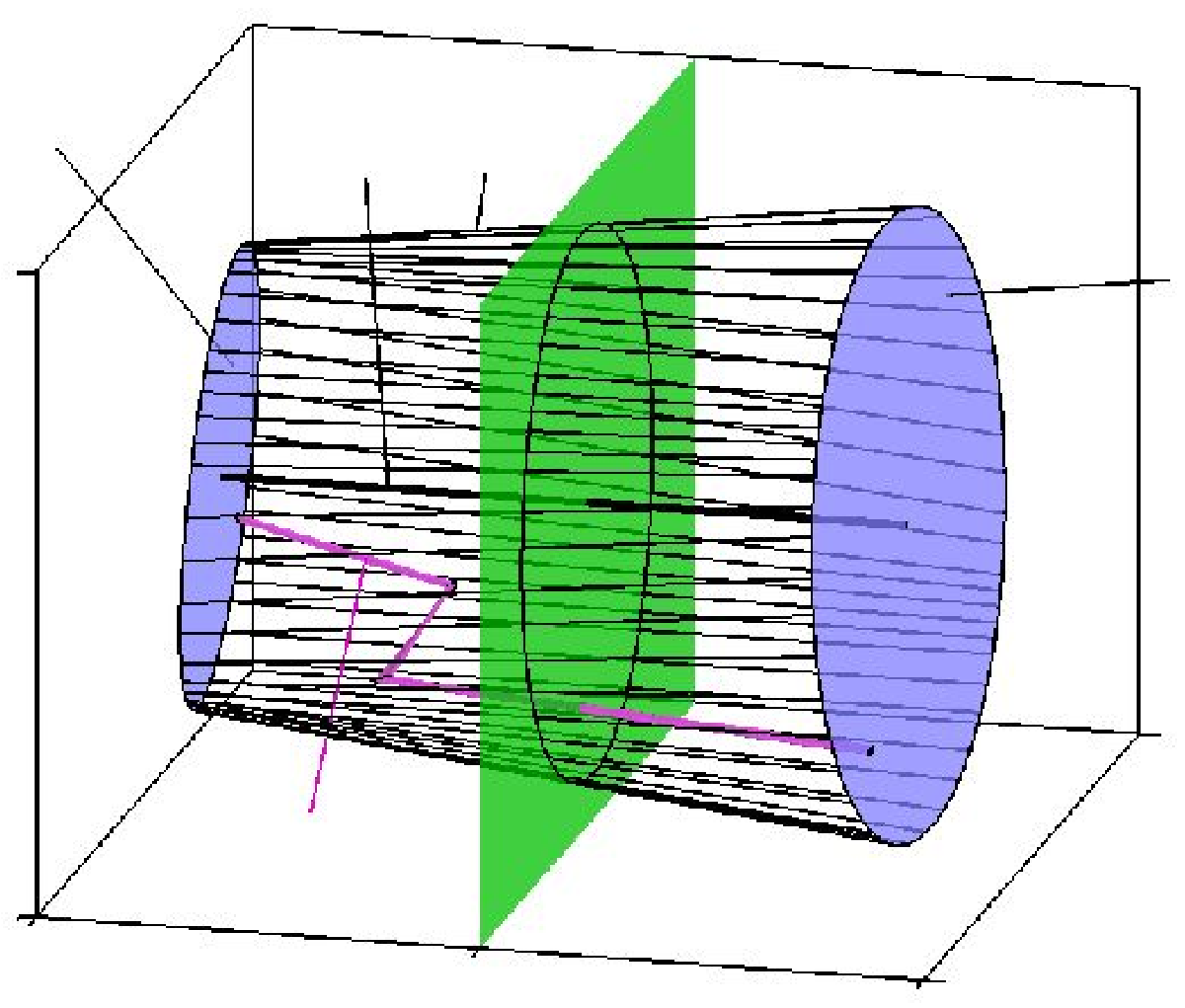}}
\put(1.2,5.1){\large \sf \bfseries A}
\put(8.7,5.1){\large \sf \bfseries B}
\put(2.6,.25){\small $s$}
\put(6,1.15){\small $\check{y}$}
\put(.5,2.5){\small $\check{z}$}
\put(.46,.56){\scriptsize $-0.7$}
\put(3.73,.29){\scriptsize $0$}
\put(4.95,.2){\scriptsize $0.3$}
\put(5.31,.4){\scriptsize $-0.006$}
\put(6.38,1.78){\scriptsize $0.006$}
\put(-.14,.79){\scriptsize $-0.006$}
\put(.06,3.58){\scriptsize $0.006$}
\put(6.38,3.54){\scriptsize $H$}
\put(-.02,4.43){\scriptsize $f^{\cS^{(-d)}}(H)$}
\put(2.41,4.25){\scriptsize $\Omega_{\Delta \ell}$}
\put(3.07,4.22){\scriptsize $\Phi$}
\put(2.62,1.27){\scriptsize $F(\Phi)$}
\put(10.17,.12){\small $s$}
\put(13.52,1.01){\small $\check{y}$}
\put(8,2.5){\small $\check{z}$}
\put(7.9,.49){\scriptsize $-0.2$}
\put(10.38,.34){\scriptsize $0$}
\put(12.45,.19){\scriptsize $0.2$}
\put(12.83,.39){\scriptsize $-0.005$}
\put(13.86,1.57){\scriptsize $0.005$}
\put(7.33,.74){\scriptsize $-0.005$}
\put(7.55,3.79){\scriptsize $0.005$}
\put(13.89,3.75){\scriptsize $H$}
\put(7.49,4.59){\scriptsize $f^{\cS^{(-d)}}(H)$}
\put(9.84,4.44){\scriptsize $\Omega_{\Delta \ell}$}
\put(10.47,4.44){\scriptsize $\Phi$}
\put(9.44,1.03){\scriptsize $F(\Phi)$}
\end{picture}
\caption{
\label{fig:iteratePhi}
The recurrent set $\Phi = \Phi_{k,\Delta \ell,Q}$ (\ref{eq:Phi}) and its image under $F$ (\ref{eq:tildef2}).
For clarity we have used coordinates representing the displacement from the centre manifold $W^c$.
Panel A shows $\Phi_{2,0,0.07}$ for the parameter values of Fig.~\ref{fig:modeLockApprox50}
with $\tau_R = -1.45$ and $\delta_L = 0.1$.
Panel B shows $\Phi_{5,0,0.13}$ for the parameter values of Fig.~\ref{fig:modeLockApprox20}
with $\tau_R = -2.12$ and $\delta_L = 0.17$.
Note, the label $H_{k,\Delta \ell,Q}$ is abbreviated as $H$.
}
\end{center}
\end{figure}
%%%%%%%%%%%%%%%%%%%%%%%%%%%%%%%%%%%%%%%%%%%%%%%%%%%%%%%%%%%%%

%---------------------------------------------------------------------
\subsection{Return dynamics}
\label{sub:tildeg}

%Recall, $\cG^+[k,\Delta \ell]$ ends in $(k-\Delta \ell)$ instances of $\cS^{(-d)}$, (\ref{eq:cGplusFormula}).
%Therefore $f^{\cG^+[k,\Delta \ell]}(x)$ lies an $\cO \!\left( \rho_{\rm max}^{k-\Delta \ell} \right)$ distance
%from $W^c$, and assuming $k \gg \Delta \ell$
%we can say that the distance is $\cO \!\left( \rho_{\rm max}^k \right)$.

Here we show that, under certain assumptions, $\Phi$ is a recurrent set for $f$
for parameter values throughout $\Sigma^+_{k,\Delta \ell}$.
Let $x \in \mathbb{R}^N$ and suppose $\rho_{\rm max} < 1$.
Then $f^{\cG^+[k,\Delta \ell]}(x)$ is an $\cO \!\left( \rho_{\rm max}^k \right)$ distance from $W^c$
(assuming $k \gg \Delta \ell$), see \S\ref{sub:centre}.
Also
\begin{equation}
f^{\cG^+[k+\Delta k,\Delta \ell]}(x) = f^{\left( \cS^{(-d)} \right)^{\Delta k}}
\!\left( f^{\cG^+[k,\Delta \ell]}(x) \right),
\label{eq:cGplusFormula2}
\end{equation}
for all $\Delta k \ge 0$ by (\ref{eq:cGplusFormula}).
Moreover, (\ref{eq:cGplusFormula2}) holds for all $\Delta k \ge -(k-\Delta \ell)$ where we take
$\left( f^{\cS^{(-d)}} \right)^{-1}(x)$
to be the unique inverse belonging to ${\rm range} \!\left( f^{\cS^{(-d)}} \right)$\removableFootnote{
We define $f^{\left( \cS^{(-d)} \right)^{-1}}(x)$
as the unique inverse that belongs to ${\rm range} \!\left( f^{\cS^{(-d)}} \right)$.
This inverse, call it $y$,
can be constructed from $x$ by using the Jordan normal form of $M_{\cS^{(-d)}}$.

Write $M_{\cS^{(-d)}} = P J P^{-1}$ where
$J = \left[ \begin{array}{@{}c|c@{}}
\check{J} & 0 \\ \hline 0 & 0 \end{array} \right]$,
and $\check{J}$ is non-singular.
Then any $z \in \mathbb{R}^N$ belongs to the column space of $M_{\cS^{(-d)}}$ if and only if,
in block form matching that of $J$,
the vector $P^{-1} z$ is zero in the bottom component,
i.e.~$P^{-1} z = \left[ \begin{array}{c} s_1 \\ \hline 0 \end{array} \right]$,
for some vector $s_1$ (of dimension equal to the rank of $M_{\cS^{(-d)}}$).

Given any $x \in {\rm range} \!\left( f^{\cS^{(-d)}} \right)$,
write $P^{-1}(x-x^{\rm int}_{-d}) = \left[ \begin{array}{c} s_1 \\ \hline 0 \end{array} \right]$.
Then the desired inverse is
\begin{equation}
y = x^{\rm int}_{-d} - \frac{s^{\rm step}_{-d}}{\lambda} \zeta_{-d} +
P \left[ \begin{array}{c} \check{J}^{-1} s_1 \\ \hline 0 \end{array} \right] \;,
\end{equation}
because $y \in {\rm range} \!\left( f^{\cS^{(-d)}} \right)$ and
\begin{align}
f^{\cS^{(-d)}}(y) &= f^{\cS^{(-d)}} \!\left( x^{\rm int}_{-d} - \frac{s^{\rm step}_{-d}}{\lambda} \zeta_{-d} \right) +
M_{\cS^{(-d)}} P \left[ \begin{array}{c} \check{J}^{-1} s_1 \\ \hline 0 \end{array} \right] \nonumber \\
&= x^{\rm int}_{-d} + P J \left[ \begin{array}{c} \check{J}^{-1} s_1 \\ \hline 0 \end{array} \right] \nonumber \\
&= x^{\rm int}_{-d} + P \left[ \begin{array}{c} s_1 \\ \hline 0 \end{array} \right] \nonumber \\
&= x \;.
\end{align}
}.
Therefore 
%\begin{equation}
%\ldots, f^{\cG^+[k-1,\Delta \ell]}(x), f^{\cG^+[k,\Delta \ell]}(x), f^{\cG^+[k+1,\Delta \ell]}(x), \ldots
%\nonumber
%\end{equation}
$\left\{ f^{\cG^+[k+\Delta k,\Delta \ell]}(x) \right\}_{\Delta k \gg -k}$
is a sequence of points near $W^c$ mapping from one to the next under $f^{\cS^{(-d)}}$.

%Under certain conditions, given below,
%for any $x \in \Phi$ exactly one of these points belongs to $\Phi$.
%The particular point is different for different points $x$
%and with different parameter values in a sector.
%For this reason we introduce a new symbol, $k^*$, and consider a sector $\Sigma^+_{k^*,\Delta \ell}$.
%For any $x \in \Phi$ we wish to identify the unique index $j$ for which
%$f^{\cG^+[k^*+j,\Delta \ell]}(x) \in \Phi$.

For any $y \in {\rm range} \!\left( f^{\cS^{(-d)}} \right)$,
let $\beta(y)$ denote the smallest value of $\Delta k$ for which 
\begin{equation}
e_1^{\sf T} f^{\cX^{\overline{0}} \left( \cY^{\overline{0}} \cX \right)^{\Delta \ell}}
\!\left( \left( f^{\cS^{(-d)}} \right)^{\Delta k}(y) \right) > 0 \;,
\label{eq:betaDefn}
\end{equation}
and assume $\beta(y) \gg -k$.
Also let
%To this end, given a suitable point $y \in {\rm range} \!\left( f^{\cS^{(-d)}} \right)$,
%we let $j = \beta(y) \in \mathbb{Z}$ denote
%the smallest $\cO(1)$ (with respect to $k^*$)
%index for which 
%$e_1^{\sf T} f^{\cX^{\overline{0}} \left( \cY^{\overline{0}} \cX \right)^{\Delta \ell}}
%\!\left( f^{\left( \cS^{(-d)} \right)^j}(y) \right) > 0$.
%We also let
\begin{equation}
T(y) \defeq \left( f^{\cS^{(-d)}} \right)^{\beta(y)}(y).
\label{eq:T}
\end{equation}
%The next result tells us that $\beta \!\left( f^{\cG^+[k,\Delta \ell]}(x) \right)$ is well-defined for all $x \in \Phi$.

%Moreover, $f^{\cG^+[k + \Delta k,\Delta \ell]}(x) \in \Phi$ if and only if
%$\Delta k = \beta \!\left( f^{\cG^+[k,\Delta \ell]}(x) \right)$,
%and similarly for $\cG^+[k + \Delta k,\Delta \ell]^{\overline{0}}$.
%Admissibility is not part of this result.
%\removableFootnote{
%I have chosen not to provide a separate lemma providing conditions under which $\beta$
%is well-defined, because it would be rather unwieldy needing
%essentially all the assumptions of Lemma \ref{le:beta}.
%}.

%.....................................................................
\begin{lemma}
Suppose $\det(J) \ne 0$ and $\rho_{\rm max} < 1$.
Let $\Delta \ell \ge 0$ and suppose $\Sigma^+_{k,\Delta \ell}$ is well-defined
for arbitrarily large values of $k$ with $\left( \kappa^+_{\Delta \ell} - \kappa^+_{\Delta \ell - 1} \right) a > 0$.
If $k \in \mathbb{Z}^+$ and $Q \in \mathbb{R}$ are sufficiently large,
then at any point in $\Sigma^+_{k,\Delta \ell}$ and any $x \in \Phi_{k,\Delta \ell,Q}$,
\begin{enumerate}
\setlength{\itemsep}{0pt}
\item
$\beta \!\left( f^{\cG^+[k,\Delta \ell]}(x) \right)$ is well-defined
and $f^{\cG^+[k + \Delta k,\Delta \ell]}(x) \in \Phi$ if and only if
$\Delta k = \beta \!\left( f^{\cG^+[k,\Delta \ell]}(x) \right)$, and similarly
\item
$\beta \!\left( f^{\cG^+[k,\Delta \ell]^{\overline{0}}}(x) \right)$ is well-defined
and $f^{\cG^+[k + \Delta k,\Delta \ell]^{\overline{0}}}(x) \in \Phi$ if and only if
$\Delta k = \beta \!\left( f^{\cG^+[k,\Delta \ell]^{\overline{0}}}(x) \right)$.
\end{enumerate}
\label{le:beta}
\end{lemma}

A proof of Lemma \ref{le:beta} is given in Appendix \ref{app:proofs}.
In Lemma \ref{le:beta} the assumption on the sign of
$\left( \kappa^+_{\Delta \ell} - \kappa^+_{\Delta \ell - 1} \right) a$
is provided as an alternative to a stronger condition of admissibility at $\cG^+[k,\Delta \ell]$-shrinking points
and is used to show that the left hand-side of (\ref{eq:betaDefn})
is an increasing function of $\Delta k$\removableFootnote{
Here we show that
\begin{equation}
e_1^{\sf T} M_{\cX^{\overline{0}} \left( \cY^{\overline{0}} \cX \right)^{\Delta \ell}} \zeta_{-d} < 0 \;,
\label{eq:betaProof1alt}
\end{equation}
given that $\cG^+[k,\Delta \ell]$-shrinking points exist for arbitrarily large values of $k$
in the case $\Delta \ell \ge 1$.

Let $\tilde{t}_i$ denote the first component of the $i^{\rm th}$ point
of the $\cG^+[k,\Delta \ell]^{\overline{0}}$-cycle, for each $i$.
Also let $\nu_{\cG^+[k,\Delta \ell]}$ denote the
$\nu$-value of the $\cG^+[k,\Delta \ell]$-shrinking point.
We have
\begin{equation}
\tilde{t}_{\left( \ell^+_k + \Delta \ell - 1 \right) d_k^+} =
\nu_{\cG^+[k,\Delta \ell]} \left( \kappa^+_{\Delta \ell-1} - \kappa^+_{\Delta \ell} \right) +
\cO \!\left( \frac{1}{k^2} \right),
\label{eq:betaProof2}
\end{equation}
which is equation (A.39) of \cite{Si17c}.
By (\ref{eq:denominator}) and (\ref{eq:betaProof2}) we can write
\begin{equation}
e_1^{\sf T} M_{\cX^{\overline{0}} \left( \cY^{\overline{0}} \cX \right)^{\Delta \ell}} \zeta_{-d} =
\frac{-c t_{(\ell-1)d} \tilde{t}_{\left( \ell^+_k + \Delta \ell - 1 \right) d_k^+}}
{a \nu_{\cG^+[k,\Delta \ell]} t_{-d}} +
\cO(\eta,\nu) \;.
\label{eq:betaProof3}
\end{equation}
The assumption that $\cG^+[k,\Delta \ell]$-shrinking points exist for arbitrarily large values of $k$
implies that $\tilde{t}_{\left( \ell^+_k + \Delta \ell - 1 \right) d_k^+} < 0$.
Throughout $\Sigma^+_{k,\Delta \ell}$ we have $a \nu > 0$.
Also $c > 0$ because $\rho_{\rm max} < 1$.
Therefore (\ref{eq:betaProof3}) implies (\ref{eq:betaProof1alt}) in the case $\Delta \ell \ge 1$.
}.

%---------------------------------------------------------------------
\subsection{Continuity of the return map in a cylindrical topology}
\label{sub:continuity}

Let
\begin{equation}
F(x) \defeq \begin{cases}
T \!\left( f^{\cG^+[k,\Delta \ell]}(x) \right), & s \le 0 \;, \\
T \!\left( f^{\cG^+[k,\Delta \ell]^{\overline{0}}}(x) \right), & s \ge 0 \;.
\end{cases}
\label{eq:tildef2}
\end{equation}
Equivalently
\begin{equation}
F(x) = \begin{cases}
f^{\cG^+[k + \Delta k,\Delta \ell]}(x), & s \le 0 \;, \\
f^{\cG^+[k + \Delta k,\Delta \ell]^{\overline{0}}}(x), & s \ge 0 \;,
\end{cases}
\label{eq:tildef3}
\end{equation}
where
\begin{equation}
\Delta k = \begin{cases}
\beta \!\left( f^{\cG^+[k,\Delta \ell]}(x) \right), & s \le 0 \;, \\
\beta \!\left( f^{\cG^+[k,\Delta \ell]^{\overline{0}}}(x) \right), & s \ge 0 \;.
\end{cases}
\end{equation}
By Lemma \ref{le:beta}, $F : \Phi \to \Phi$ throughout $\Sigma^+_{k,\Delta \ell}$\removableFootnote{
If $F$ is to represent admissible dynamics of $f$, then
\begin{equation}
F(x) = f^i(x) \;,
\label{eq:tildef}
\end{equation}
for the smallest $i \in \mathbb{Z}^+$ such that $f^i(x) \in \Phi$.
For the results to correspond to admissible dynamics,
we really need that (\ref{eq:tildef}) holds
for any point $x \in \Phi$ throughout $\Sigma^+_{k,\Delta \ell}$.
This is a very strong assumption,
and I think it would be difficult to weaken it
(perhaps I could argue that the centre manifolds are admissible?).
}.

Here we {\em identify} the left and right faces of $\Phi$
to create a topology in which $\Phi$ is cylindrical.
More specifically, we identify each $x^+ \in H_{k,\Delta \ell,Q}$
with $x^- = f^{\cS^{(-d)}}(x^+)$.

%.....................................................................
\begin{lemma}
The map $F : \Phi \to \Phi$ is continuous
in the cylindrical topology of $\Phi$.
\label{le:continuous}
\end{lemma}

\begin{proof}
$F$ is continuous on the switching manifold
because if $s=0$ then $f^{\cG^+[k,\Delta \ell]}(x) = f^{\cG^+[k,\Delta \ell]^{\overline{0}}}(x)$ by the continuity of $f$.
Also $F$ is continuous at any $x \in \Phi$ with $F(x) \in f^{\cS^{(-d)}}(H_{k,\Delta \ell,Q})$,
because for any small $\Delta x$ the point $F(x + \Delta x)$
will lie near either $F(x)$ (if the value of $\beta(y)$ in $T$ is unchanged)
or its preimage in $H_{k,\Delta \ell,Q}$ under $f^{\cS^{(-d)}}$ (if the value of $\beta(y)$ in $T$ decreases by $1$)
and these two points are identical in the cylindrical topology of $\Phi$.

Choose any $x^- \in f^{\cS^{(-d)}}(H_{k,\Delta \ell,Q})$,
and let $x^+$ be the preimage of $x^-$ in $H_{k,\Delta \ell,Q}$ under $f^{\cS^{(-d)}}$.
It remains to show that the points
\begin{align}
F(x^-) &= T \!\left( f^{\cG^+[k,\Delta \ell]}(x^-) \right), \label{eq:Fminus} \\
F(x^+) &= T \!\left( f^{\cG^+[k,\Delta \ell]^{\overline{0}}}(x^+) \right), \label{eq:Fplus} 
\end{align}
are identical.

Since $x^- = f^{\cS^{(-d)}}(x^+)$, by (\ref{eq:cGcSmdIdentity}) we have
\begin{equation}
f^{\cG^+[k,\Delta \ell]}(x^-) = f^{\cS^{(-d)}} \!\left(
f^{\cG^+[k,\Delta \ell]^{\overline{0} \,\overline{\left( \ell^+_k + \Delta \ell \right) d^+_k}}}(x^+) \right).
\nonumber
\end{equation}
Since $x^+ \in H_{k,\Delta \ell,Q}$, by the continuity of $f$ this is the same as
\begin{equation}
f^{\cG^+[k,\Delta \ell]}(x^-) = f^{\cS^{(-d)}} \!\left( f^{\cG^+[k,\Delta \ell]^{\overline{0}}}(x^+) \right).
\nonumber
\end{equation}
Therefore (\ref{eq:Fminus}) can be rewritten as
\begin{equation}
F(x^-) = T \!\left( f^{\cS^{(-d)}} \!\left( f^{\cG^+[k,\Delta \ell]^{\overline{0}}}(x^+) \right) \right),
\nonumber
\end{equation}
which is identical to (\ref{eq:Fplus}) because $T \circ f^{\cS^{(-d)}} = T$ by the definition of $T$.
\end{proof}

%---------------------------------------------------------------------
\subsection{An attracting invariant set and its topology}
\label{sub:homotopy}

Fig.~\ref{fig:iteratePhi} shows the set $F(\Phi)$ for two examples.
If we continue to iterate $\Phi$ under $F$ we obtain the attracting invariant set
\begin{equation}
\Lambda_{\Delta \ell} \defeq \bigcap_{i=0}^{\infty} F^i \!\left( \Phi \right).
\label{eq:Lambda}
\end{equation}

%.....................................................................
\begin{lemma}
The sets $\Lambda_{\Delta \ell}$ and $\Omega_{\Delta \ell}$ are homotopic
in the cylindrical topology of $\Phi$.
\label{le:homotopy2}
\end{lemma}

%In the cylindrical topology of $\Phi$, the fundamental domain $\Omega_{\Delta \ell}$ is a loop.
Roughly speaking this means that $\Lambda_{\Delta \ell}$ can be transformed to
$\Omega_{\Delta \ell}$ by distorting it in a continuous fashion \cite{Ar11,Br93}\removableFootnote{
Two continuous functions $F_1, F_2 : X \to Y$ are said to be {\it homotopic},
and we write $F_1 \simeq F_2$, if there exists a continuous function
$K : X \times [0,1] \to Y$ such that for all $x \in X$,
$K(0,x) = F_1(x)$ and $K(1,x) = F_2(x)$.

Two topological spaces $X$ and $Y$ are said to be {\it homotopic},
and we write $X \simeq Y$, if there exist continuous functions $F : X \to Y$ and $G : Y \to X$
such that $F \circ G$ and $G \circ F$ are both homotopic to the corresponding identity function
(i.e. $F \circ G \simeq \id_G$ and $G \circ F \simeq \id_F$).
Note, the functions $F$ and $G$ are called {\it homotopy equivalences}.

Let $A, B \subset X$ be connected subsets.
I cannot find a reference that gives a definition of what it means for
$A$ to be ``homotopic'' to $B$.
Nevertheless Chris Tuffley devised the following definition which seems
to be what we want and what Jim Meiss had in mind:

$A$ and $B$ are said to be {\it homotopic} if there exist continuous functions
$F : A \to B$ and $G : B \to A$
such that (i) $F \circ G \simeq \id_B$, (ii) $G \circ F \simeq \id_A$
(thus are $F$ and $G$ are homotopy equivalences),
(iii) $\iota_A \simeq F$, and (iv) $\iota_B \simeq G$
(in (iii) technically $F$ should be treated as a function from $A$ to $X$,
and similarly for $G$ in (iv), but I think we can just ignore this because
I think it is obvious and because I do not know of a good notation to indicate this).

The map $\iota_A$ is the {\it inclusion} of $A$ to $X$
and is a function from $A$ to $X$ defined by $\iota_A(x) = x$ for all $x \in A$
(i.e.~$\iota_A$ is identity map on $A$ except the set that is maps to has been enlarged).
Thus the condition $\iota_A \simeq F$ ensures that $A$ can be continuously transformed into $B$
through the ambient space $X$ (and similarly $\iota_B \simeq G$ ensures the reverse can be achieved).

As shown by the following proposition,
my use of this definition is greatly simplified by the following two facts.
First, I only need worry about subsets $B$ given by $f(A)$
for a continuous function $f : X \to X$ (the map $G$ will be used for $f$).
We will see that I will be able to use $f$ for $F$.
Second, I can always assume $B \subset A$
(because I am studying images of $\tilde{\Phi}_{\Delta \ell,Q,k}$
which repeatedly map to within one another).
We will see that I will be able to use the identity map for $G$,
or more accurately the map $\iota_B : B \to A$.

%.....................................................................
\begin{proposition}
Let $f : X \to X$ be continuous.
Let $A \subset X$ be connected and suppose $f(A) \subset A$.
Suppose $f|_A \simeq \iota_A$.
Then $f(A) \simeq A$.
\end{proposition}

%.....................................................................
\begin{proof}
We simply show that our definition of homotopic subsets
is satisfied using $F = f|_A$ and $G = \iota_{f(A)}$
(really I want $G$ to be the inclusion from $f(A)$ to $A$).
Certainly these $F$ and $G$ are (or can be viewed as) functions
from $A$ to $f(A)$ and from $f(A)$ to $A$ respectively.

Let $K : A \times [0,1] \to X$ be a homotopy for $f|_A$ and $\iota_A$.

Condition (i): here $F \circ G = f|_{f(A)}$.
So $K|_{f(A)}$ is a homotopy for $f|_{f(A)}$ and $\id_{f(A)}$ as required.
Condition (ii): here $G \circ F = f|_A$ and is homotopic to $\id_A$ using $K$.
Condition (iii): satisfied immediately in view of $K$.
Condition (iv): here $G = \iota_{f(A)}$ so trivially it is homotopic to $\iota_{f(A)}$.
\end{proof}

A related concept that I don't think I need is that of deformation retracts.
Given a connected subset $A \subset X$,
a continuous function $K : X \times [0,1] \to X$ is called a {\it deformation retraction}
(of $X$ onto $A$) if
(i) $K(x,0) = x$ for all $x \in X$,
(ii) $K(x,1) \in A$ for all $x \in X$,
(iii) $K(a,1) = a$ for all $a \in A$.
Moreover, if such a $K$ exists then $A$ is called a {\it deformation retract} of $X$.
Note, the restriction of $K(\cdot,1)$ to $A$ is the identity map.
This does not seem helpful to me because I want to use $G$ for my map,
and this is not the identity map.

Another related concept that I don't think I need is that of the fundamental group.
A {\it loop} on $X$ is a continuous function $f : [0,1] \to X$ such that $f(0) = f(1)$.
The definition of homotopic functions defines equivalence classes for
the collection of all loops on $X$.
Loops can be ``multiplied'' and hence the equivalence classes may be thought of
as elements of a group.
This group is known as the {\it fundamental group}, denoted $\pi_1(X)$
(for path-connected spaces $X$).
Since $\tilde{\Phi}_{\Delta \ell,Q,k}$ is like a cylinder,
loops are characterised (up to homotopy) by how many times they wind around $\tilde{\Phi}_{\Delta \ell,Q,k}$.
It follows that the fundamental group of $\tilde{\Phi}_{\Delta \ell,Q,k}$ is isomorphic to $(\mathbb{Z},+)$.
Moreover, if I could show that the number of times a loop winds around $\tilde{\Phi}_{\Delta \ell,Q,k}$
is the same as for the image of the loop under $G$ (this seems obvious),
I could say that the {\it induced homomorphism}
$G^* : \pi_1 \!\left( \tilde{\Phi}_{\Delta \ell,Q,k} \right) \to
\pi_1 \!\left( \tilde{\Phi}_{\Delta \ell,Q,k} \right)$ is the identity map.
Consequently the {\it degree} of $G$, which is defined as $G^*(1)$, is equal to $1$.
}.
Below we prove Lemma \ref{le:homotopy2} by using the following result.

%.....................................................................
\begin{lemma}
Let $\Psi \subset \Phi$ be connected.
If $F(\Psi) \subset \Psi$, then $\Psi$ and $F(\Psi)$ are homotopic
in the cylindrical topology of $\Phi$.
\label{le:homotopy}
\end{lemma}

To say that $\Psi$ and $F(\Psi)$ are {\em homotopic}
means that there exists a continuous function $K : \Phi \times [0,1] \to \Phi$
with $K(x,0) = x$ and $K(x,1) = F(x)$ for all $x \in \Psi$.
The function $K$ is said to be a {\em homotopy} between
$\iota_{\Psi}$ (the inclusion of $\Psi$) and $F|_{\Psi}$ (the restriction of $F$ to $\Psi$).

%.....................................................................
\begin{proof}[Proof of Lemma \ref{le:homotopy}]
Define $K : \Phi \times [0,1] \to \Phi$ by
\begin{equation}
K(x,\alpha) \defeq \begin{cases}
T \!\left( (1-\alpha) x + \alpha f^{\cG^+[k,\Delta \ell]}(x) \right), & s \le 0 \;, \\
T \!\left( (1-\alpha) x + \alpha f^{\cG^+[k,\Delta \ell]^{\overline{0}}}(x) \right), & s \ge 0 \;.
\end{cases}
\label{eq:K}
\end{equation}
By (\ref{eq:T}), $K(x,0) = x$ for all $x \in \Psi$.
Also, by (\ref{eq:tildef2}), $K(x,1) = F(x)$ for all $x \in \Psi$.
%Thus it remains to show that $K$ is continuous in the cylindrical topology of $\Phi$.
%By (\ref{eq:K}), this is achieved by showing that $T$ and $F$ are continuous in the topology of $\check{\Phi}$.

%The function $K$ is continuous on the switching manifold
%because if $s=0$ then $f^{\cG^+[k,\Delta \ell]}(x) = f^{\cG^+[k,\Delta \ell]^{\overline{0}}}(x)$.
%Also $K$ is continuous at any $x \in \Phi$ with $K(x,\alpha) \in f^{\cS^{(-d)}}(H_{k,\Delta \ell,Q})$,
%because for any small $\Delta x$ the point $K(x + \Delta x,\alpha)$
%will lie near either $K(x,\alpha)$ (if the value of $\beta(y)$ in $T$ is unchanged)
%or its preimage in $H_{k,\Delta \ell,Q}$ under $f^{\cS^{(-d)}}$ (if the value of $\beta(y)$ in $T$ decreases by $1$)
%and these two points are identical in the cylindrical topology of $\Phi$.

Choose any $x^- \in f^{\cS^{(-d)}}(H_{k,\Delta \ell,Q})$,
and let $x^+$ be the preimage of $x^-$ in $H_{k,\Delta \ell,Q}$ under $f^{\cS^{(-d)}}$.
We wish to show that 
\begin{align}
K(x^-,\alpha) &= T \!\left( (1-\alpha) x^- +
\alpha f^{\cG^+[k,\Delta \ell]}(x^-) \right), \label{eq:Kminus} \\
K(x^+,\alpha) &= T \!\left( (1-\alpha) x^+ +
\alpha f^{\cG^+[k,\Delta \ell]^{\overline{0}}}(x^+) \right), \label{eq:Kplus} 
\end{align}
are identical for all $\alpha \in [0,1]$.
Following the steps of the proof of Lemma \ref{le:continuous}, we obtain
\begin{equation}
K(x^-,\alpha) = T \!\left( (1-\alpha) f^{\cS^{(-d)}}(x^+) +
\alpha f^{\cS^{(-d)}} \!\left( f^{\cG^+[k,\Delta \ell]^{\overline{0}}}(x^+) \right) \right).
\nonumber
\end{equation}
Since $f^{\cS^{(-d)}}$ is affine, this is the same as
\begin{equation}
K(x^-,\alpha) = T \!\left( f^{\cS^{(-d)}} \!\left( (1-\alpha) x^+ +
\alpha f^{\cG^+[k,\Delta \ell]^{\overline{0}}}(x^+) \right) \right),
\nonumber
\end{equation}
which is identical to (\ref{eq:Kplus}) because $T \circ f^{\cS^{(-d)}} = T$.

This shows that $K$ is continuous at any $x^- \in f^{\cS^{(-d)}}(H_{k,\Delta \ell,Q})$.
As with $F$, the function $K$ is continuous on the switching manifold
and at any $x \in \Phi$ with $K(x,\alpha) \in f^{\cS^{(-d)}}(H_{k,\Delta \ell,Q})$.
Thus $K$ is continuous throughout $\Phi$ for any $\alpha \in [0,1]$.
Thus $K$ is indeed a homotopy and hence $\Psi$ and $F(\Psi)$ are homotopic.
\end{proof}

%Since $x^- = f^{\cS^{(-d)}}(x^+)$, by (\ref{eq:cGcSmdIdentity}) we have
%\begin{equation}
%f^{\cG^+[k,\Delta \ell]}(x^-) = f^{\cS^{(-d)}} \!\left(
%f^{\cG^+[k,\Delta \ell]^{\overline{0} \,\overline{\left( \ell^+_k + \Delta \ell \right) d^+_k}}}(x^+) \right).
%\nonumber
%\end{equation}
%Since $x^+ \in H_{k,\Delta \ell,Q}$, by the continuity of $f$ this is the same as
%\begin{equation}
%f^{\cG^+[k,\Delta \ell]}(x^-) = f^{\cS^{(-d)}} \!\left( f^{\cG^+[k,\Delta \ell]^{\overline{0}}}(x^+) \right).
%\nonumber
%\end{equation}
%Therefore (\ref{eq:Kminus}) can be rewritten as
%\begin{equation}
%K(x^-,\alpha) = T \!\left( (1-\alpha) f^{\cS^{(-d)}}(x^+) +
%\alpha f^{\cS^{(-d)}} \!\left( f^{\cG^+[k,\Delta \ell]^{\overline{0}}}(x^+) \right) \right).
%\nonumber
%\end{equation}
%Since $f^{\cS^{(-d)}}$ is affine, this is the same as
%\begin{equation}
%K(x^-,\alpha) = T \!\left( f^{\cS^{(-d)}} \!\left( (1-\alpha) x^+ +
%\alpha f^{\cG^+[k,\Delta \ell]^{\overline{0}}}(x^+) \right) \right),
%\nonumber
%\end{equation}
%which is identical to (\ref{eq:Kplus}) because $T \circ f^{\cS^{(-d)}} = T$ by the definition of $T$.

%.....................................................................
\begin{proof}[Proof of Lemma \ref{le:homotopy2}]
By applying Lemma \ref{le:homotopy} to $\Psi = \Phi$, $\Psi = F(\Phi)$,
$\Psi = F^2(\Phi)$, and so on, we deduce that $\Phi$ and $\Lambda_{\Delta \ell}$ are homotopic.
It is clear that $\Omega_{\Delta \ell}$ and $\Phi$ are homotopic in view of the definition of $\Phi$,
thus $\Lambda_{\Delta \ell}$ and $\Omega_{\Delta \ell}$ are homotopic.
\end{proof}

%=====================================================================
\section{The skew sawtooth map as an approximate return map}
\label{sec:return}
\setcounter{equation}{0}

%---------------------------------------------------------------------
%\subsection{The slopes of the skew sawtooth map and relevent parameter values}
\subsection{One-dimensional approximation}
\label{sub:slopes}

Recall, for any $x \in \mathbb{R}^N$ written as
$x = x^{\rm int}_{-d} + h \zeta_{-d} + q$ (\ref{eq:xEigCoords}),
the function $u(x)$ extracts the value of $h$,
and for any $h \in \mathbb{R}$ the function $v(h)$ returns the point $x^{\rm int}_{-d} + h \zeta_{-d} \in W^c$.
Thus for any $x \in W^c$ we have $v(u(x)) = x$
and for any $x$ near $W^c$ we have $v(u(x)) \approx x$.

The set $\Phi = \Phi_{k,\Delta \ell,Q}$ is an $\cO \!\left( \rho_{\rm max}^k \right)$
fattening of the fundamental domain $\Omega_{\Delta \ell} \subset W^c$.
Thus for any $x \in \Phi$ we have $v(u(x)) \approx x$
and $v(u(F(x))) \approx F(x)$ because $F(x) \in \Phi$.
In this way the $N$-dimensional map $F$ is well-approximated by the one-dimensional map taking $u(x)$ to $u(F(x))$.
This is the map
\begin{equation}
G \defeq u \circ F \circ v \;.
\label{eq:tildeg}
\end{equation}

The map $G$ is piecewise-linear but over the interval
$\left[ h^L_{\Delta \ell}, h^R_{\Delta \ell} \right)$ can involve many different linear pieces
corresponding to different values of $\Delta k$.
Next we show that, with fixed values of $\Delta \ell$ and $\theta$,
the leading order components of the slopes of these pieces only take two different values.

%.....................................................................
\begin{lemma}
Suppose $\det(J) \ne 0$ and $\rho_{\rm max} < 1$.
Let $\Delta \ell \ge 0$ and suppose $\Sigma^+_{k,\Delta \ell}$ is well-defined
for arbitrarily large values of $k$ with $\left( \kappa^+_{\Delta \ell} - \kappa^+_{\Delta \ell - 1} \right) a > 0$.
Then at any point in $\Sigma^+_{k,\Delta \ell}$ and any $h \in \mathbb{R}$,
\begin{align}
\frac{d}{d h} \,u \!\left( f^{\cG^+[k,\Delta \ell]} (v(h)) \right) &=
\frac{t_{-d} \kappa^+_{\Delta \ell} \tan(\theta)}{t_d} +
\cO \!\left( \frac{1}{k} \right), \label{eq:tildegLSlope} \\
\frac{d}{d h} \,u \!\left( f^{\cG^+[k,\Delta \ell]^{\overline{0}}} (v(h)) \right) &= \begin{cases}
\frac{t_{(\ell-1)d} \kappa^+_{-1} \tan(\theta)}{t_{(\ell+1)d}} + 
\cO \!\left( \frac{1}{k} \right), & \Delta \ell = 0 \;, \\
\frac{t_{-d} \kappa^+_{\Delta \ell-1} \tan(\theta)}{t_d} + 
\cO \!\left( \frac{1}{k} \right), & \Delta \ell \ge 1 \;,
\end{cases} \label{eq:tildegRSlope}
\end{align}
where $k \in \mathbb{Z}^+$.
\label{le:tildegSlopes}
\end{lemma}

Lemma \ref{le:tildegSlopes} is proved in Appendix \ref{app:proofs}.
Since $G$ is given by (\ref{eq:tildeg}), where $F$ is equal to
$f^{\cG^+[k + \Delta k,\Delta \ell]}$ or $f^{\cG^+[k + \Delta k,\Delta \ell]^{\overline{0}}}$
for some $\cO(1)$ value of $\Delta k$,
the slopes of $G$ are given by (\ref{eq:tildegLSlope}) and (\ref{eq:tildegRSlope}) to leading order.

%---------------------------------------------------------------------
\subsection{Circle map formulation}
\label{sub:parameters}

%Above we have indicated how return to $\Phi$ via $F$
%can be approximated with the one-dimensional map $G$,
%accurate to $\cO \!\left( \rho_{\rm max}^{k^*} \right)$.
%The map $G$ is piecewise-linear and continuous
%with several different slopes (often three).
%We now reduce the accuracy of the approximation
%(to polynomial in $\frac{1}{k^*}$)
%with which there are only two different values for the slopes.
%Also, in view of the topology of $\Lambda_{\Delta \ell}$,
%we make $\mathbb{S}^1$ the domain of the map.
%Consequently the map takes the form $g$ (\ref{eq:g}).

%The theorem below relates the parameters of (\ref{eq:g}),
%namely $a_L$, $a_R$ and $w$, to the $(\delta,\theta)$-coordinates
%of the sector $\Sigma^+_{k^*,\Delta \ell}$.
%Recall, $\delta = 0$ corresponds to the outer boundary of $\Sigma^+_{k^*,\Delta \ell}$.
%This boundary is either the curve 
%$\det \!\left( P_{\cG^+[k^*,\Delta \ell]} \right) = 0$ or
%$\det \!\left( P_{\cG^+[k^*,\Delta \ell]^{\left( \left(
%\tilde{\ell}-1 \right) d_{k^*}^+ \right)}} \right) = 0$.

%A proof is given in Appendix \ref{app:proofs}.
%For $\Sigma^-_{k,\Delta \ell}$ the opposite is true.

%The domain of $G$ is $h^L_{\Delta \ell} \le h < h^R_{\Delta \ell}$
%and the domain of $g$ is $\mathbb{S}^1$.
%So we define a coordinate change from $\left[ h^L_{\Delta \ell}, h^R_{\Delta \ell} \right)$ to $\mathbb{S}^1$ by

The function
\begin{equation}
\varphi(h) \defeq \begin{cases}
\frac{h - h^L_{\Delta \ell}}{h^R_{\Delta \ell} - h^L_{\Delta \ell}} \;, & a < 0 \;, \\
\frac{-h}{h^R_{\Delta \ell} - h^L_{\Delta \ell}} {\rm \,mod\,} 1 \;, & a > 0 \;.
\end{cases}
\label{eq:varphi}
\end{equation}
transforms the interval $\left[ h^L_{\Delta \ell}, h^R_{\Delta \ell} \right)$ into $[0,1)$ (the domain
of the skew sawtooth map (\ref{eq:g})).
We use this particular $a$-dependent transformation
because it leads to convenient $a$-independent formulas for $w$ and $\Delta k$\removableFootnote{
Our $a$-dependent definition of $F$ is particularly nice because then the formulas for $w$ and $k$ are independent of $a$.
This is helpful because to apply the results of this paper
one needs to be able to compute $w$ in order to evaluate $g$,
and one wants to compute $k$ in order to relate the dynamics of $g$ to those of $f$.
One ought to be less interested in the exact values of $x$ that $z$-values correspond to.

If we instead define $F(h) = \frac{h - h^L_{\Delta \ell}}{h^R_{\Delta \ell} - h^L_{\Delta \ell}}$ for $a > 0$, then
\begin{align}
w &= 1 - \frac{(1-a_L)(a_R-1)}{a_R-a_L} - k^2 \delta \;, \\
\Delta k &= -\left( G(z) - g(z) \right) \;.
\end{align}
This alternate correct set of formulas associated with $g$ is due to the fact that $g$
is unchanged by the transformation
\begin{align}
w &\mapsto 1 - \frac{(1-a_L)(a_R-1)}{a_R-a_L} - w \;, \\
z &\mapsto \left( z_{\rm sw} - z \right) {\rm \,mod\,} 1 \;.
\end{align}

Despite the piecewise nature of (\ref{eq:varphi}),
for either sign of $a$,
$0 \le z \le z_{\rm sw}$ corresponds to $h^L_{\Delta \ell} \le h \le 0$ and
$z_{\rm sw} \le z < 1$ corresponds to $0 \le h < h^R_{\Delta \ell}$.
The key difference created by (\ref{eq:varphi}) is that
for $a < 0$ the kink $z = z_{\rm sw}$ corresponds to $h = 0$,
whereas for $a > 0$ this $z$-value corresponds to $h = h^L_{\Delta \ell}$.
}.

%.....................................................................
\begin{theorem}
Suppose $\det(J) \ne 0$ and $\rho_{\rm max} < 1$.
Let $\Delta \ell \ge 0$ and suppose $\Sigma^+_{k,\Delta \ell}$ is well-defined
for arbitrarily large values of $k$ with $\left( \kappa^+_{\Delta \ell} - \kappa^+_{\Delta \ell - 1} \right) a > 0$.
Let $\theta_{\rm min} = \min \!\left( \theta^+_{\Delta \ell} ,\, \theta^+_{\Delta \ell-1} \right)$
and $\theta_{\rm max} = \max \!\left( \theta^+_{\Delta \ell} ,\, \theta^+_{\Delta \ell-1} \right)$.
Then at any point in $\Sigma^+_{k,\Delta \ell}$,
there exists $\Xi \subset \Phi$ with $\frac{{\rm meas}(\Xi)}{{\rm meas}(\Phi)} = \cO \!\left( \frac{1}{k} \right)$
such that throughout $\Phi \setminus \Xi$,
\begin{equation}
F = v \circ \varphi^{-1} \circ g \circ \varphi \circ u +
\cO \!\left( \frac{1}{k^2} \right),
\label{eq:F}
\end{equation}
where $g$ is given by (\ref{eq:g}) with
\begin{equation}
a_L = \begin{cases}
\frac{\tan(\theta)}{\tan \left( \theta_{\rm min} \right)} \;, &
\kappa^+_{\Delta \ell} > 0 \;, \kappa^+_{\Delta \ell - 1} > 0 \;, \\
-\frac{\tan(\theta)}{\tan \left( \theta_{\rm min} \right)} \;, &
\kappa^+_{\Delta \ell} \kappa^+_{\Delta \ell - 1} < 0 \;,
\end{cases} \qquad
a_R = \frac{\tan(\theta)}{\tan \!\left( \theta_{\rm max} \right)} \;, \qquad
w = k^2 \delta \;.
\label{eq:aLaRw2}
\end{equation}
Moreover, for any $x \in \Phi \setminus \Xi$,
\begin{equation}
F(x) = \begin{cases}
f^{\cG^+[k + \Delta k,\Delta \ell]}(x), & 0 \le z \le z_{\rm sw} \;, \\
f^{\cG^+[k + \Delta k,\Delta \ell]^{\overline{0}}}(x), & z_{\rm sw} \le z < 1 \;,
\end{cases}
\label{eq:F2}
\end{equation}
where $\Delta k = g_{\rm lift}(z) - g(z)$ and $z = \varphi(u(x))$.
\label{th:main}
\end{theorem}

Theorem \ref{th:main} explicitly indicates how $g$ (\ref{eq:g}) approximates the dynamics of $f$ (\ref{eq:f}).
Specifically, $g$ can be used to approximate the return map $F$ via the formula (\ref{eq:F}),
where the parameters of $g$ are given by (\ref{eq:aLaRw2})
and the relationship of $F$ to $f$ is given by (\ref{eq:F2}).
For $x \in \Xi$, either $u(x) \notin \Omega_{\Delta \ell}$, in which case (\ref{eq:F}) is not well-defined,
or $F(x)$ corresponds to a different sequence of iterates of $f$ than that given by (\ref{eq:F2}).
These discrepancies have not been seen to cause qualitatively different dynamics for $f$ and $g$ because
these are continuous maps and $\Xi$ constitutes a vanishingly small fraction of $\Phi$ as $k \to \infty$.

%=====================================================================
\section{Discussion}
\label{sec:conc}
\setcounter{equation}{0}

This paper continues the work of \cite{Si17c} to rigorously characterise the
dynamics of (\ref{eq:f}) near an $\cS$-shrinking point (where $\cS = \cF[\ell,m,n]$).
Relative to the distance from an $\cS$-shrinking point,
of all nearby mode-locking regions the $\cG^\pm_k$-mode-locking regions are the largest.
This is because, symbolically, they belong to the first level of complexity relative to $\cS$.
The $\cG^\pm_k$-mode-locking regions, and their shrinking points, were analysed in \cite{Si17c}.
The results of the present paper help explain the dynamics that occurs between the $\cG^\pm_k$-mode-locking regions:
higher period solutions, quasiperiodic dynamics, and chaos.

We have shown that such dynamics is captured by a skew sawtooth map (\ref{eq:g})
with three parameters: the slopes $a_L$ and $a_R$, and the vertical displacement $w$.
To relate (\ref{eq:f}) to (\ref{eq:g}) quantitatively,
we defined an array of sectors $\Sigma^\pm_{k,\Delta \ell}$ near the $\cS$-shrinking point.
Within each $\Sigma^\pm_{k,\Delta \ell}$ the dynamics of (\ref{eq:f}) is well-approximated by
that of (\ref{eq:g}) with fixed formulas for $a_L$, $a_R$ and $w$ in terms of $(\delta,\theta)$-coordinates.
%The main result is presented formally as Theorem \ref{th:main}
%for sectors $\Sigma^+_{k,\Delta \ell}$ with $\Delta \ell \ge 0$.
We identified a region $\Phi$ to which forward orbits regularly return,
and considered the first return map $F : \Phi \to \Phi$ that
equals the composition of (\ref{eq:f}) with itself many times (proportional to $k n$).
Theorem \ref{th:main} indicates precisely how (\ref{eq:g}) can be used to approximate $F$.

The results of \S\ref{sub:homotopy} describe the nature of the attracting invariant set $\Lambda_{\Delta \ell}$
on which the dynamics is well approximated by (\ref{eq:g}).
By Lemma \ref{le:homotopy2}, $\Lambda_{\Delta \ell}$ is homotopic to a loop
with a cylindrical topology on $\Phi$.
This tells us that the invariant set $\cup_{i \ge 0} f^i(\Lambda_{\Delta \ell})$ of (\ref{eq:f})
is either an invariant loop or a complicated geometric object that is homotopic to an invariant loop in $\mathbb{R}^N$.

For the map (\ref{eq:g}), the value of $w$ varies linearly from $0$ to $1$ as we move from
the outer boundary of $\Sigma^\pm_{k,\Delta \ell}$ to the inner boundary of $\Sigma^\pm_{k,\Delta \ell}$,
and so is given by $w = k^2 \delta$.
%The values of $a_L$ and $a_R$ depend only on $\theta$ and are given by Table \ref{tb:aLaR}.
%These expressions can be explained heuristically as follows.
Lemma \ref{le:tildegSlopes} essentially tells us that the slopes $a_L$ and $a_R$ are proportional to $\lambda^k$,
where $\lambda$ is the critical eigenvalue of $M_{\cS}$ (that is, $\lambda = 1$ at the $\cS$-shrinking point).
By equations (7.43)--(7.44) of \cite{Si17c},
in $\Sigma^+_{k,\Delta \ell}$ we have $\lambda^k = \tan(\theta) + \cO \!\left( \frac{1}{k} \right)$,
whereas in $\Sigma^-_{k,\Delta \ell}$ we have $\lambda^k = \frac{1}{\tan(\theta)} + \cO \!\left( \frac{1}{k} \right)$.
The slopes $a_L$ and $a_R$ are also such that $a_R = 1$
on one of the two linear boundaries of $\Sigma^\pm_{k,\Delta \ell}$
and either $a_L = 1$ or $a_L = -1$ on the other linear boundary of $\Sigma^\pm_{k,\Delta \ell}$,
see Table \ref{tb:aLaR}.

The dynamics in $\Sigma^\pm_{k,\Delta \ell}$
corresponds to a two-dimensional cross-section of the three-dimensional parameter space of (\ref{eq:g})
as characterised by the value of $\frac{a_R}{a_L}$ (which is constant throughout $\Sigma^\pm_{k,\Delta \ell}$).
Qualitatively, there are two distinct scenarios as determined by the sign of $\frac{a_R}{a_L}$.
If $\frac{a_R}{a_L} > 0$, then (\ref{eq:g}) is a homeomorphism and has a unique rotation number.
Periodic dynamics occurs within mode-locking regions that exhibit shrinking points
and non-periodic dynamics is quasiperiodic.
The value of $\frac{a_R}{a_L}$ dictates the overall
fatness of the mode-locking regions, see Fig.~\ref{fig:modeLockSkewSaw50}.
This explains the presence of the roughly horizontal strip in the left half of Fig.~\ref{fig:modeLockApprox50}
where mode-locking regions are relatively sparse.
Here $\cG^+[k,0]$ and $\cG^+[k,1]$-shrinking points are relatively close together,
hence $\theta^+_0 - \theta^+_1$ is relatively small.
For sectors between these two sequences of shrinking points we have
$\theta_{\rm min} = \theta^+_1$ and $\theta_{\rm max} = \theta^+_0$,
thus the value of $\frac{a_R}{a_L} = \frac{\tan(\theta_{\rm min})}{\tan(\theta_{\rm max})}$
is relatively small and so the mode-locking regions are relatively narrow.
It remains to quantify this observation by calculating, say, the fraction of
the parameter space of (\ref{eq:g}), with $\frac{a_R}{a_L} > 0$ fixed,
for which (\ref{eq:g}) has a rational rotation number.

If $\frac{a_R}{a_L} < 0$, then (\ref{eq:g}) is not invertible.
There may be coexisting attractors and chaotic dynamics \cite{CaGa96}.
Mode-locking regions do not have shrinking points but rather terminate
due to a loss of stability by attaining a stability multiplier of $-1$.
It remains to further understand these dynamics
to explain how curves connecting shrinking points can form a boundary for chaotic dynamics,
as identified numerically in \cite{SiMe08b}.

%\newpage
\appendix

%=====================================================================
\section{Additional proofs}
\label{app:proofs}
\setcounter{equation}{0}

\begin{proof}[Proof of Lemma \ref{le:hLhR}]
For brevity we derive (\ref{eq:hL2}) and (\ref{eq:hR2}) only for $\Delta \ell \ge 1$.
The result for $\Delta \ell = 0$ can be obtained in a similar fashion (and is a little simpler).

By (\ref{eq:centreDyns}) and (\ref{eq:hL}),
\begin{equation}
h^L_{\Delta \ell} = s^{\rm step}_{-d} + h^R_{\Delta \ell} \lambda \;.
\label{eq:hLhRProof4}
\end{equation}
Then (\ref{eq:hL2}) is obtained by combining (\ref{eq:hR2}), 
$\lambda = 1 + \cO \!\left( \eta, \nu \right)$, and equation (A.9) of \cite{Si17c}:
\begin{equation}
s^{\rm step}_{-d} = \frac{a t_{-d}}{c t_{(\ell-1)d}} \nu +
\cO \!\left( \left( \eta,\nu \right)^2 \right).
\label{eq:gammamd}
\end{equation}
It therefore just remains for us to derive (\ref{eq:hR2}).

We begin by manipulating the numerator of (\ref{eq:hLhRProof2}),
where $\left\{ x^{\cS^{\overline{\ell d}}}_i \right\}$
denotes the unique $\cS^{\overline{\ell d}}$-cycle:
\begin{align}
e_1^{\sf T} f^{\cX^{\overline{0}} \left( \cY^{\overline{0}} \cX \right)^{\Delta \ell}} \!\left( x^{\rm int}_{-d} \right)
&= e_1^{\sf T} f^{\cX \left( \cY^{\overline{0}} \cX \right)^{\Delta \ell}} \!\left( x^{\rm int}_{-d} \right) \nonumber \\
&= e_1^{\sf T} f^{\left( \cS^{\overline{\ell d}} \right)^{\Delta \ell} \cX} \!\left( x^{\rm int}_{-d} \right) \nonumber \\
&= e_1^{\sf T} f^{\left( \cS^{\overline{\ell d}} \right)^{\Delta \ell} \cX}
\!\left( x^{\cS^{\overline{\ell d}}}_0 \right) +
e_1^{\sf T} M_{\left( \cS^{\overline{\ell d}} \right)^{\Delta \ell} \cX}
\left( x^{\rm int}_{-d} - x^{\cS^{\overline{\ell d}}}_0 \right) \nonumber \\
&= e_1^{\sf T} f^{\cX}
\!\left( x^{\cS^{\overline{\ell d}}}_0 \right) +
e_1^{\sf T} M_{\cX} M_{\cS^{\overline{\ell d}}}^{\Delta \ell}
\left( x^{\rm int}_{-d} - x^{\cS^{\overline{\ell d}}}_0 \right) \nonumber \\
&= s^{\cS^{\overline{\ell d}}}_{\ell d} +
e_1^{\sf T} M_{\cX} M_{\cS^{\overline{\ell d}}}^{\Delta \ell}
\left( x^{\rm int}_{-d} - x^{\cS^{\overline{\ell d}}}_0 \right).
\label{eq:hLhRProof11}
\end{align}

In a neighbourhood of the $\cS$-shrinking point,
one component of its corresponding mode-locking region is where $\eta \ge 0$ and $\nu \ge 0$,
while the other component is where $\eta \le \psi_1(\nu)$ and $\nu \le \psi_2(\eta)$,
for some $C^K$ functions $\psi_1$ and $\psi_2$ with
$\psi_1(0) = \psi_1'(0) = \psi_2(0) = \psi_2'(0) = 0$, see Theorem 6.9 of \cite{Si17c}.

First suppose $\eta = \psi_1(\nu)$.
Then $s^{\cS^{\overline{\ell d}}}_{\ell d} = 0$
and so by (\ref{eq:hLhRProof11}) the numerator of (\ref{eq:hLhRProof2}) is
\begin{align}
e_1^{\sf T} f^{\cX^{\overline{0}} \left( \cY^{\overline{0}} \cX \right)^{\Delta \ell}} \!\left( x^{\rm int}_{-d} \right)
&= e_1^{\sf T} M_{\cX} M_{\cS^{\overline{\ell d}}}^{\Delta \ell}
\left( x^{\rm int}_{-d} - x^{\cS^{\overline{\ell d}}}_0 \right).
\label{eq:hLhRProof12}
\end{align}
By (\ref{eq:u0Identity}),
\begin{equation}
e_1^{\sf T} M_{\cX} M_{\cS^{\overline{\ell d}}}^{\Delta \ell} \big|_{(\eta,\nu) = (0,0)} =
\frac{c t_{(\ell+1)d}}{b t_d} \,u_0^{\sf T}
M_{\cS^{\overline{\ell d}}}^{\Delta \ell} \left( I - M_{\cS^{\overline{\ell d}}} \right) \big|_{(\eta,\nu) = (0,0)} \;.
\label{eq:hLhRProof13}
\end{equation}
Since $x^{\cS^{\overline{\ell d}}}_0$ is a fixed point of $f^{\cS^{\overline{\ell d}}}$,
we have $x^{\cS^{\overline{\ell d}}}_0 = \left( I - M_{\cS^{\overline{\ell d}}} \right)^{-1}
P_{\cS^{\overline{\ell d}}} B \mu$,
from which we obtain
\begin{align}
\left( I - M_{\cS^{\overline{\ell d}}} \right) \left( x^{\rm int}_{-d} - x^{\cS^{\overline{\ell d}}}_0 \right)
&= \left( I - M_{\cS^{\overline{\ell d}}} \right) x^{\rm int}_{-d} - P_{\cS^{\overline{\ell d}}} B \mu \nonumber \\
&= x^{\rm int}_{-d} - \left( M_{\cS^{\overline{\ell d}}} x^{\rm int}_{-d} + P_{\cS^{\overline{\ell d}}} B \mu \right) \nonumber \\
&= x^{\rm int}_{-d} - f^{\cS^{\overline{\ell d}}} \!\left( x^{\rm int}_{-d} \right) \nonumber \\
&= -s^{\rm step}_{-d} \zeta_{-d} \;,
\label{eq:hLhRProof14}
\end{align}
using also (\ref{eq:centreDyns}) in the last line.
By combining (\ref{eq:fourtIdentity}), (\ref{eq:gammamd}), (\ref{eq:hLhRProof12}),
(\ref{eq:hLhRProof13}), (\ref{eq:hLhRProof14})
and the formula for $\kappa^+_{\Delta \ell}$ given in Appendix \ref{app:formulas}, we obtain
\begin{equation}
e_1^{\sf T} f^{\cX^{\overline{0}} \left( \cY^{\overline{0}} \cX \right)^{\Delta \ell}} \!\left( x^{\rm int}_{-d} \right) =
\kappa^+_{\Delta \ell} \nu + \cO \!\left( \nu^2 \right),
\label{eq:hLhRProof15}
\end{equation}
under the assumption $\eta = \psi_1(\nu)$.

Now suppose $\nu = \psi_2(\eta)$.
Then $s^{\cS^{\overline{\ell d}}}_0 = 0$ and the $\cS$-cycle is unique with
$x^{\cS}_i = x^{\cS^{\overline{\ell d}}}_{i+d}$, for each $i$.
Thus by (\ref{eq:varphimdgammamd}),
\begin{equation}
x^{\rm int}_{-d} = \left( I - \zeta_{-d} e_1^{\sf T} \right) x^{\cS^{\overline{\ell d}}}_0
= x^{\cS^{\overline{\ell d}}}_0 \;.
\label{eq:hLhRProof16}
\end{equation}
The numerator of (\ref{eq:hLhRProof2}) is then
\begin{align}
e_1^{\sf T} f^{\cX^{\overline{0}} \left( \cY^{\overline{0}} \cX \right)^{\Delta \ell}} \!\left( x^{\rm int}_{-d} \right)
&= e_1^{\sf T} f^{\cX \left( \cY^{\overline{0}} \cX \right)^{\Delta \ell}} \!\left( x^{\cS^{\overline{\ell d}}}_0 \right) \nonumber \\
&= e_1^{\sf T} f^{\cX} \!\left( x^{\cS^{\overline{\ell d}}}_0 \right) \nonumber \\
&= e_1^{\sf T} x^{\cS^{\overline{\ell d}}}_{\ell d} \nonumber \\
&= s^{\cS^{\overline{\ell d}}}_{\ell d} \nonumber \\
&= \frac{t_{(\ell-1)d}}{t_{-d}} \eta + \cO \!\left( \eta^2 \right),
\label{eq:hLhRProof17}
\end{align}
where in the last step we have used (\ref{eq:fourtIdentity}) and equation (58) of \cite{SiMe10}.

By combining (\ref{eq:hLhRProof15}) and (\ref{eq:hLhRProof17}) we obtain
\begin{equation}
e_1^{\sf T} f^{\cX^{\overline{0}} \left( \cY^{\overline{0}} \cX \right)^{\Delta \ell}} \!\left( x^{\rm int}_{-d} \right) =
\frac{t_{(\ell-1)d}}{t_{-d}} \eta + \kappa^+_{\Delta \ell} \nu + \cO \!\left( \left( \eta, \nu \right)^2 \right).
\label{eq:hLhRProof18}
\end{equation}
Then (\ref{eq:hR2}) is given by (\ref{eq:hLhRProof18}) divided by (\ref{eq:denominator}).
\end{proof}

%.....................................................................
\begin{proof}[Proof of Lemma \ref{le:beta}]
For brevity we prove the result only for $\cG^+[k,\Delta \ell]$.
The result for $\cG^+[k,\Delta \ell]^{\overline{0}}$ can be proved in the same fashion.

%^^^^^^^^^^^^^^^^^^^^^^^^^^^^^^^^^^^^^^^^^^^^^^^^^^^^^^^^^^^^^^^^^^^^^^^^^^^^^^^
\myStep{1}
We first use Lemma \ref{le:denominator} to show that
\begin{equation}
e_1^{\sf T} M_{\cX^{\overline{0}} \left( \cY^{\overline{0}} \cX \right)^{\Delta \ell}} \zeta_{-d} < 0 \;,
\label{eq:betaProof1}
\end{equation}
throughout $\Sigma^+_{k,\Delta \ell}$.

At the $\cS$-shrinking point the $\cS^{\overline{0}}$-cycle, denoted $\{ y_i \}$,
is assumed to be admissible with only $y_0$ and $y_{\ell d}$ on the switching manifold.
Therefore $t_{(\ell-1)d} < 0$ and $t_{-d} > 0$.
Since $\rho_{\rm max} < 1$, we have $c > 0$.
%This is can be demonstrated directly from the definitions of $\rho_{\rm max}$ and $c$, (\ref{eq:rhoMax})--(\ref{eq:c})
%(it is also by Lemma 7.4 of \cite{Si17c}).
Also $\left( \kappa^+_{\Delta \ell} - \kappa^+_{\Delta \ell - 1} \right) a > 0$ by assumption.
By applying these inequalities to (\ref{eq:denominator}) we obtain (\ref{eq:betaProof1}).

%^^^^^^^^^^^^^^^^^^^^^^^^^^^^^^^^^^^^^^^^^^^^^^^^^^^^^^^^^^^^^^^^^^^^^^^^^^^^^^^
\myStep{2}
Next we show that $h^L_{\Delta \ell} < h^R_{\Delta \ell}$.

From now on we only consider parameter values in $\Sigma^+_{k,\Delta \ell}$.
Then ${\rm sgn}(\nu) = {\rm sgn}(a)$, see \S\ref{sub:nearby}.
Also $\eta$ and $\nu$ are $\cO \!\left( \frac{1}{k} \right)$.
Thus, by (\ref{eq:gammamd}), $s^{\rm step}_{-d} < 0$ for sufficiently large values of $k$.
%where $s^{\rm step}_{-d}$ is defined by (\ref{eq:varphimdgammamd}).

By (\ref{eq:uv}), (\ref{eq:hL}), and (\ref{eq:centreDyns}),
\begin{equation}
h^L_{\Delta \ell} = s^{\rm step}_{-d} + h^R_{\Delta \ell} \lambda \;.
\label{eq:betaProof10}
\end{equation}
Since $s^{\rm step}_{-d} = \cO \!\left( \frac{1}{k} \right)$,
$\lambda = 1 + \cO \!\left( \frac{1}{k} \right)$,
and $h^L_{\Delta \ell}$ and $h^R_{\Delta \ell}$ are $\cO \!\left( \frac{1}{k} \right)$ by (\ref{eq:hL2})--(\ref{eq:hR2}), we have
\begin{equation}
h^R_{\Delta \ell} - h^L_{\Delta \ell} = -s^{\rm step}_{-d} + \cO \!\left( \frac{1}{k^2} \right).
\label{eq:hRhLDiff}
\end{equation}
Therefore $h^L_{\Delta \ell} < h^R_{\Delta \ell}$
(throughout $\Sigma^+_{k,\Delta \ell}$) for sufficiently large values of $k$.

%^^^^^^^^^^^^^^^^^^^^^^^^^^^^^^^^^^^^^^^^^^^^^^^^^^^^^^^^^^^^^^^^^^^^^^^^^^^^^^^
\myStep{3}
We have thus shown that $\Omega_{\Delta \ell}$ is well-defined by (\ref{eq:Omega}).
Moreover, $\Omega_{\Delta \ell}$ is an $\cO \!\left( \frac{1}{k} \right)$ distance from
the switching manifold, and so is $\Phi$.
In this step we show that $f^{\cG^+[k,\Delta \ell]}(\Phi)$ is also an
$\cO \!\left( \frac{1}{k} \right)$ distance from the switching manifold.

On the outer boundary of $\Sigma^+_{k,\Delta \ell}$,
the matrix $P_{\cG^+[k,\Delta \ell]^{(i)}}$ is singular
for either $i = 0$ or $i = (\ell^+_k + \Delta \ell - 1) d_k^+$.
For brevity we just treat the case $i=0$.
Then by (\ref{eq:tildegLApproxFixedPoint}), on the outer boundary of $\Sigma^+_{k,\Delta \ell}$
the point $f^{\cG^+[k,\Delta \ell]} \!\left( x^{\rm int}_{-d} \right)$ is an $\cO \!\left( \rho_{\rm max}^k \right)$
distance from the switching manifold.
%and it is sufficient to make the weaker statement that
%$f^{\cG^+[k,\Delta \ell]} \!\left( x^{\rm int}_{-d} \right)$ is an $\cO \!\left( \frac{1}{k} \right)$
%distance from the switching manifold on the outer boundary of $\Sigma^+_{k,\Delta \ell}$.

In order to describe $f^{\cG^+[k,\Delta \ell]} \!\left( x^{\rm int}_{-d} \right)$
for parameter values on the inner boundary of $\Sigma^+_{k,\Delta \ell}$,
we first note that on this boundary $P_{\cG^+[k+1,\Delta \ell]}$ is singular.
Thus by (\ref{eq:tildegLApproxFixedPoint})
\begin{equation}
f^{\cG^+[k+1,\Delta \ell]} \!\left( x^{\rm int}_{-d} \right) = x^{\rm int}_{-d} + \cO \!\left( \rho_{\rm max}^k \right).
\label{eq:betaProof5}
\end{equation}
By (\ref{eq:cGplusFormula2}), $f^{\cG^+[k,\Delta \ell]} \!\left( x^{\rm int}_{-d} \right)$ is the unique inverse of 
$f^{\cG^+[k+1,\Delta \ell]} \!\left( x^{\rm int}_{-d} \right)$ under $f^{\cS^{(-d)}}$
that belongs to the range of $f^{\cS^{(-d)}}$.
%By (\ref{eq:centreDyns}) the inverse image of $x^{\rm int}_{-d}$ under $f^{\cS^{(-d)}}$ on $W^c$
%is $x^{\rm int}_{-d} - \frac{s^{\rm step}_{-d}}{\lambda} \zeta_{-d}$.
Thus by (\ref{eq:centreDyns})
\begin{equation}
f^{\cG^+[k,\Delta \ell]} \!\left( x^{\rm int}_{-d} \right) =
x^{\rm int}_{-d} - \frac{s^{\rm step}_{-d}}{\lambda} \zeta_{-d} + \cO \!\left( \rho_{\rm max}^k \right),
\label{eq:betaProof6}
\end{equation}
on the inner boundary of $\Sigma^+_{k,\Delta \ell}$.
Therefore $f^{\cG^+[k,\Delta \ell]} \!\left( x^{\rm int}_{-d} \right)$ is an $\cO \!\left( \frac{1}{k} \right)$
distance from the switching manifold on the inner boundary of $\Sigma^+_{k,\Delta \ell}$.
It follows that $f^{\cG^+[k,\Delta \ell]} \!\left( x^{\rm int}_{-d} \right)$ is an $\cO \!\left( \frac{1}{k} \right)$
distance from the switching manifold throughout $\Sigma^+_{k,\Delta \ell}$.

Now we show that for any $x \in \Phi$ we have
\begin{equation}
f^{\cG^+[k,\Delta \ell]}(x) - f^{\cG^+[k,\Delta \ell]}
\!\left( x^{\rm int}_{-d} \right) = \cO \!\left( \frac{1}{k} \right),
\label{eq:betaProof7}
\end{equation}
which will complete our demonstration that $f^{\cG^+[k,\Delta \ell]}(\Phi)$ is an
$\cO \!\left( \frac{1}{k} \right)$ distance from the switching manifold.
In view of (\ref{eq:cGplusFormula}), since
$x - x^{\rm int}_{-d} = \cO \!\left( \frac{1}{k} \right)$ we have
$f^{\left( \cX \cY^{\overline{0}} \right)^{\Delta \ell} \hat{\cX}}(x) -
f^{\left( \cX \cY^{\overline{0}} \right)^{\Delta \ell} \hat{\cX}}
\!\left( x^{\rm int}_{-d} \right) = \cO \!\left( \frac{1}{k} \right)$
because $f^{\left( \cX \cY^{\overline{0}} \right)^{\Delta \ell} \hat{\cX}}$ is independent of $k$.
Since $\rho_{\rm max} < 1$, all eigenvalues of the matrix part of $f^{\cS^{(-d)}}$ have modulus less than $1$ except possibly $\lambda$.
But $\lambda^k = \cO(1)$ because $\lambda = 1 + \cO(\eta,\nu)$ and $\eta$ and $\nu$ are $\cO \!\left( \frac{1}{k} \right)$.
Therefore $\left( f^{\cS^{(-d)}} \right)^{k-\Delta \ell}$ is not expanding by more than an $\cO(1)$ factor,
which verifies (\ref{eq:betaProof7}) by (\ref{eq:cGplusFormula}).

%^^^^^^^^^^^^^^^^^^^^^^^^^^^^^^^^^^^^^^^^^^^^^^^^^^^^^^^^^^^^^^^^^^^^^^^^^^^^^^^
\myStep{4}
Next we show that $\beta(y)$ is well-defined with $y = f^{\cG^+[k,\Delta \ell]}(x)$ and any $x \in \Phi$.
%Next we show that with $x \in \Phi$ and $y = f^{\cG^+[k,\Delta \ell]}(x)$,
%the left hand-side of (\ref{eq:betaDefn}) is an increasing function of $\Delta k$ and thus
%$\beta(y)$ is well-defined.

For any $y = v(h) \in W^c$, by (\ref{eq:centreDyns}) we have
\begin{align}
e_1^{\sf T} f^{\cX^{\overline{0}} \left( \cY^{\overline{0}} \cX \right)^{\Delta \ell}}
\!\left( f^{\cS^{(-d)}}(y) \right) -
e_1^{\sf T} f^{\cX^{\overline{0}} \left( \cY^{\overline{0}} \cX \right)^{\Delta \ell}}(y) &=
e_1^{\sf T} M_{\cX^{\overline{0}} \left( \cY^{\overline{0}} \cX \right)^{\Delta \ell}}
\left( f^{\cS^{(-d)}}(y) - y \right) \nonumber \\
&=
\left( s^{\rm step}_{-d} + (\lambda-1) h \right)
e_1^{\sf T} M_{\cX^{\overline{0}} \left( \cY^{\overline{0}} \cX \right)^{\Delta \ell}} \zeta_{-d} \;.
\label{eq:betaProof20}
\end{align}
Since $s^{\rm step}_{-d} < 0$, by (\ref{eq:betaProof1}) we have that
(\ref{eq:betaProof20}) is negative for sufficiently small values of $h$.
Moreover, since $s^{\rm step}_{-d}$ and $(\lambda-1)$ are $\cO \!\left( \frac{1}{k} \right)$,
(\ref{eq:betaProof20}) is negative if $h = \cO \!\left( \frac{1}{k} \right)$,
and this is also true for points $y$ sufficiently close to $W^c$.

For any $x \in \Phi$, $f^{\cG^+[k,\Delta \ell]}(x)$
is an $\cO \!\left( \frac{1}{k} \right)$ distance from the switching manifold
and an $\cO \!\left( \rho_{\rm max}^{k} \right)$ distance from $W^c$,
and the same is true for $f^{\cG^+[k + \Delta k,\Delta \ell]}(x)$ with $\Delta k = \cO(1)$.
This shows that (\ref{eq:betaProof20}) is negative using
$y = f^{\cG^+[k + \Delta k,\Delta \ell]}(x)$.
Therefore the left hand-side of (\ref{eq:betaDefn}) is an increasing function of $\Delta k$.
Thus $\beta(y)$ is well-defined.

%^^^^^^^^^^^^^^^^^^^^^^^^^^^^^^^^^^^^^^^^^^^^^^^^^^^^^^^^^^^^^^^^^^^^^^^^^^^^^^^
\myStep{5}
Now we show that if $\Delta k = \beta \!\left( f^{\cG^+[k,\Delta \ell]}(x) \right)$ for some $x \in \Phi$,
then $f^{\cG^+[k + \Delta k,\Delta \ell]}(x) \in \Phi$.
%To do this we find points $x_1 \in H_{k^*,\Delta \ell,Q}$ and
%$x_2 \in f^{\cS^{(-d)}} \!\left( H_{k^*,\Delta \ell,Q} \right)$
%such that $f^{\cG^+[k^*+j,\Delta \ell]}(x) = \alpha x_1 + (1-\alpha) x_2$
%for some $0 \le \alpha < 1$.

Let $\omega_{-d}^\perp$ denote the orthogonal complement of $\omega_{-d}$.
We define a function $\psi : \omega_{-d}^\perp \to \mathbb{R}$ by
\begin{equation}
\psi(q) \defeq h^R_{\Delta \ell} -
\frac{e_1^{\sf T} M_{\cX^{\overline{0}} \left( \cY^{\overline{0}} \cX \right)^{\Delta \ell}} q}
{e_1^{\sf T} M_{\cX^{\overline{0}} \left( \cY^{\overline{0}} \cX \right)^{\Delta \ell}} \zeta_{-d}} \;.
\label{eq:psi}
\end{equation}
Given any $q \in \omega_{-d}^\perp$, let $x = x^{\rm int}_{-d} + \psi(q) \zeta_{-d} + q$.
It is straight-forward to show that\linebreak				% <------ manual linebreak !!!
$e_1^{\sf T} f^{\cX^{\overline{0}} \!\left( \cY^{\overline{0}} \cX \right)^{\Delta \ell}}(x) = 0$.
Thus if $q$ is sufficiently small and $x \in {\rm range} \!\left( f^{\cS^{(-d)}} \right)$,
then $x \in H_{k,\Delta \ell,Q}$.

Choose any $x \in \Phi$ and let $\Delta k = \beta \!\left( f^{\cG^+[k,\Delta \ell]}(x) \right)$.
By the definition of $\beta$ we have 
\begin{align}
e_1^{\sf T} f^{\cX^{\overline{0}} \left( \cY^{\overline{0}} \cX \right)^{\Delta \ell}}
\!\left( f^{\cG^+[k + \Delta k,\Delta \ell]}(x) \right) &> 0 \;, \label{eq:betaProof30} \\
e_1^{\sf T} f^{\cX^{\overline{0}} \left( \cY^{\overline{0}} \cX \right)^{\Delta \ell}}
\!\left( f^{\cG^+[k + \Delta k - 1,\Delta \ell]}(x) \right) &\le 0 \label{eq:betaProof31} \;.
\end{align}
Write $f^{\cG^+[k + \Delta k,\Delta \ell]}(x) = x^{\rm int}_{-d} + h \zeta_{-d} + q$.
Let $h_1 = \psi(q)$ and $x_1 = x^{\rm int}_{-d} + h_1 \zeta_{-d} + q$.
Then $x_1 \in H_{k,\Delta \ell,Q}$, assuming $Q$ is sufficiently large.
Moreover,
\begin{equation}
e_1^{\sf T} f^{\cX^{\overline{0}} \left( \cY^{\overline{0}} \cX \right)^{\Delta \ell}}
\!\left( f^{\cG^+[k + \Delta k,\Delta \ell]}(x) \right) =
(h-h_1) e_1^{\sf T} M_{\cX^{\overline{0}} \left( \cY^{\overline{0}} \cX \right)^{\Delta \ell}} \zeta_{-d} \;,
\nonumber
\end{equation}
thus $h < h_1$ by (\ref{eq:betaProof1}) and (\ref{eq:betaProof30}).

Similarly write $f^{\cG^+[k + \Delta k - 1,\Delta \ell]}(x) = x^{\rm int}_{-d} + \hat{h} \zeta_{-d} + \hat{q}$.
%and note that $h = s^{\rm step}_{-d} + \hat{h} \lambda$.
Let $\hat{h}_1 = \psi(\hat{q})$ and
$\hat{x}_1 = x^{\rm int}_{-d} + \hat{h}_1 \zeta_{-d} + \hat{q}$.
Then $\hat{x}_1 \in H_{k,\Delta \ell,Q}$, assuming $Q$ is sufficiently large.
Let $h_2 = s^{\rm step}_{-d} + \hat{h}_1 \lambda$
and $x_2 = x^{\rm int}_{-d} + h_2 \zeta_{-d} + q$.
Then $x_2 \in f^{\cS^{(-d)}} \!\left( H_{k,\Delta \ell,Q} \right)$ because
$x_2 = f^{\cS^{(-d)}} \!\left( \hat{x}_1 \right)$.
Moreover,
\begin{equation}
e_1^{\sf T} f^{\cX^{\overline{0}} \left( \cY^{\overline{0}} \cX \right)^{\Delta \ell}}
\!\left( f^{\cG^+[k + \Delta k,\Delta \ell]}(x) \right) =
(h-h_2) e_1^{\sf T} M_{\cX^{\overline{0}} \left( \cY^{\overline{0}} \cX \right)^{\Delta \ell}} \zeta_{-d} \;,
\nonumber
\end{equation}
thus $h \ge h_2$ by (\ref{eq:betaProof1}) and (\ref{eq:betaProof31}).

Let $\alpha = \frac{h - h_2}{h_1 - h_2}$.
Then $f^{\cG^+[k + \Delta k,\Delta \ell]}(x) = \alpha x_1 + (1-\alpha) x_2$ and $0 \le \alpha < 1$, hence
$f^{\cG^+[k + \Delta k,\Delta \ell]}(x) \in \Phi$.

%^^^^^^^^^^^^^^^^^^^^^^^^^^^^^^^^^^^^^^^^^^^^^^^^^^^^^^^^^^^^^^^^^^^^^^^^^^^^^^^
\myStep{6}
Choose any $x \in \Phi$ and suppose $f^{\cG^+[k + \Delta k,\Delta \ell]}(x) \in \Phi$
for some $\Delta k = \cO(1)$.
To complete the proof it remains to show that $\Delta k = \beta \!\left( f^{\cG^+[k,\Delta \ell]}(x) \right)$.

%Lastly we show that if $f^{\cG^+[k^*+j,\Delta \ell]}(x) \in \Phi$
%then $j = \beta \!\left( f^{\cG^+[k^*,\Delta \ell]}(x) \right)$.
%Suppose $f^{\cG^+[k,\Delta \ell]}(x) \in \Phi$ for some $k = k^* + \cO(1)$.
Write $f^{\cG^+[k + \Delta k,\Delta \ell]}(x) = \alpha x_1 + (1-\alpha) x_2$
where $x_1 \in H_{k,\Delta \ell,Q}$, 
$x_2 \in f^{\cS^{(-d)}} \!\left( H_{k,\Delta \ell,Q} \right)$
and $0 \le \alpha < 1$.
%There exist unique $\hat{x}_1, \hat{x}_2 \in {\rm range} \!\left( f^{\cS^{(-d)}} \right)$
%such that $x_1 = f^{\cS^{(-d)}} \!\left( \hat{x}_1 \right)$
%and $x_2 = f^{\cS^{(-d)}} \!\left( \hat{x}_2 \right)$.
Since $x_1 = x^{\rm int}_{-d} + h^R_{\Delta \ell} \zeta_{-d} + \cO \!\left( \rho_{\rm max}^k \right)$
and $x_2 = x^{\rm int}_{-d} + h^L_{\Delta \ell} \zeta_{-d} + \cO \!\left( \rho_{\rm max}^k \right)$,
we have $x_1 - x_2 = \left( h^R_{\Delta \ell} - h^L_{\Delta \ell} \right) \zeta_{-d}
+ \cO \!\left( \rho_{\rm max}^k \right)$.
%Similarly $\hat{x}_1 - \hat{x}_2 = \hat{h} \zeta_{-d} + \cO \!\left( \rho_{\rm max}^k \right)$ for some $\hat{h} > 0$.
%In view of (\ref{eq:betaProof1}),
%we can assume that $k^*$ is large enough that
Thus by (\ref{eq:betaProof1}),
\begin{equation}
e_1^{\sf T} M_{\cX^{\overline{0}} \left( \cY^{\overline{0}} \cX \right)^{\Delta \ell}} (x_1-x_2) < 0 \;.
%e_1^{\sf T} M_{\cX^{\overline{0}} \left( \cY^{\overline{0}} \cX \right)^{\Delta \ell}}
%\left( \hat{x}_1 - \hat{x}_2 \right) < 0 \;.
\label{eq:betaProof40}
\end{equation}
By writing
$f^{\cG^+[k + \Delta k,\Delta \ell]}(x) = x_1 - (1-\alpha)(x_1-x_2)$, we obtain
\begin{equation}
e_1^{\sf T} M_{\cX^{\overline{0}} \left( \cY^{\overline{0}} \cX \right)^{\Delta \ell}}
\left( f^{\cG^+[k + \Delta k,\Delta \ell]}(x) \right) =
-(1-\alpha) e_1^{\sf T} M_{\cX^{\overline{0}} \left( \cY^{\overline{0}} \cX \right)^{\Delta \ell}} (x_1-x_2) \;,
\end{equation}
which is positive by (\ref{eq:betaProof40}).
%Also we have
%$f^{\cG^+[k-1,\Delta \ell]}(x) = \alpha \hat{x}_1 + (1-\alpha) \hat{x}_2$
%because $f^{\cS^{(-d)}}$ is affine,
%$f^{\cG^+[k-1,\Delta \ell]}(x) \in {\rm range} \!\left( f^{\cS^{(-d)}} \right)$ and
%$f^{\cG^+[k,\Delta \ell]}(x) = f^{\cS^{(-d)}} \!\left( f^{\cG^+[k-1,\Delta \ell]}(x) \right)$.
%Then by writing $f^{\cG^+[k-1,\Delta \ell]}(x) = \hat{x}_2 + \alpha \left( \hat{x}_1 - \hat{x}_2 \right)$, we obtain
%\begin{equation}
%e_1^{\sf T} M_{\cX^{\overline{0}} \left( \cY^{\overline{0}} \cX \right)^{\Delta \ell}}
%\left( f^{\cG^+[k-1,\Delta \ell]}(x) \right) =
%\alpha e_1^{\sf T} M_{\cX^{\overline{0}} \left( \cY^{\overline{0}} \cX \right)^{\Delta \ell}}
%\left( \hat{x}_1 - \hat{x}_2 \right) \;,
%\end{equation}
%which is negative by (\ref{eq:betaProof40}).
This verifies (\ref{eq:betaProof30}).
Equation (\ref{eq:betaProof31}) can be verified in the same fashion which
shows that $\Delta k = \beta \left( f^{\cG^+[k,\Delta \ell]}(x) \right)$ by the definition of $\beta$.
\end{proof}

%.....................................................................
\begin{proof}[Proof of Lemma \ref{le:tildegSlopes}]
By (\ref{eq:fS2}) and (\ref{eq:uv}),
\begin{equation}
u \!\left( f^{\cG^+[k,\Delta \ell]}(v(h)) \right) =
\omega_{-d}^{\sf T} \Big( M_{\cG^+[k,\Delta \ell]} \left( x^{\rm int}_{-d} + h \zeta_{-d} \right) +
P_{\cG^+[k,\Delta \ell]} B - x^{\rm int}_{-d} \Big),
\nonumber
\end{equation}
thus
\begin{equation}
\frac{d}{d h} \,u \!\left( f^{\cG^+[k,\Delta \ell]}(v(h)) \right) =
\omega_{-d}^{\sf T} M_{\cG^+[k,\Delta \ell]} \zeta_{-d} \;.
\label{eq:tildegSlopesProof1}
\end{equation}
By (\ref{eq:cGplusFormula}) and (\ref{eq:fastDyns}),
\begin{equation}
M_{\cG^+[k,\Delta \ell]} = M_{\cS^{(-d)}}^{k - \Delta \ell} \zeta_{-d} \omega_{-d}^{\sf T}
M_{\hat{\cX}} M_{\cS^{\overline{\ell d}}}^{\Delta \ell} + \cO \!\left( \rho_{\rm max}^k \right).
\label{eq:tildegSlopesProof2}
\end{equation}
Thus using $M_{\cS^{(-d)}} \zeta_{-d} = \lambda \zeta_{-d}$
and $\omega_{-d}^{\sf T} \zeta_{-d} = 1$ we obtain
\begin{equation}
\frac{d}{d h} \,u \!\left( f^{\cG^+[k,\Delta \ell]}(v(h)) \right) =
\lambda^{k - \Delta \ell} \omega_{-d}^{\sf T}
M_{\hat{\cX}} M_{\cS^{\overline{\ell d}}}^{\Delta \ell} \zeta_{-d} + \cO \!\left( \rho_{\rm max}^k \right).
\label{eq:tildegSlopesProof3}
\end{equation}
By Lemma 7.1 of \cite{Si17c},
\begin{equation}
\frac{d}{d h} \,u \!\left( f^{\cG^+[k,\Delta \ell]}(v(h)) \right) =
\lambda^{k - \Delta \ell} u_0^{\sf T}
M_{\cS^{\overline{\ell d}}}^{\Delta \ell} M_{\hat{\cX}} v_0 \big|_{(\eta,\nu) = (0,0)} +
\cO \!\left( \frac{1}{k} \right),
\label{eq:tildegSlopesProof4}
\end{equation}
where $u_0^{\sf T}$ and $v_0$ are the left and right eigenvectors of $M_{\cS}$
with $u_0^{\sf T} v_0 = 1$ and $e_1^{\sf T} v_0 = 1$.

In $\Sigma^+_{k,\Delta \ell}$ we have
$\lambda^{k - \Delta \ell} = \lambda^k + \cO \!\left( \frac{1}{k} \right)
= -\tan(\theta) + \cO \!\left( \frac{1}{k} \right)$ (see equation (A.24) of \cite{Si17c}).
Then by the identity $v_{-d} = \frac{-t_d M_{\hat{\cX}} v_0}{t_{-d}}$
(Lemma 6.6 of \cite{Si17c}) and the formula for $\kappa^+_{\Delta \ell}$ (\ref{eq:kappaPlus}),
equation (\ref{eq:tildegSlopesProof4}) reduces to (\ref{eq:tildegLSlope}).
Equation (\ref{eq:tildegRSlope}) can be derived in a similar fashion\removableFootnote{
Towards the end of the Proof of Lemma \ref{le:tildegSlopes} (where?)
we require the formula $\omega_{-d}^{\sf T} M_{\hat{\cX}^{\overline{0}}} \zeta_{-d} =
\omega_{\ell d}^{\sf T} M_{\check{\cX}} \zeta_{\ell d}$,
which can be derived as follows.

Write
\begin{equation}
\zeta_{-d} = C_1 M_{\hat{\cX}} M_{\cY} \zeta_{\ell d} \;, \qquad
\omega_{-d}^{\sf T} = C_2 \omega_{\ell d}^{\sf T} M_{\check{\cY}} M_{\cX^{\overline{0}}} \;,
\end{equation}
for some $C_1, C_2 \in \mathbb{R}$.
Then
\begin{align}
1 &= \omega_{-d}^{\sf T} \zeta_{-d} \nonumber \\
&= C_1 C_2 \omega_{\ell d}^{\sf T} M_{\check{\cY}} M_{\cX^{\overline{0}}}
M_{\hat{\cX}} M_{\cY} \zeta_{\ell d} \nonumber \\
&= C_1 C_2 \omega_{\ell d}^{\sf T} M_{\check{\cY}}
M_{\check{\cX}} M_{\cX} M_{\cY} \zeta_{\ell d} \nonumber \\
&= \lambda C_1 C_2 \omega_{\ell d}^{\sf T} M_{\check{\cY}} M_{\check{\cX}} \zeta_{\ell d} \nonumber \\
&= \lambda^2 C_1 C_2 \omega_{\ell d}^{\sf T} \zeta_{\ell d} \nonumber \\
&= \lambda^2 C_1 C_2 \;.
\end{align}
Thus
\begin{align}
\omega_{-d}^{\sf T} M_{\hat{\cX}^{\overline{0}}} \zeta_{-d}
&= \frac{1}{\lambda^2} \omega_{\ell d}^{\sf T} M_{\check{\cY}} M_{\cX^{\overline{0}}}
M_{\hat{\cX}^{\overline{0}}} M_{\hat{\cX}} M_{\cY} \zeta_{\ell d} \nonumber \\
&= \frac{1}{\lambda^2} \omega_{\ell d}^{\sf T} M_{\check{\cY}} M_{\check{\cX}}
M_{\cX^{\overline{0}}} M_{\hat{\cX}} M_{\cY} \zeta_{\ell d} \nonumber \\
&= \frac{1}{\lambda} \omega_{\ell d}^{\sf T}
M_{\cX^{\overline{0}}} M_{\hat{\cX}} M_{\cY} \zeta_{\ell d} \nonumber \\
&= \frac{1}{\lambda} \omega_{\ell d}^{\sf T} M_{\check{\cX}} M_{\cX} M_{\cY} \zeta_{\ell d} \nonumber \\
&= \omega_{\ell d}^{\sf T} M_{\check{\cX}} \zeta_{\ell d} \;.
\end{align}
}.
%Equation (\ref{eq:tildegRSlope}) for $\Delta \ell \ge 1$ can be obtained in the same fashion.
%It may be shown that
%$\omega_{-d}^{\sf T} M_{\hat{\cX}^{\overline{0}}} \zeta_{-d} =
%\omega_{\ell d}^{\sf T} M_{\check{\cX}} \zeta_{\ell d}$,
%and using also 
%$M_{\check{\cX}} \zeta_{\ell d} =
%-\frac{t_{(\ell-1)d} v_{(\ell-1)d}}{t_{\ell d}} + \cO \!\left( \frac{1}{k} \right)$,
%we arrive at (\ref{eq:tildegRSlope}) for $\Delta \ell = 0$.
\end{proof}

%.....................................................................
\begin{proof}[Proof of Theorem \ref{th:main}]
For brevity we just prove the result in the case $a < 0$.

%^^^^^^^^^^^^^^^^^^^^^^^^^^^^^^^^^^^^^^^^^^^^^^^^^^^^^^^^^^^^^^^^^^^^^^^^^^^^^^^
\myStep{1}
Define $g_{\rm lift,approx} : [0,1) \to \mathbb{R}$ by
\begin{equation}
g_{\rm lift,approx} \defeq \begin{cases}
\varphi \circ u \circ f^{\cG^+[k,\Delta \ell]} \circ v \circ \varphi^{-1} \;, &
{\rm on~} \left[ 0, \frac{-h^L_{\Delta \ell}}{h^R_{\Delta \ell} - h^L_{\Delta \ell}} \right], \\
\varphi \circ u \circ f^{\cG^+[k,\Delta \ell]^{\overline{0}}} \circ v \circ \varphi^{-1} \;, &
{\rm on~} \left[ \frac{-h^L_{\Delta \ell}}{h^R_{\Delta \ell} - h^L_{\Delta \ell}}, 1 \right).
\end{cases}
\label{eq:gliftApprox}
\end{equation}
We first show that
\begin{equation}
g_{\rm lift,approx} = g_{\rm lift} + \cO \!\left( \frac{1}{k} \right).
\label{eq:gliftApprox2}
\end{equation}

Both $g_{\rm lift,approx}$ and $g_{\rm lift}$ are
real-valued, piecewise-linear continuous functions on $[0,1)$ that are comprised of two pieces.
The slopes of $g_{\rm lift,approx}$ are given by (\ref{eq:tildegLSlope}) and (\ref{eq:tildegRSlope}).
The slopes of $g_{\rm lift}$ are $a_L$ and $a_R$ (\ref{eq:aLaRw2}).
To show that the leading order components of $g_{\rm lift,approx}$ and $g_{\rm lift}$ are the same,
observe that $\theta^+_{\Delta \ell} < \theta^+_{\Delta \ell - 1}$
because $a < 0$ and $\left( \kappa^+_{\Delta \ell} - \kappa^+_{\Delta \ell - 1} \right) a > 0$ by assumption.
Thus $\theta_{\rm min} = \theta^+_{\Delta \ell}$ and $\theta_{\rm max} = \theta^+_{\Delta \ell - 1}$.
Then by (\ref{eq:kappaPlus}), $a_L$ and $a_R$
match (\ref{eq:tildegLSlope}) and (\ref{eq:tildegRSlope}) to leading order.
The kink of $g_{\rm lift,approx}$ is located at $z = \frac{-h^L_{\Delta \ell}}{h^R_{\Delta \ell} - h^L_{\Delta \ell}}$.
Using (\ref{eq:hL2}), (\ref{eq:hR2}), and
$\frac{\nu}{\eta} = \frac{t_{(\ell-1)d}}{t_d} \,\tan(\theta)$ (\ref{eq:polarCoords}),
we see that this matches $z_{\rm sw} = \frac{a_R - 1}{a_R - a_L}$ (the kink of $h_{\rm lift}$) to leading order.
To complete our demonstration of (\ref{eq:gliftApprox2}), we show
that the values of $g_{\rm lift,approx}$ and $g_{\rm lift}$
differ by $\cO \!\left( \frac{1}{k} \right)$ at their kink points.

The value of $g_{\rm lift}$ at its kink point is
\begin{equation}
g_{\rm lift}(z_{\rm sw}) = k^2 \delta + z_{\rm sw} \;.
\label{eq:gliftKinkValue}
\end{equation}
To evaluate $g_{\rm lift,approx}$ at $z = \frac{-h^L_{\Delta \ell}}{h^R_{\Delta \ell} - h^L_{\Delta \ell}}$, notice that
\begin{equation}
v \!\left( \varphi^{-1} \!\left( \frac{-h^L_{\Delta \ell}}{h^R_{\Delta \ell} - h^L_{\Delta \ell}} \right) \right) =
v(0) = x^{\rm int}_{-d} \;.
\label{eq:mainProof10}
\end{equation}
At any point on the outer boundary of $\Sigma^+_{k,\Delta \ell}$, where $\delta = 0$,
we have $\det \!\left( P_{\cG^+[k,\Delta \ell]} \right) = 0$.
%(instead of $\det \!\left( P_{\cG^+[k,\Delta \ell]^{\left( \left(
%\ell^+_k + \Delta \ell - 1 \right) d^+_k \right)}} \right) = 0$).
This is because $a < 0$ and is a straight-forward
consequence of Theorem 2.4 of \cite{Si17c}\removableFootnote{
%.....................................................................
\begin{lemma}
For any well-defined $\Sigma^+_{k,\Delta \ell}$,
if $a < 0$ then $\delta = 0$ corresponds to 
$\det \!\left( P_{\cG^+[k,\Delta \ell]} \right) = 0$,
whereas if $a > 0$ then $\delta = 0$ corresponds to 
$\det \!\left( P_{\cG^+[k,\Delta \ell]^{\left( \left(
\tilde{\ell}-1 \right) d_k^+ \right)}} \right) = 0$.
\label{le:deltaZero}
\end{lemma}

%.....................................................................
\begin{proof}[Proof of Lemma \ref{le:deltaZero}]
Since $\Sigma^+_{k,\Delta \ell}$ is well-defined,
either $\cG^+[k,\Delta \ell]$ or $\cG^+[k,\Delta \ell-1]$-shrinking points exist (or both).
Suppose for simplicity that $\cG^+[k,\Delta \ell]$-shrinking points exist
(the result can be shown similarly with instead $\Delta \ell-1$).

We are working in $(\eta,\nu)$-coordinates about an $\cS$-shrinking point,
Fig.~\ref{fig:shrPointSchem}.
The $\nu$-axis corresponds to $\det \!\left( P_{\cS} \right) = 0$
and the $\eta$-axis corresponds to
$\det \!\left( P_{\cS^{((\ell-1)d)}} \right) = 0$.
Such coordinates can be computed for any shrinking point
which is unfolded by the parameters in a generic fashion.
As in \cite{Si17c}, we denote the analogous
coordinates about a $\cG^+[k,\Delta \ell]$-shrinking point by $\left( \tilde{\eta}, \tilde{\nu} \right)$.
The $\tilde{\nu}$-axis corresponds to
$\det \!\left( P_{\cG^+[k,\Delta \ell]} \right) = 0$,
and the $\tilde{\eta}$-axis corresponds to
$\det \!\left( P_{\cG^+[k,\Delta \ell]^{\left( \left(
\tilde{\ell}-1 \right) d_k^+ \right)}} \right) = 0$.

Theorem 2.4 of \cite{Si17c} tells us that the
$\left( \tilde{\eta}, \tilde{\nu} \right)$-coordinates are well-defined
and have the same orientation as the $(\eta,\nu)$-coordinates.
Moreover, in the proof of the theorem it was shown that both
$\tilde{\eta}$ and $\tilde{\nu}$-axes point in the direction of the first
quadrant of $(\eta,\nu)$-coordinates.

If $a < 0$, then $\theta^+_{\Delta \ell} \in \left( \frac{3 \pi}{2}, 2 \pi \right)$,
that is the $\cG^+[k,\Delta \ell]$-shrinking point lies in the fourth quadrant of
$(\eta,\nu)$-coordinates, as in Fig.~\ref{fig:shrPointSchem}.
Therefore the $\tilde{\nu}$-axis lies closer to $(\eta,\nu) = (0,0)$
than the $\tilde{\eta}$-axis.
Therefore the inner boundary of $\Sigma^+_{k,\Delta \ell}$ corresponds
to the $\tilde{\nu}$-axis and hence to $\det \!\left( P_{\cG^+[k,\Delta \ell]} \right) = 0$.

If instead $a > 0$, then $\theta^+_{\Delta \ell} \in \left( \frac{\pi}{2}, \pi \right)$,
so the $\cG^+[k,\Delta \ell]$-shrinking point lies in the second quadrant of
$(\eta,\nu)$-coordinates.
In this case the $\tilde{\eta}$-axis lies closer to $(\eta,\nu) = (0,0)$
and so the inner boundary of $\Sigma^+_{k,\Delta \ell}$ corresponds to
$\det \!\left( P_{\cG^+[k,\Delta \ell]^{\left( \left(
\tilde{\ell}-1 \right) d_k^+ \right)}} \right) = 0$.
\end{proof}
}.
Thus by Lemma \ref{le:tildegJApproxFixedPoint},
$f^{\cG^+[k,\Delta \ell]} \!\left( x^{\rm int}_{-d} \right) = x^{\rm int}_{-d} +
\cO \!\left( \rho_{\rm max}^k \right)$ when $\delta = 0$.

At any point on the inner boundary of $\Sigma^+_{k,\Delta \ell}$,
where $\delta = \frac{1}{k^2} + \cO \!\left( \frac{1}{k^3} \right)$ (see \S\ref{sub:sectors}),
we similarly have $\det \!\left( P_{\cG^+[k+1,\Delta \ell]} \right) = 0$.
Thus by Lemma \ref{le:tildegJApproxFixedPoint},
$f^{\cG^+[k+1,\Delta \ell]} \!\left( x^{\rm int}_{-d} \right) = x^{\rm int}_{-d} +
\cO \!\left( \rho_{\rm max}^k \right)$.
Since $f^{\cG^+[k+1,\Delta \ell]} = f^{\cS^{(-d)}} \circ f^{\cG^+[k,\Delta \ell]}$,
by (\ref{eq:centreDyns}) we have
$f^{\cG^+[k,\Delta \ell]} \!\left( x^{\rm int}_{-d} \right) = x^{\rm int}_{-d} - s^{\rm step}_{-d} \zeta_{-d} +
\cO \!\left( \rho_{\rm max}^k \right)$,
where here $\delta = \frac{1}{k^2} + \cO \!\left( \frac{1}{k^3} \right)$.

Linear interpolation of these two values of $f^{\cG^+[k,\Delta \ell]} \!\left( x^{\rm int}_{-d} \right)$ gives
\begin{equation}
f^{\cG^+[k,\Delta \ell]} \!\left( x^{\rm int}_{-d} \right) = x^{\rm int}_{-d} - k^2 \delta s^{\rm step}_{-d} \zeta_{-d} +
\cO \!\left( \frac{1}{k^2} \right),
\label{eq:mainProof20}
\end{equation}
where the error term can be justified from the smoothness of $f^{\cG^+[k,\Delta \ell]}$.
Then
\begin{equation}
\varphi \!\left( u \!\left( f^{\cG^+[k,\Delta \ell]} \!\left( x^{\rm int}_{-d} \right) \right) \right) =
\frac{-k^2 \delta s^{\rm step}_{-d} - h^L_{\Delta \ell}}{h^R_{\Delta \ell} - h^L_{\Delta \ell}} +
\cO \!\left( \frac{1}{k} \right),
\label{eq:mainProof21}
\end{equation}
and so by (\ref{eq:hRhLDiff}) and (\ref{eq:mainProof10}) we obtain
\begin{equation}
g_{\rm lift,approx} \!\left( \frac{-h^L_{\Delta \ell}}{h^R_{\Delta \ell} - h^L_{\Delta \ell}} \right) =
k^2 \delta - \frac{h^L_{\Delta \ell}}{h^R_{\Delta \ell} - h^L_{\Delta \ell}} +
\cO \!\left( \frac{1}{k} \right),
\nonumber
\end{equation}
matching the value of $g_{\rm lift}(z_{\rm sw})$ (\ref{eq:gliftKinkValue}) to leading order.

%^^^^^^^^^^^^^^^^^^^^^^^^^^^^^^^^^^^^^^^^^^^^^^^^^^^^^^^^^^^^^^^^^^^^^^^^^^^^^^^
\myStep{2}
Next we derive a formula that is used below to show that $\Delta k = g_{\rm lift}(z) - g(z)$.

By (\ref{eq:centreDyns}),
$u \!\left( f^{\cS^{(-d)}} (v(h)) \right) = s^{\rm step}_{-d} + h \lambda$.
Using (\ref{eq:hRhLDiff}), since $\lambda = 1 + \cO \!\left( \frac{1}{k} \right)$,
for any $\cO \!\left( \frac{1}{k} \right)$ value of $h$ we can write
\begin{equation}
u \!\left( f^{\cS^{(-d)}}(v(h)) \right) =
h^L_{\Delta \ell} - h^R_{\Delta \ell} + h + \cO \!\left( \frac{1}{k^2} \right).
\nonumber
\end{equation}
Then
\begin{equation}
\varphi \!\left( u \!\left( f^{\cS^{(-d)}}(v(h)) \right) \right) = \varphi(h) - 1 +
\cO \!\left( \frac{1}{k} \right).
\nonumber
\end{equation}
Repeating this result yields the desired formula
\begin{equation}
\varphi(u(T(v(h)))) = \varphi(h) - \Delta k + \cO \!\left( \frac{1}{k} \right),
\label{eq:mainProof30}
\end{equation}
where $T = \left( f^{\cS^{(-d)}} \right)^{\Delta k}$.

%^^^^^^^^^^^^^^^^^^^^^^^^^^^^^^^^^^^^^^^^^^^^^^^^^^^^^^^^^^^^^^^^^^^^^^^^^^^^^^^
\myStep{3}
Define
\begin{equation}
g_{\rm approx} \defeq \varphi \circ G \circ \varphi^{-1} \;,
\label{eq:gApprox}
\end{equation}
where $G$ given by (\ref{eq:tildeg}).
Here we show that
\begin{equation}
g_{\rm approx}(z) = g(z) + \cO \!\left( \frac{1}{k} \right),
\label{eq:gApprox2}
\end{equation}
for all $z \in [0,1) \setminus \Xi_1$, for some set $\Xi_1 \subset [0,1)$
with ${\rm meas}(\Xi_1) = \cO \!\left( \rho_{\rm max}^k \right)$.

By (\ref{eq:tildef2}) and (\ref{eq:tildeg}),
\begin{equation}
g_{\rm approx} = \begin{cases}
\varphi \circ u \circ T \circ f^{\cG^+[k,\Delta \ell]} \circ v \circ \varphi^{-1} \;, &
{\rm on~} \left[ 0, \frac{-h^L_{\Delta \ell}}{h^R_{\Delta \ell} - h^L_{\Delta \ell}} \right], \\
\varphi \circ u \circ T \circ f^{\cG^+[k,\Delta \ell]^{\overline{0}}} \circ v \circ \varphi^{-1} \;, &
{\rm on~} \left[ \frac{-h^L_{\Delta \ell}}{h^R_{\Delta \ell} - h^L_{\Delta \ell}}, 1 \right).
\end{cases}
\label{eq:gApprox3}
\end{equation}
Since $v(u(x)) = x + \cO \!\left( \rho_{\rm max}^k \right)$ for any $x$ located
an $\cO \!\left( \rho_{\rm max}^k \right)$ distance from $W^c$,
we can insert $v \circ u$ into (\ref{eq:gApprox3}) to produce
\begin{equation}
g_{\rm approx} = \begin{cases}
\varphi \circ u \circ T \circ v \circ u \circ f^{\cG^+[k,\Delta \ell]} \circ v \circ \varphi^{-1} +
\cO \!\left( \rho_{\rm max}^k \right), &
{\rm on~} \left[ 0, \frac{-h^L_{\Delta \ell}}{h^R_{\Delta \ell} - h^L_{\Delta \ell}} \right], \\
\varphi \circ u \circ T \circ v \circ u \circ f^{\cG^+[k,\Delta \ell]^{\overline{0}}} \circ v \circ \varphi^{-1} +
\cO \!\left( \rho_{\rm max}^k \right), &
{\rm on~} \left[ \frac{-h^L_{\Delta \ell}}{h^R_{\Delta \ell} - h^L_{\Delta \ell}}, 1 \right).
\end{cases}
\label{eq:gApprox4}
\end{equation}
Then by (\ref{eq:mainProof30}),
$g_{\rm approx} = g_{\rm lift,approx} - \Delta k + \cO \!\left( \frac{1}{k} \right)$, 
and so by (\ref{eq:gliftApprox2}),
\begin{equation}
g_{\rm approx} = g_{\rm lift} - \Delta k + \cO \!\left( \frac{1}{k} \right).
\label{eq:mainProof40}
\end{equation}
But $F \!\left( v \!\left( \varphi^{-1}(z) \right) \right) \in \Phi$ for all $z \in [0,1)$.
Thus $u \!\left( F \!\left( v \!\left( \varphi^{-1}(z) \right) \right) \right) \in
\left[ h^L_{\Delta \ell}, h^R_{\Delta \ell} \right)$ for all $z \in [0,1) \setminus \Xi_1$,
where $\Xi_1 \subset [0,1)$ contains points near the left and right faces of $\Phi$ and
${\rm meas}(\Xi_1) = \cO \!\left( \rho_{\rm max}^k \right)$.
Thus $g_{\rm approx}(z) \in [0,1)$ for all $z \in [0,1) \setminus \Xi_1$.
Therefore
$g_{\rm approx}(z) = g_{\rm lift}(z) {\rm \,mod\,} 1 + \cO \!\left( \frac{1}{k} \right)$
for all $z \in [0,1) \setminus \Xi_1$, which verifies (\ref{eq:gApprox2}).

%^^^^^^^^^^^^^^^^^^^^^^^^^^^^^^^^^^^^^^^^^^^^^^^^^^^^^^^^^^^^^^^^^^^^^^^^^^^^^^^
\myStep{4}
For any $x \in \Phi$, we have $x - v(u(x)) = \cO \!\left( \rho_{\rm max}^k \right)$.
It is straight-forward to show that similarly 
$F(x) - F(v(u(x))) = \cO \!\left( \rho_{\rm max}^k \right)$
because iterates approach $W^c$ under $f^{\cS^{(-d)}}$ and $\lambda^k = \cO(1)$.
Since $F(v(u(x))) \in \Phi$, we also have
$F(v(u(x))) - v(G(u(x))) = \cO \!\left( \rho_{\rm max}^k \right)$.
Thus $F(x) - v(G(u(x))) = \cO \!\left( \rho_{\rm max}^k \right)$.
Thus by (\ref{eq:gApprox}) and (\ref{eq:gApprox2}),
\begin{equation}
F = v \circ \varphi^{-1} \circ g \circ \varphi \circ u + \cO \!\left( \frac{1}{k^2} \right),
\label{eq:mainProof60}
\end{equation}
for all $x \in \Phi$ with $z \in [0,1) \setminus \Xi_1$ where $z = \varphi(u(x))$.

By (\ref{eq:gApprox2}) and (\ref{eq:mainProof40}),
$\Delta k = g_{\rm lift}(z) - g(z)$ for all but a measure $\cO \!\left( \frac{1}{k} \right)$ subset of $[0,1)$
due to the error terms in these expressions.
In view of (\ref{eq:tildef3}),
equation (\ref{eq:F2}) holds unless $z = \varphi(u(x))$ belongs to an $\cO \!\left( \frac{1}{k} \right)$
subset of $[0,1)$ where (\ref{eq:F2}) is not satisfied due to the $\cO \!\left( \frac{1}{k} \right)$ difference
in the kink points $\frac{-h^L_{\Delta \ell}}{h^R_{\Delta \ell} - h^L_{\Delta \ell}}$ and $z_{\rm sw}$.
\end{proof}

%=====================================================================
\section{Additional formulas}
\label{app:formulas}
\setcounter{equation}{0}

%Here we provide explicit expressions for $\theta^\pm_{\Delta \ell}$: the leading order components
%of the angular coordinates of $\cG^\pm[k,\Delta \ell]$-shrinking points.
%These expressions are derived in \cite{Si17c}.

Suppose (\ref{eq:f}) is at an $\cS$-shrinking point where $\cS = \cF[\ell,m,n]$.
For each $j = 0$, $(\ell-1)d$, $\ell d$ and $-d$ (taken modulo $n$),
let $u_j^{\sf T}$ and $v_j$ be the left and right eigenvectors of $M_{\cS^{(j)}}$
corresponding to the eigenvalue $1$ and normalised by $u_j^{\sf T} v_j = 1$ and $e_1^{\sf T} v_j = 1$.
Then for all $\Delta \ell \in \mathbb{Z}$, let
\begin{align}
\kappa^+_{\Delta \ell} &\defeq \begin{cases}
u_{\ell d}^{\sf T} M_{\cS^{\overline{0} (\ell d)}}^{-\Delta \ell-1} v_{(\ell-1)d} \;, &
\Delta \ell \le -1 \;, \\
u_0^{\sf T} M_{\cS^{\overline{\ell d}}}^{\Delta \ell} v_{-d} \;, &
\Delta \ell \ge 0 \;,
\end{cases} \label{eq:kappaPlus} \\
\kappa^-_{\Delta \ell} &\defeq \begin{cases}
u_{-d}^{\sf T} M_{\cS^{\overline{0}}}^{-\Delta \ell} v_0 \;, &
\Delta \ell \le 0 \;, \\
u_{(\ell-1)d}^{\sf T} M_{\cS^{\overline{\ell d} (\ell d)}}^{\Delta \ell - 1} v_{\ell d} \;, &
\Delta \ell \ge 1 \;,
\end{cases} \label{eq:kappaMinus}
\end{align}
and assuming $\kappa^{\pm}_{\Delta \ell} \ne 0$, let
\begin{align}
\theta^+_{\Delta \ell} &\defeq \begin{cases}
\tan^{-1} \!\left( \frac{t_{(\ell+1)d}}
{t_{(\ell-1)d} \left| \kappa^+_{\Delta \ell} \right|} \right), &
\Delta \ell \le -1 \;, \\
\tan^{-1} \!\left( \frac{t_d}
{t_{-d} \left| \kappa^+_{\Delta \ell} \right|} \right), &
\Delta \ell \ge 0 \;,
\end{cases} \label{eq:thetaPlus} \\
\theta^-_{\Delta \ell} &\defeq \begin{cases}
\tan^{-1} \!\left( \frac{t_d \left| \kappa^-_{\Delta \ell} \right|}
{t_{-d}} \right), &
\Delta \ell \le 0 \;, \\
\tan^{-1} \!\left( \frac{t_{(\ell+1)d} \left| \kappa^-_{\Delta \ell} \right|}
{t_{(\ell-1)d}} \right), &
\Delta \ell \ge 1 \;,
\end{cases} \label{eq:thetaMinus}
\end{align}
where $\theta^+_{\Delta \ell} \in \left( \frac{3 \pi}{2}, 2 \pi \right)$
and $\theta^-_{\Delta \ell} \in \left( \frac{\pi}{2}, \pi \right)$ if $a < 0$,
and $\theta^+_{\Delta \ell} \in \left( \frac{\pi}{2}, \pi \right)$
and $\theta^-_{\Delta \ell} \in \left( \frac{3 \pi}{2}, 2 \pi \right)$ if $a > 0$.

\section*{Acknowledgements}
The author thanks James Meiss and Chris Tuffley for invaluable discussions regarding homotopies.

%{\footnotesize
%\bibliographystyle{unsrt}
%\bibliography{../../DynSyst,../../MathBio,../../Misc,../../OtherTheory,../../PWS,../../Stoch}
%}

\end{document}